\documentclass{amsart}
\usepackage{graphicx}
\usepackage{tikz}
\usepackage{xcolor}
\usepackage{comment}
\usepackage[T1]{fontenc}
\usepackage{xypic}

\usepackage{amsmath,amsthm,amssymb,amscd,enumerate}

\newtheorem{theorem}{Theorem}[section]
\newtheorem{proposition}[theorem]{Proposition}
\newtheorem{lemma}[theorem]{Lemma}
\newtheorem{corollary}[theorem]{Corollary}

\theoremstyle{definition}
\newtheorem{definition}[theorem]{Definition}

\theoremstyle{remark}

\newtheorem{remark}[theorem]{Remark}

\numberwithin{equation}{section}

\allowdisplaybreaks[1]

\newcommand{\QQ}{\mathbb{Q}}
\newcommand{\RR}{\mathbb{R}}
\newcommand{\ZZ}{\mathbb{Z}}
\newcommand{\NN}{\mathbb{N}}

\newcommand{\QQQ}{\mathrm{Q}}
\newcommand{\hQQQ}{\widehat{\mathrm{Q}}}
\newcommand{\rQQQ}{\widehat{\mathrm{Q}}}
\newcommand{\HHH}{\mathrm{H}}
\newcommand{\CCC}{\mathrm{C}}
\newcommand{\ZZZ}{\mathrm{Z}}
\newcommand{\BBB}{\mathrm{B}}

\DeclareMathOperator{\cl}{\mathrm{cl}}
\DeclareMathOperator{\scl}{\mathrm{scl}}
\DeclareMathOperator{\fl}{\mathrm{fill}}
\DeclareMathOperator{\Ker}{\mathrm{Ker}}
\DeclareMathOperator{\Coker}{\mathrm{Coker}}

\newcommand{\Gg}{G}
\newcommand{\Ng}{N}

\newcommand{\Gam}{\Gamma}

\newcommand{\CG}{[\Gg,\Gg]}
\newcommand{\CGN}{[\Gg,\Ng]}

\newcommand{\clG}{\cl_{\Gg}}
\newcommand{\sclG}{\scl_{\Gg}}
\newcommand{\clGN}{\cl_{\Gg,\Ng}}
\newcommand{\sclGN}{\scl_{\Gg,\Ng}}

\newcommand{\xl}{x}
\newcommand{\yl}{y}
\newcommand{\zl}{z}
\newcommand{\ul}{u}
\newcommand{\vl}{v}

\newcommand{\gl}{g}
\newcommand{\hl}{h}
\newcommand{\gbr}{\check{g}}
\newcommand{\xbr}{\check{x}}
\newcommand{\ff}{f}
\newcommand{\hf}{h}

\newcommand{\nuf}{\nu}
\newcommand{\muf}{\mu}
\newcommand{\psf}{\psi}
\newcommand{\phf}{\phi}
\newcommand{\phfh}{\hat{\phi}}
\newcommand{\psfh}{\hat{\psi}}
\newcommand{\mufh}{\hat{\mu}}

\newcommand{\QG}{\QQQ(\Gg)}
\newcommand{\QhG}{\widehat{\QQQ}(\Gg)}

\newcommand{\QhN}{\widehat{\QQQ}(\Ng)}
\newcommand{\QNG}{\QQQ(\Ng)^{\Gg}}

\newcommand{\HG}{\HHH^1(\Gg)}
\newcommand{\HGR}{\HHH^1(\Gg;\RR)}
\newcommand{\HNG}{\HHH^1(\Ng)^{\Gg}}
\newcommand{\HNRG}{\HHH^1(\Ng;\RR)^{\Gg}}

\newcommand{\QGN}{\hQQQ_{\Ng}(\Gg)}

\newcommand{\WW}{\mathrm{W}}
\newcommand{\WGN}{\WW(\Gg,\Ng)}

\newcommand{\DD}{D}
\newcommand{\DDG}{D_{\Gg}}
\newcommand{\DDN}{D_{\Ng}}

\newcommand{\genus}{g}

\renewcommand{\Im}{\mathrm{Im}}

\newcommand{\Homeo}{\mathrm{Homeo}}
\newcommand{\tHomeo}{\mathrm{H}\widetilde{\mathrm{omeo}}}

\newcommand{\EH}{\mathrm{EH}}

\newcommand{\SL}{\mathrm{SL}}

\newcommand{\rot}{\mathrm{rot}}
\newcommand{\trot}{\widetilde{\rot}}

\newcommand{\qm}{\mu}

\newcommand{\siph}{\varphi}
\newcommand{\coch}{u}

\newcommand{\CR}{\mathcal{C}_{\mathbb{R}}}
\newcommand{\CQ}{\mathcal{C}_{\mathbb{Q}}}
\newcommand{\CZ}{\mathcal{C}_{\mathbb{Z}}}
\newcommand{\CA}{\mathcal{C}_A}

\title{Invariant quasimorphisms and generalized mixed Bavard duality}

\author[M. Kawasaki]{Morimichi Kawasaki}
\address[Morimichi Kawasaki]{Department of Mathematics, Faculty of Science, Hokkaido University, North 10, West 8, Kita-ku, Sapporo, Hokkaido 060-0810, Japan.}
\email{kawasaki@math.sci.hokudai.ac.jp}

\author[M. Kimura]{Mitsuaki Kimura}
\address[Mitsuaki Kimura]{Department of Mathematics, Osaka Dental University
 8-1 Kuzuha-hanazono-cho, Hirakata, Osaka 573-1121 Japan}
\email{kimura-m@cc.osaka-dent.ac.jp}

\author[S. Maruyama]{Shuhei Maruyama}
\address[Shuhei Maruyama]{ School of Mathematics and Physics, College of Science and Engineering, Kanazawa University, Kakuma-machi, Kanazawa, Ishikawa, 920-1192, Japan}
\email{smaruyama@se.kanazawa-u.ac.jp}

\author[T. Matsushita]{Takahiro Matsushita}
\address[Takahiro Matsushita]{Department of Mathematical Sciences, Shinshu University, Matsumoto, Nagano 390-8621, Japan}
\email{matsushita@shinshu-u.ac.jp}

\author[M. Mimura]{Masato Mimura}
\address[Masato Mimura]{Mathematical Institute, Tohoku University, 6-3, Aramaki Aza-Aoba, Aoba-ku, Sendai 980-8578, Japan}
\email{m.masato.mimura.m@tohoku.ac.jp}

\begin{document}

\begin{abstract}
This article provides an  expository account of the celebrated duality theorem of Bavard and three its strengthenings. The Bavard duality theorem connects scl (stable commutator length) and quasimorphisms on a group. Calegari extended the framework from a group element to a chain on the group, and established the generalized Bavard duality. Kawasaki, Kimura, Matsushita and Mimura studied the setting of a pair of a group and its normal subgroup, and obtained the mixed Bavard duality. The first half of the present article is devoted to an introduction to these three Bavard dualities. In the latter half, we present a new strengthening, the generalized mixed Bavard duality, and provide a self-contained proof of it. This third strengthening recovers all of the Bavard dualities treated in the first half; thus, we supply complete proofs of these four Bavard dualities in a unified manner. In addition, we state several results on the space $\mathrm{W}(G,N)$ of non-extendable quasimorphisms, which is related to the comparison problem between scl and mixed scl via the mixed Bavard duality.
\end{abstract}

\maketitle

\section{Introduction}
\label{section=Intro}
The main subject of this article is the Bavard duality theorem (Theorem~\ref{theorem=Bavard}) and  three its strengthenings (Theorem~\ref{theorem=generalized_Bavard}, Theorem~\ref{theorem=mixed_Bavard} and Theorem~\ref{theorem=generalized_mixed_Bavard}). We visualize the deductions between these four theorems in Figure~\ref{fig=visual}. The first half of this article provides an  expository account of the three Bavard dualities: the original one (Theorem~\ref{theorem=Bavard}), the generalized one (Theorem~\ref{theorem=generalized_Bavard}) and the mixed one (Theorem~\ref{theorem=mixed_Bavard}). The latter half is devoted to the self-contained proof of the generalized mixed Bavard duality theorem (Theorem~\ref{theorem=generalized_mixed_Bavard}), which is a new result presented in this article. Since Theorem~\ref{theorem=generalized_mixed_Bavard} recovers the other three Bavard dualities, we provide complete proofs of the four Bavard dualities in a unified manner.

\begin{figure}[h]
  \centering
    \begin{tikzpicture}[auto]
    \node[draw, align=center] (01) at (0, 2) {\underline{original (Theorem~\ref{theorem=Bavard}, \cite{Bavard}):} \\ for $\hl \in \CG$};
    \node[draw, align=center] (11) at (8.5, 2) {\underline{generalized  (Theorem~\ref{theorem=generalized_Bavard}, \cite{Calegari}):} \\ for $c \in \BBB_1(\Gg;\mathbb{R})$};
    \node[draw, align=center] (00) at (0, 0) {\underline{mixed  (Theorem~\ref{theorem=mixed_Bavard}, \cite{KKMMM1}):} \\ for $\yl \in \CGN$};
    \node[draw, align=center] (10) at (8.5, 0) {\underline{generalized mixed (Theorem~\ref{theorem=generalized_mixed_Bavard}):} \\ for $c \in \CR(\Gg,\Ng)$ };
    \draw[<-] (01) to node {$\Ng=\Gg$} (00);
    \draw[<-] (01) to node {$c=\hl$ (Lemma~\ref{lem=domainG})} (11);
    \draw[<-] (11) to node {$\Ng=\Gg$  (Remark~\ref{rem=BBB_C})} (10);
    \draw[<-] (00) to node {$c=\yl$ (Lemma~\ref{lemma=domain})} (10);
    \end{tikzpicture}
    \caption{the Bavard duality and three its strengthenings}  \label{fig=visual}
\end{figure}

First, we exhibit the statement of the original Bavard duality. For a group $\Gg$, the \emph{stable commutator length} $\scl$ of an element $\hl$ in the commutator subgroup $[\Gg,\Gg]$ is defined as the limit of the ratio of the commutator length ($\cl$) of $\hl^n$ over $n$ as $n\to \infty$ (Definition~\ref{defn=scl}). We refer the reader to Calegari's book \cite{Calegari} for various applications of $\cl$ and $\scl$ in geometric topology and geometric group theory. Surprisingly, computations of $\scl$ may be considerably easier than those of $\cl$; for instance, Calegari \cite{Calegari1} provides an algorithm to compute $\scl_{F}(\hl)$ for every $\hl\in [F,F]$ where $F$ is a free group. This is mainly attributed to the celebrated \emph{duality theorem of Bavard} \cite{Bavard}, which describes the value of $\scl$ in terms of homogeneous quasimorphisms on the group. Here, a real-valued function on a group is called a \emph{quasimorphism} if it satisfies the equality of being a group homomorphism up to uniformly bounded additive error; see Definition~\ref{defn=qm} for the definition of homogeneous quasimorphisms. From this point of view, the space $\HGR$ can be regarded as the space of genuine homomorphisms from $\Gg$ to $\RR$. By setting $\QG$ as the $\RR$-linear space of homogeneous quasimorphisms on a group $\Gg$, we can state the Bavard duality theorem as follows.

\begin{theorem}[Bavard duality theorem, \cite{Bavard}]\label{theorem=Bavard}
Let $\Gg$  be a group. Then, for every $\hl\in \CG$, we have
\[
\sclG(\hl)=\sup_{\phf\in \QG\setminus \HGR}\frac{|\phf(\hl)|}{2\DD(\phf)}.
\]
\end{theorem}
Here, if $\QG\setminus \HGR=\emptyset$, then we regard the right-hand side of the equality above as $0$; we adopt conventions similar to this in the three strengthenings (Theorem~\ref{theorem=generalized_Bavard}, Theorem~\ref{theorem=mixed_Bavard} and Theorem~\ref{theorem=generalized_mixed_Bavard}) of Theorem~\ref{theorem=Bavard}.

In \cite{Calegari1, Calegari2}, Calegari introduced and developed theory of $\scl$ of chains on a group. Indeed, this plays a key role to the aforementioned algorithm in \cite{Calegari1}. The strengthening of the Bavard duality theorem to chains is called the \emph{generalized Bavard duality theorem}, whose statement goes follows. In Section~\ref{sec=gBavard}, we will present the definitions of the notions appearing in Theorem~\ref{theorem=generalized_Bavard} for the case where $c$ is an integral chain (more precisely, if $c\in \BBB_1(\Gg;\ZZ)$), and there we will state the duality theorem as Theorem~\ref{thm=gBavardZ}.
\begin{theorem}[generalized Bavard duality theorem, \cite{Calegari}]\label{theorem=generalized_Bavard}
Let $\Gg$  be a group. Then, for every chain $c\in \BBB_1(\Gg;\RR)$, we have
\[
\sclG(c)=\sup_{\phf\in \QG\setminus \HGR}\frac{|\phf(c)|}{2\DD(\phf)}.
\]
\end{theorem}
This is the first strengthening of the Bavard duality in the present article: when the chain $c$ is of the form $c=\hl$ with $\hl\in \CG$, Theorem~\ref{theorem=generalized_Bavard} recovers Theorem~\ref{theorem=Bavard}.

In what follows, we discuss the second strengthening of the Bavard duality in the present article. In \cite{KKMM1}, Kawasaki, Kimura, Matsushita and Mimura proved the \emph{mixed Bavard duality theorem} (previously obtained by \cite{KK} with an extra assumption that $\CGN=\Ng$); we state it here as Theorem~\ref{theorem=mixed_Bavard}. Here, `mixed' means that we treat a pair $(\Gg,\Ng)$ of a group $\Gg$ and its normal subgroup $\Ng$. In this setting, the notion of the \emph{stable mixed commutator length} $\scl_{\Gg,\Ng}$ (mixed $\scl$) on the mixed commutator subgroup $[\Gg,\Ng]$ is defined (Definition~\ref{defn=mixedscl}). The counterpart of $\QG$ in the mixed Bavard duality is $\QNG$, which is the $\RR$-linear space of \emph{$\Gg$-invariant} homogeneous quasimorphisms on $\Ng$. Here, we discuss the $\Gg$-invariance under the $\Gg$-action on functions on $\Ng$ via the $\Gg$-action by conjugation; see Definition~\ref{defn=invqm} for more details.
\begin{theorem}[mixed Bavard duality theorem, \cite{KKMM1}]\label{theorem=mixed_Bavard}
Let $\Gg$ be a group and $\Ng$ its normal subgroup. Then, for every $\yl\in \CGN$, we have
\[
\sclGN(\yl)=\sup_{\muf\in \QNG\setminus \HNRG}\frac{|\muf(\yl)|}{2\DD(\muf)}.
\]
\end{theorem}
This is the second strengthening of the Bavard duality in this article. Indeed, the `mixed' setting for $\Ng=\Gg$ is exactly the setting for a single group $\Gg$. In particular, when $\Ng=\Gg$, Theorem~\ref{theorem=mixed_Bavard} recovers Theorem~\ref{theorem=Bavard}.

The third (and last) strengthening of the Bavard duality in this article is the \emph{generalized mixed Bavard duality theorem}; as is mentioned at the beginning of this introduction, this is a new result. In Section~\ref{sec=gmBavard}, we will present the definitions of the notions for chains appearing in Theorem~\ref{theorem=generalized_mixed_Bavard}.
\begin{theorem}[generalized mixed Bavard duality theorem]\label{theorem=generalized_mixed_Bavard}
Let $\Gg$ be a group and $\Ng$ its normal subgroup. For every chain $c \in \CR(\Gg,\Ng)$, we have
\[
\sclGN(c)=\sup_{\muf\in \QNG\setminus \HNRG}\frac{|\muf(c)|}{2\DD(\muf)}.
\]
\end{theorem}
When $\Ng=\Gg$, Theorem~\ref{theorem=generalized_mixed_Bavard} recovers Theorem~\ref{theorem=generalized_Bavard} (by Remark~\ref{rem=BBB_C} in Section~\ref{sec=gmBavard}); when the chain $c$ is of the form $c=\yl$ with $\yl\in \CGN$, Theorem~\ref{theorem=generalized_mixed_Bavard} recovers Theorem~\ref{theorem=mixed_Bavard} (by Lemma~\ref{lemma=domain}).

In the remaining part of this article, we study the $\RR$-linear space $\WGN$ defined as \eqref{eq=W} below. Let $\Gg$ be a group and $\Ng$ its normal subgroup. By Theorem~\ref{theorem=Bavard},  $\scl_{\Gg}$ may be described in terms of the quotient vector space $\QG/\HGR$ (Remark~\ref{rem=quotient}). By Theorem~\ref{theorem=mixed_Bavard}, for $\scl_{\Gg,\Ng}$ the counterpart of this space is $\QNG/\HNRG$ (Remark~\ref{rem=quotientGN}). The inclusion map $i\colon \Ng \hookrightarrow \Gg$ induces an $\RR$-linear map $\overline{i^{\ast}}\colon \QG/\HGR\to \QNG/\HNRG$. Then, $\WGN$ is defined as
\begin{equation}\label{eq=W}
\WGN=\Coker \Big(\overline{i^{\ast}}\colon \QG/\HGR\to \QNG/\HNRG \Big).
\end{equation}
The map $i$ induces an $\RR$-linear map $i^{\ast}\colon \QG\to \QNG$; $\phf\mapsto \phf|_{\Ng}$. In terms of $i^{\ast}$, \eqref{eq=W} is rewritten as
\[
\WGN=\QNG/ \big(\HNRG +i^{\ast}\QG \big);
\]
this space was introduced in \cite{KKMMM1} without symbol, and the symbol $\WGN$ was given in \cite{KKMMMsurvey}.
The space $\WGN$ is related to the comparison problem between $\scl_{\Gg}$ and $\scl_{\Gg,\Ng}$ on $[\Gg,\Ng]$ via the mixed Bavard duality (Theorem~\ref{theorem=mixed_Bavard}). For instance, in \cite{KKMMM1} the authors of the present article proved that if $\WGN=0$ and if $\Gam=\Gg/\Ng$ is amenable, then
\[
\scl_{\Gg,\Ng}(\yl)\leq 2\scl_{\Gg}(\yl)
\]
holds for every $\yl\in \CGN$ (we note that $\scl_{\Gg,\Ng}(\yl)\geq \scl_{\Gg}(\yl)$ follows from the definition; hence, $\scl_{\Gg}$ and $\scl_{\Gg,\Ng}$ are bi-Lipschitzly equivalent on $[\Gg,\Ng]$ in this setting). We will see this in Subsection~\ref{subsec=Wscl} (Theorem \ref{thm=comparison_amenable}). We refer the reader to \cite[Section~8]{KKMMMsurvey} and \cite{KKMMMcg} for further directions on the comparison problem.

\

\noindent
\textbf{Organization of the present article}
In Section~\ref{sec=qmBavard}, we present the definitions of quasimorphisms and scl; there, we prove the easy direction (Corollary~\ref{cor=sclabove}) of the Bavard duality theorem. In Section~\ref{sec=gBavard}, we treat the generalized Bavard duality theorem, restricted to integral chains. In Section~\ref{sec=invqm}, we discuss basic properties of invariant quasimorphisms and mixed scl, and state the mixed Bavard duality theorem. In Section~\ref{sec=gmBavard}, we provide a self-contained proof of the generalized mixed Bavard duality theorem (Theorem~\ref{theorem=generalized_mixed_Bavard}), which is a new result. As Figure~\ref{fig=visual} illustrates, this proof for certain special cases may be regarded as  proofs of  Theorem~\ref{theorem=Bavard}, Theorem~\ref{theorem=generalized_Bavard} and Theorem~\ref{theorem=mixed_Bavard}, respectively. In Section~\ref{sec=further}, we study further properties of mixed scl for chains. In Section~\ref{sec=W}, we first discuss some relation between the space $\WGN$ and Bavard dualities (Theorem~\ref{thm=comparison_amenable} and Proposition~\ref{prop=bilip_criterion}); then, we present basic properties of this space, such as dimension computations and natural isomorphisms. In Section~\ref{sec=exactW}, we obtain exact sequences related to $\WGN$.

\

\noindent
\textbf{Notation and conventions}

For a group $\Gg$, the group unit of $\Gg$ is written as $1_{\Gg}$. Our convention of group commutators is: $[a,b]=aba^{-1}b^{-1}$.
For $s,t\in \RR$ and for $D\in \mathbb{R}_{\geq 0}$, we write $s\sim_D t$ if $|s-t|\leq D$. Let $\mathbb{N}=\{1,2,3,\ldots\}$ be the set of positive integers.  In this article, unless otherwise stated, a surface means a compact orientable 2-dimensional manifold; we do not assume that a surface is connected. This convention for surfaces is convenient specially for our discussions in Section~\ref{sec=gmBavard} and Section~\ref{sec=further}.  We equip a group with the discrete topology (when we discuss group cohomology and other concepts).

\section{Quasimorphisms and the Bavard duality}\label{sec=qmBavard}

In this section, we review the definitions of quasimorphisms and scl, and we state the original Bavard duality theorem. We refer the reader to \cite{Calegari} as a treatise on quasimorphisms and scl.

\subsection{Quasimorphisms}\label{subsec=qm}

\begin{definition}[quasimorphisms]\label{defn=qm}
Let $\Gg$ be a group.
\begin{enumerate}[(1)]
  \item A map $\phfh\colon \Gg\to \RR$ is called a \emph{quasimorphism} on $\Gg$ if there exists $D\in \RR_{\geq 0}$ such that for every $\gl_1,\gl_2\in \Gg$
\[
\phfh(\gl_1\gl_2)\sim_D \phfh(\gl_1)+\phfh(\gl_2)
\]
holds. The minimum of such $D$ is called the \emph{defect} $\DD(\phfh)$ of $\phfh$. In other words,
\[
\DD(\phfh)=\sup\{|\phfh(\gl_1\gl_2)-\phfh(\gl_1)-\phfh(\gl_2)| \;|\; \gl_1,\gl_2\in \Gg\}.
\]
We write  $\QhG$ for the $\RR$-linear space of quasimorphisms on $\Gg$.
\item A map $\phf\colon \Gg\to \RR$ is said to be \emph{homogeneous} if its restriction to every cyclic subgroup is a homomorphism, that is, $\phf(\gl^n)=n\phf(\gl)$ holds for every $\gl\in \Gg$ and every $n\in \ZZ$.
\item We write  $\QG$ for the $\RR$-linear space of homogeneous quasimorphisms on $\Gg$.
\end{enumerate}
\end{definition}

Lemma~\ref{lem=homoge1} below  allows us to focus on \emph{homogeneous} quasimorphisms in the study of quasimorphisms. For a real-valued function $\ff\colon X\to \RR$ on a set $X$, the \emph{$\ell^{\infty}$-norm} of $\ff$ is defined by $\|\ff\|_{\infty}=\sup\{|\ff(x)|\;|\;x\in X\}$. We write $\ell^{\infty}(X;\RR)$ for the real-valued $\ell^{\infty}$-space on $X$:
\[
\ell^{\infty}(X;\RR)=\{\ff\colon X\to \RR\; |\; \|\ff\|_{\infty}<\infty\}.
\]

We recall Fekete's lemma on subadditive sequences, which plays a key role to the homogenization process for quasimorphisms (Lemma~\ref{lem=homoge1} and Lemma~\ref{lem=homoge2}).

\begin{lemma}[Fekete's lemma]\label{lem=Fekete}
Let $(a_n)_{n\in \mathbb{N}}$ be a real sequence that is subadditive, meaning that $a_{m+n}\leq a_m+a_n$ holds for every $m,n\in \mathbb{N}$. Then the limit $\lim\limits_{n\to \infty}\dfrac{a_n}{n}$ exists in $[-\infty,\infty)$, and it equals $\inf\limits_{n\in \mathbb{N}}\dfrac{a_n}{n}$.

In particular, if $(a_n)_{n\in \mathbb{N}}$ is a non-negative real sequence that is subadditive, then the limit $\lim\limits_{n\to \infty}\dfrac{a_n}{n}$ exists in $[0,\infty)$, and it equals $\inf\limits_{n\in \mathbb{N}}\dfrac{a_n}{n}$.
\end{lemma}

\begin{lemma}[homogenization of quasimorphisms]\label{lem=homoge1}
Let $\Gg$ be a group. Then, for every $\phfh\in \QhG$, there exists a unique element $\phfh_{\mathrm{h}}$ in $\QG$ such that $\|\phfh-\phfh_{\mathrm{h}}\|_{\infty}<\infty$.
\end{lemma}

\begin{lemma}\label{lem=homoge2}
Let $\Gg$ be a group, and let $\phfh\in \QhG$. Let $\phfh_{\mathrm{h}}$ be the homogenization of $\phfh$. Then the following hold.
\begin{enumerate}
  \item[\textup{(1)}] For every $\gl\in \Gg$, $\phfh_{\mathrm{h}}(\gl)=\lim\limits_{n\to \infty}\dfrac{\phfh(\gl^n)}{n}$.
  \item[\textup{(2)}] $\|\phfh-\phfh_{\mathrm{h}}\|_{\infty}\leq \DD(\phfh)$.
\end{enumerate}
\end{lemma}

\begin{proof}[Proofs of Lemma~\textup{\ref{lem=homoge1}} and Lemma~\textup{\ref{lem=homoge2}}]
Let $\phfh\in \QhG$. Take an arbitrary $\gl\in \Gg$. Since the sequence $(\phfh(\gl^n)+\DD(\phfh))_n$ is subadditive, Fekete's lemma (Lemma~\ref{lem=Fekete}) implies that the limit $\lim\limits_{n\to \infty}\frac{\phfh(\gl^n)}{n}$ exists in $[-\infty,\infty)$. By applying Lemma~\ref{lem=Fekete}  to another sequence $(-\phfh(\gl^n)+\DD(\phfh))_n$, we conclude that $\lim\limits_{n\to \infty}\frac{\phfh(\gl^n)}{n}$ actually belongs to $\RR$. Hence $\phfh_{\mathrm{h}}\colon \Gg\to \RR$ can be defined as in the manner in Lemma~\ref{lem=homoge2}~(1). In what follows, we show that $\phfh_{\mathrm{h}}$ is homogeneous. Let $\gl\in \Gg$ and $n\in \ZZ$. If $n\geq 0$, then we have $\phfh_{\mathrm{h}}(\gl^n)=n\phfh_{\mathrm{h}}(\gl)$ by construction. If $n=-1$, then for every $m\in \NN$, we have
\[
\phfh(1_{\Gg})=\phfh(\gl^{-m} \gl^{m})\sim_{\DD(\phfh)}\phfh(\gl^{-m})+\phfh(\gl^{m}),
\]
so that
\[
\phfh(\gl^{-m})\sim_{|\phfh(1_{\Gg})|+\DD(\phfh)} -\phfh(\gl^{m}).
\]
Hence, we have
\begin{equation}\label{eq=-1vai}
\phfh_{\mathrm{h}}(\gl^{-1})=-\phfh_{\mathrm{h}}(\gl).
\end{equation}
Finally, by construction, \eqref{eq=-1vai} implies that $\phfh_{\mathrm{h}}(\gl^n)=n\phfh_{\mathrm{h}}(\gl)$ if $n<0$. Therefore, $\phfh_{\mathrm{h}}$ is homogeneous. Since $\phfh(\gl^n)\sim_{(n-1)\DD(\phfh)}n\phfh(\gl)$ for every $\gl\in \Gg$ and every $n\in \NN$, we have $\|\phfh-\phfh_{\mathrm{h}}\|_{\infty}\leq \DD(\phfh)$.
It is straightforward to show the uniqueness of the homogenization of $\phfh$.  Thus, we have proved Lemma~\ref{lem=homoge1} and Lemma~\ref{lem=homoge2}.
\end{proof}

By Lemma~\ref{lem=homoge2}~(2), we have an inequality $\DD(\phfh_{\mathrm{h}})\leq 4\DD(\phfh)$. In Corollary~\ref{cor=defect2bai}, we will prove a stronger inequality $\DD(\phfh_{\mathrm{h}})\leq 2\DD(\phfh)$.

Lemma~\ref{lem=homoge1} yields the following corollary.

\begin{corollary}\label{cor=isom}
Let $\Gg$ be a group. Then the map $\QhG\to \QG$ sending $\phfh\in \QhG$ to its homogenization $\phfh_{\mathrm{h}}\in \QG$ induces an isomorphism between $\RR$-linear spaces:
\[
\QhG/\left(\HGR +\ell^{\infty}(\Gg;\RR)\right) \cong \QG/\HGR.
\]
\end{corollary}

\begin{proof}
Note that $\ell^{\infty}(\Gg;\RR) \cap\QG=0$. Now, it is straightforward to see the isomorphism.
\end{proof}

\begin{lemma}\label{lem=defect}
Let $\Gg$ be a group, and let $\phf\in \QG$.
\begin{enumerate}
 \item[\textup{(1)}] For every $\gl_1,\gl_2\in \Gg$, we have $\phf(\gl_1\gl_2\gl_1^{-1})=\phf(\gl_2)$.
 \item[\textup{(2)}] For every $\gl_1,\gl_2\in \Gg$, we have $\phf([\gl_1,\gl_2])\sim_{\DD(\phf)}0$.
\end{enumerate}
\end{lemma}

\begin{proof}
Let $n\in \NN$. Then, by the homogeneity of $\phf$, we have
\begin{align*}
n\phf(\gl_1\gl_2\gl_1^{-1})&=\phf(\gl_1\gl_2^n\gl_1^{-1})\\
&\sim_{2\DD(\phf)}\phf(\gl_1)+\phf(\gl_2^n)+\phf(\gl_1^{-1})=n\phf(\gl_2).
\end{align*}
Hence, $\phf(\gl_1\gl_2\gl_1^{-1})\sim_{2n^{-1}\DD(\phf)}\phf(\gl_2)$. By letting $n\to\infty$, we obtain item (1). For item (2), we deduce from item (1) that
\begin{align*}
\phf([\gl_1,\gl_2])&=\phf(\gl_1\gl_2\gl_1^{-1}\gl_2^{-1})\\
&\sim_{\DD(\phf)}\phf(\gl_1\gl_2\gl_1^{-1})+\phf(\gl_2^{-1})=\phf(\gl_2)-\phf(\gl_2)=0,
\end{align*}
as desired.
\end{proof}

The following result by Bavard states that the estimate in Lemma~\ref{lem=defect}~(2) is tight if $\gl_1$ and $\gl_2$ run over $\Gg$.

\begin{proposition}[{\cite{Bavard}}] \label{prop=commutator_bavard}
Let $\Gg$ be a group, and let $\phf \in \QQQ(\Gg)$. Then, we have
\[
\DD(\phf) = \sup_{\gl_1, \gl_2 \in \Gg} \big| \phf([\gl_1, \gl_2]) \big|.
\]
\end{proposition}

To prove Proposition~\ref{prop=commutator_bavard}, we use Lemma~\ref{lem=commutatorcal} below. For a group $\Gg$, we call an element in a group $\Gg$ of the form $[\gl_1,\gl_2](=\gl_1\gl_2\gl_1^{-1}\gl_2^{-1})$ for $\gl_1,\gl_2\in \Gg$ a \emph{simple commutator} in $\Gg$.

\begin{lemma}\label{lem=commutatorcal}
Let $\Gg$ be a group. Then, for every $n\in\NN$ and for every $\gl_1,\gl_2\in \Gg$, the element  $(\gl_1\gl_2)^{-2n}\gl_1^{2n}\gl_2^{2n}$ may be written as a product of $n$ simple commutators in $\Gg$.
\end{lemma}

\begin{proof}
Observe the following identity for every $n\in \ZZ$:
\begin{equation}\label{eq:n-comm}
[\gl_1^{-2n}\gl_2^{-2n}\gl_1^{-1},\gl_2^{-1}\gl_1^{2n}]=\gl_1^{-2n}\gl_2^{-2n}\gl_1^{-1}\gl_2^{-1}\gl_1^{2n+1}\gl_2^{2n+1}.
\end{equation}
We will show the assertion of the lemma by induction on $n$. For $n=1$, the assertion follows because  $(\gl_1\gl_2)^{-2}\gl_1^{2}\gl_2^{2}=\gl_2^{-1}[\gl_1^{-1},\gl_2^{-1}]\gl_2$. For induction step, assume the assertion for $n$, and we will prove the assertion for $n+1$. By \eqref{eq:n-comm}, we have
\begin{align*}
&(\gl_1\gl_2)^{-2n-2}\gl_1^{2n+2}\gl_2^{2n+2}\\
&=\gl_2^{-1}\left((\gl_2\gl_1)^{-2n-1}\gl_1^{2n+1}\gl_2^{2n+1}\right)\gl_2\\
&=\gl_2^{-1}\left((\gl_2\gl_1)^{-2n-1}\gl_2\gl_1\gl_2^{2n}\gl_1^{2n}[\gl_1^{-2n}\gl_2^{-2n}\gl_1^{-1},\gl_2^{-1}\gl_1^{2n}]\right)\gl_2\\
&=\gl_2^{-1}\left((\gl_2\gl_1)^{-2n}\gl_2^{2n}\gl_1^{2n}[\gl_1^{-2n}\gl_2^{-2n}\gl_1^{-1},\gl_2^{-1}\gl_1^{2n}]\right)\gl_2
\end{align*}
By induction hypothesis, $(\gl_2\gl_1)^{-2n}\gl_2^{2n}\gl_1^{2n}$ is a product of $n$ simple commutators. Then, we conclude that $\gl_2^{-1}\left((\gl_1\gl_2)^{-2n-2}\gl_1^{2n+2}\gl_2^{2n+2}\right)\gl_2$ is a product of $n+1$ simple commutators. Since $a_3[a_1,a_2]a_3^{-1}=[a_3a_1a_3^{-1},a_3a_2a_3^{-1}]$ for every $a_1,a_2,a_3\in \Gg$, this implies that $(\gl_1\gl_2)^{-2n-2}\gl_1^{2n+2}\gl_2^{2n+2}$ may be written as a product of $n+1$ simple commutators, as desired. This completes the proof.
\end{proof}

\begin{proof}[Proof of Proposition~\textup{\ref{prop=commutator_bavard}}]
Take an arbitrary $\varepsilon \in \RR_{>0}$ and take $\gl_1,\gl_2\in \Gg$ such that $|\phf(\gl_1)+\phf(\gl_2)-\phf(\gl_1\gl_2)|\geq \DD(\phf)-\varepsilon$. For $n\in \NN$, set $\hl_n=(\gl_1\gl_2)^{-2n}\gl_1^{2n}\gl_2^{2n}$. By Lemma~\ref{lem=commutatorcal}, $\hl_n$ may be expressed as $\hl_n=\hl_n^{(1)}\cdots \hl_n^{(n)}$, where for every $i\in \{1,\ldots,n\}$ the element $\hl_n^{(i)}$ is a simple commutator in $\Gg$. Then,
\[
\phf(\hl_n)\sim_{(n-1)\DD(\phf)}\phf(\hl_n^{(1)})+\cdots+\phf(\hl_n^{(n)}).
\]
We also  have
\begin{align*}
\phf(\hl_n)&\sim_{2\DD(\phf)}\phf(\gl_1^{2n})+\phf(\gl_2^{2n})+\phf((\gl_1\gl_2)^{-2n})\\
&=2n \Bigl(\phf(\gl_1)+\phf(\gl_2)-\phf(\gl_1\gl_2)\Bigr).
\end{align*}
Therefore, we conclude that
\[
\left|\sum_{i=1}^n\phf(\hl_n^{(i)})\right| \geq (n-1)\DD(\phf)-2n\varepsilon.
\]
In particular, there exists $i\in \{1,\ldots,n\}$ such that
\[
|\phf(\hl_n^{(i)})|\geq \left(1-\frac{1}{n}\right)\DD(\phf)-2\varepsilon.
\]
Since such $\hl_n^{(i)}$ is a simple commutator in $\Gg$, this implies that
\[
\sup_{\gl_1, \gl_2 \in \Gg} \big| \phf([\gl_1, \gl_2]) \big|\geq \left(1-\frac{1}{n}\right)\DD(\phf)-2\varepsilon.
\]
By letting $n\to \infty$ and $\varepsilon \searrow 0$, we obtain that
\[
\DD(\phf) \leq \sup_{\gl_1, \gl_2 \in \Gg} \big| \phf([\gl_1, \gl_2]) \big|.
\]
The converse inequality follows from Lemma~\ref{lem=defect}~(2). This completes our proof.
\end{proof}

\subsection{Ordinary and bounded cohomology of groups}\label{subsec=cohomology}

Here we briefly recall the definition of ordinary and bounded cohomology of groups with coefficients in an abelian group $A$ (with trivial $\Gg$-action; also equipped with a norm for defining bounded cohomology). We refer the reader to \cite{Monod, Frigerio} for details on bounded cohomology. Let $\Gg$ be a group. Let $n\in \mathbb{Z}$. Let $\CCC^n(\Gg;A)$ be the space of $A$-valued functions on
 the $n$-fold direct product $\Gg^{n}$ of $\Gg$
if $n\geq 0$, and set $\CCC^n(\Gg;A) = 0$ if $n < 0$.
Define $\delta \colon \CCC^n(\Gg;A) \to \CCC^{n+1}(\Gg;A)$ by
\begin{align*}
\delta c(g_1, \cdots, g_{n+1}) = c(g_2, \cdots, g_{n+1}) &+ \sum_{i=1}^n (-1)^i c(g_1, \cdots, g_{i} g_{i+1}, \cdots, g_{n+1}) \\
&+ (-1)^{n+1} c(g_1, \cdots, g_{n}).
\end{align*}
The \emph{$n$-th group cohomology} $\HHH^n(\Gg;A)$ of $\Gg$ with trivial $A$-coefficients is the $n$-th cohomology group of the cochain complex $(\CCC^{\ast}(\Gg;A), \delta)$. In particular, $\HHH^1(\Gg;A)$ can be naturally identified with  $\mathrm{Hom}(\Gg,A)$.

Now, assume that $A$ is equipped with a norm $\|\cdot\|$. For a set $S$, an $A$-valued function $\ff\colon S\to A$ is said to be \emph{bounded} if $\sup\{\|\ff(s)\|\;|\; s\in S\}$ is bounded. Let $\CCC^n_b(\Gg;A)$ be the space of bounded $A$-valued functions on
$\Gg^{n}$
if $n\geq 0$; $\CCC^n_b(\Gg;A)=0$ if $n<0$. Then $\CCC^{\ast}_b(\Gg;A)$ is a subcomplex of $\CCC^*(G;A)$. Define $\HHH^{\ast}_b(G;A)$ as the cohomology of $\CCC^{\ast}_b(G;A)$: this is called the \emph{bounded cohomology} of $\Gg$ with trivial $A$-coefficients. The map $c_{\Gg;A}^n \colon \HHH^{n}_b(\Gg;A) \to \HHH^{n}(\Gg;A)$ induced by the inclusion $\CCC^{n}_b(\Gg;A) \hookrightarrow \CCC^{n}(\Gg;A)$ is called the \emph{$n$-th comparison map}  with trivial $A$-coefficients. For bounded cohomology, the case where $(A,\|\cdot\|)=(\RR,|\cdot|)$ ($\RR$ equipped with the ordinary absolute value) is of special interest. In this case, we may abbreviate $c_{\Gg;\RR}^n$ as $c_{\Gg}^n$.

The following lemma describes a relation between quasimorphisms and bounded cohomology.

\begin{lemma}\label{lem=Q/H}
Let $\Gg$ be a group. Then, the coboundary map $\delta\colon \CCC^1(\Gg;\RR)\to \CCC^2(\Gg;\RR)$ induces an isomorphism between  the quotient space $\QG/\HGR$ and the kernel of $c^2_{\Gg;\RR}=c^2_{\Gg}\colon \HHH^2_b(\Gg;\RR)\to \HHH^2(\Gg;\RR)$.
\end{lemma}

\begin{proof}
For $\gl_1,\gl_2\in \Gg$, $\delta\phf(\gl_1,\gl_2)=\phf(\gl_2)-\phf(\gl_1\gl_2)+\phf(\gl_1)$. Hence, $\delta$ sends $\QhG$ to $\mathrm{Z}^2_b(\Gg;\RR)$. Here $\mathrm{Z}^2_b(\Gg;\RR)$ is the space of bounded $2$-cocycles with trivial real coefficients (the kernel of the coboundary map $\delta \colon \CCC^2_b(\Gg;\RR) \to \CCC^{3}_b(\Gg;\RR)$). Now we compose it with the canonical projection $\mathrm{Z}^2_b(\Gg;\RR)\to \HHH^2_b(\Gg;\RR)$ and obtain a map from $\QhG$ to $\HHH^2_b(\Gg;\RR)$. It is straightforward to show that the kernel of this map is $\HGR+\ell^{\infty}(\Gg;\RR)$ and that the image of this map is the kernel of $c^2_{\Gg}$. Hence, $\delta$ induces an isomorphism
\[
\QhG/\left(\HGR+\ell^{\infty}(\Gg;\RR)\right) \cong \Ker(c^2_{\Gg}).
\]
By Corollary~\ref{cor=isom}, we complete our proof.
\end{proof}

\subsection{$\mathrm{scl}_{G}$ and the Bavard duality}\label{subsec=scl}
For a group $\Gg$, the \emph{commutator subgroup} $\CG$ of $\Gg$ is the subgroup of $\Gg$ generated by the set of simple commutators $\{[\gl_1,\gl_2]\;|\;\gl_1,\gl_2\in \Gg\}$.
\begin{definition}[scl]\label{defn=scl}
Let $\Gg$ be a group.
\begin{enumerate}[(1)]
  \item The \emph{commutator length} $\clG\colon \CG\to \ZZ_{\geq 0}$ is defined to be the word length with respect to the set of simple commutators. That is, for every $\hl\in \CG$, $\clG(\hl)$ is the least number of simple commutators in $\Gg$ whose product is $\hl$. In particular, $\clG(1_{\Gg})=0$.
  \item The \emph{stable commutator length} $\sclG\colon \CG\to \RR_{\geq 0}$ is defined as
\begin{equation}\label{eq=sclGG}
\sclG(\hl)=\lim_{n\to \infty}\frac{\clG(\hl^n)}{n}
\end{equation}
for every $\hl\in \CG$.
\end{enumerate}
\end{definition}
We note that the limit in \eqref{eq=sclGG}  exists by Fekete's lemma (Lemma~\ref{lem=Fekete}). We also observe that a conjugate of a simple commutator is also a simple commutator. Indeed, for $\gl_1,\gl_2,\gl_3\in \Gg$, we have $\gl_3[\gl_1,\gl_2]\gl_3^{-1}=[\gl_3\gl_1\gl_3^{-1},\gl_3\gl_2\gl_3^{-1}]$. Hence, $\clG$ and $\sclG$ are both invariant under the $\Gg$-action on $[\Gg,\Gg]$ by conjugation (the $\Gg$-invariance of $\clG$ already appears in the proof of Lemma~\ref{lem=commutatorcal}).

\begin{lemma}\label{lem=cl}
Let $\Gg$ be a group. Then, for every $\hl\in \CG$ and for every $\phf\in \QG$, we have
\[
2\DD(\phf)\clG(\hl)\geq |\phf(\hl)|.
\]
\end{lemma}

\begin{proof}
Set $n=\clG(\hl)$. We may assume that $n\geq 1$; otherwise the desired inequality is clear. Express $\hl$ as $\hl=\hl_1\hl_2\cdots \hl_{n}$, where $\hl_1,\ldots,\hl_n$ are simple commutators in $\Gg$. Then, by Lemma~\ref{lem=defect}~(2) we have
\begin{align*}
\phf(\hl)&\sim_{(n-1)\DD(\phf)} \phf(\hl_1)+\phf(\hl_2)+\cdots+\phf(\hl_n)\\
&\sim_{n\DD(\phf)}0.
\end{align*}
Therefore we have $\phf(\hl)\sim_{(2n-1)\DD(\phf)}0$ if $n=\clG(\hl)\geq 1$. In particular, we have $|\phf(\hl)|\leq 2\DD(\phf)\clG(\hl)$.
\end{proof}

\begin{corollary}\label{cor=sclabove}
Let $\Gg$ be a group, and let $\hl\in \CG$. Then, for every $\phf\in \QG$,
\[
2\DD(\phf)\sclG(\hl)\geq |\phf(\hl)|
\]
holds.
\end{corollary}

\begin{proof}
Let $n\in \NN$. Apply Lemma~\ref{lem=cl} to $\hl^n$. Then we have
\[
2\DD(\phf)\clG(\hl^n)\geq |\phf(\hl^n)|.
\]
Hence we obtain that
\[
2\DD(\phf)\cdot \frac{\clG(\hl^n)}{n}\geq |\phf(\hl)|.
\]
By letting $n\to \infty$, we have the desired inequality.
\end{proof}

The celebrated Bavard duality theorem states that the `converse' direction to the inequality in Corollary~\ref{cor=sclabove} holds if $\phf$ runs over $\QG\setminus \HGR$. We state the theorem, again from Theorem~\ref{theorem=Bavard}.

\begin{theorem}[Bavard duality theorem, \cite{Bavard}]\label{theorem=Bavardagain}
Let $\Gg$ be a group. Then for every $\hl\in \CG$, we have
\begin{equation}\label{eq=Bavard}
\sclG(\hl)=\sup_{\phf\in \QG\setminus \HGR}\frac{|\phf(\hl)|}{2\DD(\phf)}.
\end{equation}
\end{theorem}

In the present article, we will prove the `\emph{generalized} and \emph{mixed}' version of this theorem.

\begin{remark}\label{rem=quotient}
In the setting of Theorem~\ref{theorem=Bavardagain}, for the equivalence class $[\phf]\in \QG/ \HGR$ of $\phf\in \QG$, both of the value $\phf'(\hl)$ and the defect $\DD(\phf')$ are independent of the choices of a representative $\phf'\in [\phf]$ (recall that $\hl\in [\Gg,\Gg]$). From the latter independence, $\DD$ induces a genuine norm on the quotient vector space $\QG/\HGR$; we write $\DD$ for this norm as well. Thus, \eqref{eq=Bavard} may be restated as
\[
\sclG(\hl)=\sup_{[\phf]\in \big(\QG/\HGR\big)\setminus \{0\} }\frac{|\phf(\hl)|}{2\DD([\phf])}.
\]
\end{remark}

As we mentioned in Subsection~\ref{subsec=qm}, we have the following inequality between the defect of a quasimorphism and that of its homogenization.

\begin{proposition}\label{prop=defect2bai}
Let $\Gg$ be a group. Let $\phfh\in \QhG$, and let $\phfh_{\mathrm{h}}\in \QG$ be its homogenization. Set
\[
\DD'_{\Gg}(\phfh)=\sup\left\{\left|\phfh(\gl_1\gl_2\gl_1^{-1})-\phfh(\gl_2)\right|\;\middle|\; \gl_1,\gl_2\in \Gg\right\}.
\]
Then, we have
\[
\DD(\phfh_{\mathrm{h}})\leq \DD(\phfh)+\frac{1}{2}\DD'_{\Gg}(\phfh).
\]
\end{proposition}

The real number $\DD'_{\Gg}(\phfh)$ measures the ($\Gg$-)quasi-invariance of $\phfh$; we will treat a related notion in Definition~\ref{defn=qinv}. To prove Proposition~\ref{prop=defect2bai}, we will use the anti-symmetrization of a quasimorphism. A real-valued function $\ff$ on a group $\Gg$ is said to be \emph{anti-symmetric} if $\ff(\gl^{-1})=-\ff(\gl)$ holds for every $\gl\in \Gg$.

\begin{lemma}[anti-symmetrization of quasimorphisms]\label{lem=antisym}
Let $\Gg$ be a group. Let $\phfh\in \QhG$. Set $\phfh_{\mathrm{as}}$ by
\[
\phfh_{\mathrm{as}}(\gl)=\frac{\phfh(\gl)-\phfh(\gl^{-1})}{2}
\]
for $\gl\in \Gg$. Then the following hold.
\begin{enumerate}[$(1)$]
  \item $\phfh_{\mathrm{as}}$ is anti-symmetric.
  \item $\|\phfh_{\mathrm{as}}-\phfh\|_{\infty}<\infty$.
  \item $\phfh_{\mathrm{as}}\in \QhG$ and $\DD(\phfh_{\mathrm{as}})\leq \DD(\phfh)$.
  \item $\DD'_{\Gg}(\phfh_{\mathrm{as}})\leq \DD'_{\Gg}(\phfh)$.
  \item $(\phfh_{\mathrm{as}})_{\mathrm{h}}=\phfh_{\mathrm{h}}$, where $\psfh_{\mathrm{h}}$ means the homogenization of $\psfh$ for $\psfh\in \QhG$.
   \item $\DD'_{\Gg}(\phfh_{\mathrm{as}})\leq 2\DD(\phfh_{\mathrm{as}})$.
  \item For every simple commutator $\hl$ in $\Gg$, $\phfh_{\mathrm{as}}(\hl)\sim_{\DD(\phfh_{\mathrm{as}})+\DD'_{\Gg}(\phfh_{\mathrm{as}})} 0$.
  \item For $\hl\in \CG$,
\[
|\phfh_{\mathrm{as}}(\hl)|\leq \left(2\DD(\phfh_{\mathrm{as}})+\DD'_{\Gg}(\phfh_{\mathrm{as}})\right)\clG(\hl).
\]
\end{enumerate}
\end{lemma}

\begin{proof}
By construction, item (1) follows. Since $\phfh\in \QhG$, we have item (2). The triangle inequality shows that $\DD(\phfh_{\mathrm{as}})\leq \DD(\phfh)$; this proves (3). Item (4) also follows from the triangle inequality. Item (5) follows from item (1) and Lemma~\ref{lem=homoge1}. To show item (6), for $\gl_1,\gl_2\in \Gg$, we have
\begin{align*}
\phfh_{\mathrm{as}}(\gl_1\gl_2\gl_1^{-1})&\sim_{2\DD(\phfh_{\mathrm{as}})}\phfh_{\mathrm{as}}(\gl_1)+\phfh_{\mathrm{as}}(\gl_2)+\phfh_{\mathrm{as}}(\gl_1^{-1})\\
&=\phfh_{\mathrm{as}}(\gl_1)+\phfh_{\mathrm{as}}(\gl_2)-\phfh_{\mathrm{as}}(\gl_1)=\phfh_{\mathrm{as}}(\gl_2).
\end{align*}
Hence $\DD'_{\Gg}(\phfh_{\mathrm{as}})\leq 2\DD(\phfh_{\mathrm{as}})$.
For item (7), write $\hl=[\gl_1,\gl_2]$, where $\gl_1,\gl_2\in \Gg$. Then we have
\begin{align*}
\phfh_{\mathrm{as}}(\hl)&\sim_{\DD(\phfh_{\mathrm{as}})} \phfh_{\mathrm{as}}(\gl_1\gl_2\gl_1^{-1})+\phfh_{\mathrm{as}}(\gl_2^{-1})\\
&\sim_{\DD'_{\Gg}(\phfh_{\mathrm{as}})} \phfh_{\mathrm{as}}(\gl_2)+\phfh_{\mathrm{as}}(\gl_2^{-1})\\
&=\phfh_{\mathrm{as}}(\gl_2)-\phfh_{\mathrm{as}}(\gl_2)=0,
\end{align*}
as desired. Finally, item (8) can be deduced from item (7) in a manner similar to one in the proof of Lemma~\ref{lem=cl}.
\end{proof}

The map $\phfh_{\mathrm{as}}$ in Lemma~\ref{lem=antisym} is called the \emph{anti-symmetrization} of $\phfh$.

\begin{proof}[Proof of Proposition~\textup{\ref{prop=defect2bai}}]
Set $\phfh_{\mathrm{as}}$ to be the anti-symmetrization of $\phfh$.
Let $\gl_1,\gl_2\in \Gg$. Let $n\in \NN$. Then, by anti-symmetricity of $\phfh_{\mathrm{as}}$ we have
\begin{align*}
&\phfh_{\mathrm{as}}(\gl_1^{2n})+\phfh_{\mathrm{as}}(\gl_2^{2n})-\phfh_{\mathrm{as}}((\gl_1\gl_2)^{2n})\\
&=\phfh_{\mathrm{as}}(\gl_1^{2n})+\phfh_{\mathrm{as}}(\gl_2^{2n})+\phfh_{\mathrm{as}}((\gl_1\gl_2)^{-2n})\\
&\sim_{2\DD(\phfh_{\mathrm{as}})} \phfh_{\mathrm{as}}((\gl_1\gl_2)^{-2n}\gl_1^{2n}\gl_2^{2n})\\
&\sim_{n\left(2\DD(\phfh_{\mathrm{as}})+\DD'_{\Gg}(\phfh_{\mathrm{as}})\right)} 0.
\end{align*}
Here, the last estimate follows from Lemma~\ref{lem=commutatorcal} and Lemma~\ref{lem=antisym}~(8). Therefore, we obtain that
\[
\left|\frac{\phfh_{\mathrm{as}}(\gl_1^{2n})}{2n}+\frac{\phfh_{\mathrm{as}}(\gl_2^{2n})}{2n}-\frac{\phfh_{\mathrm{as}}((\gl_1\gl_2)^{2n})}{2n}\right|\leq \left(1+\frac{1}{n}\right)\DD(\phfh_{\mathrm{as}})+\frac{1}{2}\DD'_{\Gg}(\phfh_{\mathrm{as}}).
\]
By letting $n\to \infty$ we deduce from Lemma~\ref{lem=homoge2}~(1) and Lemma~\ref{lem=antisym}~(5) that
\[
\left|\phfh_{\mathrm{h}}(\gl_1)+\phfh_{\mathrm{h}}(\gl_2)-\phfh_{\mathrm{h}}(\gl_1\gl_2)\right|\leq \DD(\phfh_{\mathrm{as}})+\frac{1}{2}\DD'_{\Gg}(\phfh_{\mathrm{as}}).
\]
By taking the supremum over $\gl_1,\gl_2\in \Gg$, we conclude that
\[
\DD(\phfh_{\mathrm{h}})\leq \DD(\phfh_{\mathrm{as}})+\frac{1}{2}\DD'_{\Gg}(\phfh_{\mathrm{as}}).
\]
Together with Lemma~\ref{lem=antisym}~(3) and (4), we obtain the desired inequality.
\end{proof}


\begin{corollary}[defect estimate under the homogenization process]\label{cor=defect2bai}
Let $\Gg$ be a group. Let $\phfh\in \QhG$, and let $\phfh_{\mathrm{h}}$ be its homogenization. Then, we have
\[
\DD(\phfh_{\mathrm{h}})\leq 2\DD(\phfh).
\]
\end{corollary}

\begin{proof}
We have
\begin{align*}
\DD(\phfh_{\mathrm{h}})&=\DD((\phfh_{\mathrm{as}})_{\mathrm{h}})\\
&\leq \DD(\phfh_{\mathrm{as}})+\frac{1}{2}\DD'_{\Gg}(\phfh_{\mathrm{as}})\\
&\leq \DD(\phfh_{\mathrm{as}})+\frac{1}{2}\cdot 2\DD(\phfh_{\mathrm{as}})\\
&\leq \DD(\phfh)+\frac{1}{2}\cdot 2\DD(\phfh)=2\DD(\phfh).
\end{align*}
Here, the first equality, the second inequality, the third inequality, and the fourth inequailty follow from Lemma~\ref{lem=antisym}~(5), Proposition~\ref{prop=defect2bai}, Lemma~\ref{lem=antisym}~(6), and Lemma~\ref{lem=antisym}~(3), respectively.
\end{proof}

\section{The generalized Bavard duality}\label{sec=gBavard}

In this section, we state the generalized Bavard duality theorem due to Calegari \cite{Calegari1}. For simplicity, here we only treat integral chains.

Let $\Gg$ be a group. Fix a unital commutative ring $A$. Recall that $\CCC_2(\Gg ; A)$ denotes the set of $2$-chains of $\Gg$ with $A$-coefficients. Namely, $\CCC_2(\Gg ; A)$ is the free $A$-module generated by the $2$-fold direct product $\Gg^2=\Gg \times \Gg$ of $\Gg$. Similarly, $\CCC_1(\Gg ; A)$ is the free $A$-module generated by $\Gg$. The boundary map $\partial\colon \CCC_2(\Gg ; A)\to \CCC_1(\Gg ; A)$ is defined as
\[
\partial ((\gl_1,\gl_2))=\gl_2-\gl_1\gl_2+\gl_1
\]
for every $(\gl_1,\gl_2)\in \Gg^2$, extended as an $A$-linear map. Finally, we recall  the definition of ($1$-)boundaries.

\begin{definition}\label{defn=BBB}
Define the set
$\BBB_1(\Gg;A)$ to be $\partial \CCC_2(\Gg;A)$.
\end{definition}

In the case of $A=\ZZ$, we have the following lemma.

\begin{lemma}\label{lem=domainG}
Let $\Gg$ be a group. Then, the following two conditions on $c\in \CCC_1(\Gg;\ZZ)$ are equivalent.
\begin{enumerate}
 \item[\textup{(i)}] There exist $k,l\in \ZZ_{\geq 0}$ and $\gl_1, \cdots, \gl_k, \gbr_1, \cdots, \gbr_l \in \Gg$ such that
\[ c = \gl_1 + \cdots + \gl_k - \gbr_1 - \cdots - \gbr_l\]
and $\gl_1 \cdots \gl_k \gbr_1^{-1} \cdots \gbr_l^{-1} \in [\Gg,\Gg]$.
 \item[\textup{(ii)}] $c$ belongs to $\BBB_1(\Gg;\ZZ)$.
\end{enumerate}
\end{lemma}

Here, we do not prove Lemma~\ref{lem=domainG}. Instead, we will prove a generalization (Lemma~\ref{lemma=domain}) of Lemma~\ref{lem=domainG} in Section~\ref{sec=gmBavard}. Indeed, Lemma~\ref{lemma=domain} for the special case of $\Ng=\Gg$ recovers Lemma~\ref{lem=domainG} (observe Remark~\ref{rem=BBB_C}).

For an integral chain $c\in \BBB_1(\Gg;\ZZ)$, we define the following notions.

\begin{definition}[cl and scl for integral chains]\label{defn=chainonG}
Let $\Gg$ be a group. Let $c\in \BBB_1(\Gg;\ZZ)$. Write $c=\gl_1 + \cdots + \gl_k - \gbr_1 - \cdots - \gbr_l$ as in Lemma~\ref{lem=domainG}~\textup{(i)}.
\begin{enumerate}[(1)]
  \item For $\phf\in \QG$, define
\[
\phf(c)=\phf(\gl_1) + \cdots + \phf(\gl_k) - \phf(\gbr_1) - \cdots - \phf(\gbr_l).
\]
  \item Define
\[
\cl_{\Gg}(c)=\inf \cl_{\Gg}(\xi_1\cdots \xi_k \eta_1^{-1} \cdots \eta_l^{-1} ).
\]
Here, in the infimum in the right-hand side of the equality, for every $i\in \{1,\ldots,k\}$, $\xi_i$ runs over the conjugacy class of $\gl_i$ in $\Gg$; for every $j\in \{1,\ldots,l\}$, $\eta_j$ runs over the conjugacy class of $\gbr_j$ in $\Gg$.
  \item Define
\[
\scl_{\Gg}(c)=\lim_{n\to \infty} \frac{1}{n} \cl_{\Gg}(\gl_1^n + \cdots + \gl_k^n - \gbr_1^n - \cdots - \gbr_l^n).
\]
\end{enumerate}
\end{definition}

The well-definedness of $\cl_{\Gg}(c)$ and $\scl_{\Gg}(c)$ are both non-trivial; but here we do not present the proofs. Instead, we treat these notions (Definition~\ref{definition=commutator_lengths_for_integral_chains}, Definition~\ref{defn=scl_chain} and Definition~\ref{definition=scl_integral_2}) in the setting of a pair $(\Gg,\Ng)$ in Section~\ref{sec=gmBavard}: the special case of $\Ng=\Gg$ guarantees the well-definedness of Definition~\ref{defn=chainonG}.

Now we are ready to state the generalized Bavard duality theorem for integral chains.

\begin{theorem}[generalized Bavard duality theorem for integral chains, \cite{Calegari}]\label{thm=gBavardZ}
Let $\Gg$  be a group. Then, for every chain $c\in \BBB_1(\Gg;\ZZ)$, we have
\[
\sclG(c)=\sup_{\phf\in \QG\setminus \HGR}\frac{|\phf(c)|}{2\DD(\phf)}.
\]
\end{theorem}

As is mentioned in the introduction, Calegari \cite{Calegari} in fact established the generalized Bavard duality for real chains (Theorem~\ref{theorem=generalized_Bavard}). Here, we do not present the formulations in real chains. Instead, we discuss the general case (Definition~\ref{defn=scl_Q}, Definition~\ref{definition=scl_real} and Theorem~\ref{generalized mixed bavard}) for a pair $(\Gg,\Ng)$ in Section~\ref{sec=gmBavard}.

Here we briefly review the background of the generalized Bavard duality theorem.
In \cite{Calegari1}, Calegari proved that $\scl$ on a free group $F$ takes on only rational number values.
Furthermore, given an element of $[F,F]$, he provided an explicit algorithm to compute the $\scl$ of it.
By considering $\scl_F$ as a seminorm on $\BBB_1(F;\RR)$, he further discovered the following:
for every finite dimensional $\QQ$-linear subspace $V$ of $\BBB_1(F;\RR)$, the unit ball of the scl seminorm restricted to $V$ is a finite-sided rational convex polyhedron.
Based on this phenomenon, which is analogous to the Gromov--Thurston norm of a $3$-manifold, he investigated polyhedral structures on the unit ball with respect to the  $\scl$ seminorm in \cite{Calegari2}.
Furthermore, he studied the problem of whether an immersed loop on a surface bounds (in a certain sense) an immersed surface, and established a theorem \cite[Theorem C]{Calegari2}. To prove this theorem, he obtained a characterization of  this bounding in terms of scl \cite[Proposition 3.8]{Calegari2}; for this characterization, the generalized Bavard theorem plays the key role.

We finally note that even for the study of  scl of group elements, scl of chains may naturally show up. Here we exhibit the following two results (Proposition~\ref{prop=chain1} and \ref{prop=chain2}) from \cite{Calegari}.

\begin{proposition}[{see \cite[Corollary 2.81]{Calegari}}]\label{prop=chain1}
Let $\Gg$ be a group, and let $\Ng$ be a normal subgroup of $\Gg$ of finite index.
Let $A\subset \Gg$ be a complete set of representatives of $\Gam=\Gg/\Ng$. Then, for every $\xl \in [\Ng,\Ng]$, we have

\[ \scl_{\Gg} (\xl) = \frac{1}{\#
A} \scl_{\Ng} \left( \sum_{a \in A} a\xl a^{-1} \right). \]

\end{proposition}

By definition (Definition~\ref{defn=chainonG}), the right-hand side of the formula above does not depend on the choice of a complete set $A$ of representatives of $\Gam$.

\begin{proposition}[{see \cite[Theorem 2.101]{Calegari}}]\label{prop=chain2}
Let $\Gg$ be a group. Let $t$ be a generator of $\ZZ$, and set $\tilde{G} = \Gg \ast \ZZ(=\Gg\ast \langle t\rangle)$. Then, for every $\gl_1,\gl_2\in \Gg$ that are of infinite order, we have
\[ \scl_{\tilde{\Gg}}(\gl_1t\gl_2t^{-1}) = \sclG(\gl_1 + \gl_2) + \frac{1}{2}.\]\end{proposition}


\section{Invariant quasimorphisms and the mixed Bavard duality}\label{sec=invqm}

In this section, we extend the framework of Section~\ref{sec=qmBavard} to that for a pair $(\Gg,\Ng)$ of a group and its normal subgroup; the case where $\Ng=\Gg$ recovers the original framework. We refer the reader to \cite{KKMMMsurvey} for a survey on this `mixed' framework; for instance, one of the motivations of studying this setting comes from symplectic geometry.

\begin{definition}[invariant quasimorphisms]\label{defn=invqm}
Let $\Gg$ be a group and $\Ng$ its normal subgroup.
\begin{enumerate}[(1)]
  \item A function $\ff\colon \Ng\to \RR$ is said to be \emph{$\Gg$-invariant} if for every $\gl\in \Gg$ and every $\xl\in \Ng$, $\ff(\gl\xl\gl^{-1})=\ff(\xl)$ holds.
  \item Let $V$ be a space of real valued functions on $\Ng$ that admits the $\Gg$-action by conjugation. Then, we write $V^{\Gg}$ for the $\Gg$-invariant part, that is, the subspace of $V$ consisting of $G$-invariant functions in $V$. In particular, the space $\QNG$ means the space of $\Gg$-invariant homogeneous quasimorphisms on $\Ng$; the space $\HNRG$ means the space of $\Gg$-invariant (genuine) homomorphisms from $\Ng$ to $\RR$.
\end{enumerate}
\end{definition}

Lemma~\ref{lem=defect}~(1) can be summarized in the following manner, which supplies us one of the motivations to study invariant quasimorphisms.

\begin{lemma}\label{lem=invariant}
Let $\Gg$ be a group. Then we have
\[
\QG=\QG^{\Gg}.
\]
\end{lemma}

The concepts of $\clG$ and $\sclG$ can be extended to the present framework as follows.

\begin{definition}[mixed commutator subgroups, mixed scl]\label{defn=mixedscl}
Let $\Gg$ be a group and $\Ng$ its normal subgroup.
\begin{enumerate}[(1)]
  \item A \emph{simple mixed commutator} (or, a \emph{simple $(\Gg,\Ng)$-commutator}) means an element in $\Gg$ of the form $[\gl,\xl]$, where $\gl\in \Gg$ and $\xl\in \Ng$. The {mixed commutator subgroup} (or, the \emph{$(\Gg,\Ng)$-commutator subgroup}) $\CGN$ is defined as the subgroup of $\Gg$ generated by the set of simple mixed commutators.
  \item The \emph{mixed commutator length} (or, the \emph{$(\Gg,\Ng)$-commutator length}) $\clGN\colon \CGN\to \ZZ_{\geq 0}$ is defined to be the word length with respect to the set of simple mixed commutators.
  \item The \emph{stable mixed commutator length} (or, the \emph{stable $(\Gg,\Ng)$-commutator length}) $\sclGN\colon \CGN\to \RR_{\geq 0}$ is defined as
\[
\sclGN(\yl)=\lim_{n\to \infty}\frac{\clGN(\yl^n)}{n}
\]
for every $\yl\in \CGN$.
\end{enumerate}
\end{definition}

We note that $\CGN$ is a subgroup of $\Ng$ because $\Ng$ is normal in $\Gg$. We also observe that a $\Gg$-conjugate of a simple mixed commutator is a simple mixed commutator: indeed, for $\gl_1,\gl_2\in \Gg$ and for $\xl\in \Ng$, we have $
\gl_2[\gl_1,\xl]\gl_2^{-1}=[\gl_2\gl_1\gl_2^{-1},\gl_2\xl\gl_2^{-1}]$.
In particular, $[\xl,\gl]$ is a simple mixed commutator for $\gl\in \Gg$ and $\xl\in \Ng$ because $[\xl,\gl]=\gl[\gl^{-1},\xl]\gl^{-1}$. Also, $\clGN$ and $\sclGN$ are both $\Gg$-invariant.

\begin{proposition}\label{prop=sclabove}
Let $\Gg$ be a group and $\Ng$ its normal subgroup.
\begin{enumerate}
  \item[\textup{(1)}] Let $\muf\in \QNG$. Then we have
\[
\DD(\muf)= \sup_{\gl\in \Gg,\, \xl\in \Ng} \big| \muf([\gl, \xl]) \big|.
\]
  \item[\textup{(2)}] Let $\yl\in \CGN$. Then for every $\muf\in \QNG$, we have
\[
2\DD(\muf)\clGN(\yl)\geq |\muf(\yl)|.
\]
  \item[\textup{(3)}] Let $\yl\in \CGN$. Then for every $\muf\in \QNG$, we have
\[
2\DD(\muf)\sclGN(\yl)\geq |\muf(\yl)|.
\]
\end{enumerate}
\end{proposition}

\begin{proof}
For item (1), by Proposition~\ref{prop=commutator_bavard} we have
\[
\DD(\muf)=\sup_{\xl_1,\xl_2\in \Ng} \big| \muf([\xl_1, \xl_2]) \big|\leq \sup_{\gl\in \Gg,\, \xl\in \Ng} \big| \muf([\gl, \xl]) \big|.
\]
Here, observe that $\muf\in \QQQ(\Ng)$. To show $\DD(\muf)\geq  \sup\limits_{\gl\in \Gg,\, \xl\in \Ng} \big| \muf([\gl, \xl]) \big|$, note that for every $\gl\in \Gg$ and $\xl\in \Ng$,
\[
\muf([\gl,\xl])=\muf(\gl\xl\gl^{-1}\xl^{-1})\sim_{\DD(\muf)}\muf(\gl\xl\gl^{-1})+\muf(\xl^{-1})=0.
\]
Item (2) may be shown in a manner similar to the proof of Lemma~\ref{lem=cl}; now item (3) can be deduced from item (2) by the same argument as in the proof of Corollary~\ref{cor=sclabove}.
\end{proof}

Now we state the mixed Bavard duality theorem, again from Theorem~\ref{theorem=mixed_Bavard}.

\begin{theorem}[mixed Bavard duality theorem, \cite{KKMM1}]\label{theorem=mixedBavardagain}
Let $\Gg$ be a group and $\Ng$ its normal subgroup. Then for every $\yl\in \CGN$, we have
\begin{equation}\label{eq=BavardGN}
\sclGN(\yl)=\sup_{\muf\in \QNG\setminus \HNRG}\frac{|\muf(\yl)|}{2\DD(\muf)}.
\end{equation}
\end{theorem}

\begin{remark}\label{rem=quotientGN}
By an argument similar to one in Remark~\ref{rem=quotient}, $\DD$ induces a genuine norm on the quotient vector space $\QNG/\HNRG$; we write $\DD$ for this norm as well. Then,  \eqref{eq=BavardGN} may be restated as
\[
\sclGN(\yl)=\sup_{[\muf]\in \big(\QNG/\HNRG\big)\setminus \{0\} }\frac{|\muf(\yl)|}{2\DD([\muf])}.
\]
\end{remark}

The following notion of quasi-invariance will be used in Section~\ref{sec=gmBavard}.

\begin{definition}[quasi-invariance]\label{defn=qinv}
Let $\Gg$ be a group and $\Ng$ its normal subgroup. A quasimorphism $\mufh \colon \Ng \to \RR$ is said to be \emph{$\Gg$-quasi-invariant} if
\[
\DD'_{\Gg,\Ng}(\mufh)=\sup\left\{\left|\mufh(\gl \xl \gl^{-1}) - \mufh(\xl)\right|\; \middle|\; \gl\in \Gg,\ \xl\in \Ng\right\}
\]
is finite.
\end{definition}

An element $\mufh$ of $\hQQQ(\Ng)$ is $\Gg$-invariant if and only if $\DD'_{\Gg,\Ng}(\mufh)=0$. We have the following lemma.

\begin{lemma}\label{lem=qinv_inv}
Let $\Gg$ be a group and $\Ng$ its normal subgroup. If a $\Gg$-quasi-invariant quasimorphism $\muf \colon \Ng \to \RR$ is homogeneous, then $\muf$ is $\Gg$-invariant.
\end{lemma}

\begin{proof}
Let $\gl\in \Gg$ and $\xl\in \Ng$. Then for every $n\in \NN$,
\begin{align*}
n\muf(\gl \xl\gl^{-1})&=\muf((\gl \xl\gl^{-1})^n)=\muf(\gl \xl^n\gl^{-1})\\
&\sim_{\DD'_{\Gg,\Ng}(\muf)}\muf(\xl^n)=n\muf(\xl)
\end{align*}
by homogeneity. Hence, we have
\[
\muf(\gl \xl\gl^{-1})\sim_{\DD'_{\Gg,\Ng}(\muf)/n}\muf(\xl).
\]
By letting $n\to \infty$, we obtain the conclusion.
\end{proof}

\begin{corollary}\label{cor=homoge_qinv}
Let $\Gg$ be a group and $\Ng$ its normal subgroup. If a $\mufh\in \QhN$ is $\Gg$-quasi-invariant, then its homogenization $\mufh_{\mathrm{h}}$ belongs to $\QNG$.
\end{corollary}

\begin{proof}
It suffices to show that $\mufh_{\mathrm{h}}$ is $\Gg$-invariant. Since $\|\mufh_{\mathrm{h}}-\mufh\|_{\infty}<\infty$ and $\DD'_{\Gg,\Ng}(\mufh)<\infty$, we conclude that $\DD'_{\Gg,\Ng}(\mufh_{\mathrm{h}})<\infty$. Then, Lemma~\ref{lem=qinv_inv} applies to $\mufh_{\mathrm{h}}$.
\end{proof}

In relation to Remark~\ref{rem=quotientGN}, the study of $\sclGN$ might be completely done if the quotient space $\QNG/\HNRG$ is finite dimensional. However, the following proposition provides a wide class of pairs $(\Gg,\Ng)$ for which the dimension of $\QNG/\HNRG$ is the cardinality of the continuum. Here, we refer the reader to \cite{Osin} for the concept of acylindrically hyperbolic groups; non-elementary Gromov-hyperbolic groups, relatively hyperbolic groups and the mapping class group of  an (oriented) connected closed surface of genus at least two are typical examples. In Proposition~\ref{prop=AH}, the dimension of an $\RR$-linear space means the Hamel dimension over $\RR$, that is, the cardinality of a Hamel basis over $\RR$.

\begin{proposition}\label{prop=AH}
Let $\Gg$ be an acylindrically hyperbolic group and $\Ng$ an infinite normal subgroup of $\Gg$. Then, the  dimension of $\QNG/\HNRG$  is at least the cardinality of the continuum. If $\Ng$ is countable, then the  dimension of $\QNG/\HNRG$  equals the cardinality of the continuum.

In particular, if $\Gg$ is an acylindrically hyperbolic group that is countable, then the dimension of $\QG/\HGR$ equals the cardinality of the continuum.
\end{proposition}

\begin{proof}
Let $i_{[\Ng,\Ng],\Ng}$ be the inclusion map from $[\Ng,\Ng]$ to $\Ng$. Set $i_{[\Ng,\Ng],\Ng}^{\ast}\colon \QQQ(\Ng)^{\Gg}\to \QQQ([\Ng,\Ng])^{\Gg}$ be the induced map. In a manner similar to this, we set the map $i_{[\Ng,\Ng],\Gg}^{\ast}\colon \QQQ(\Gg)\to \QQQ([\Ng,\Ng])^{\Gg}$. By \cite[Corollary~1.5]{Osin}, every infinite normal subgroup of an acylindrically hyperbolic group is again acylindrially hyperbolic. Hence,  $\Ng$ is acylindrically hyperbolic; so is $[\Ng,\Ng]$. Then the argument of \cite[Corollary 4.3]{FW} (based on \cite{BF}) shows that $i_{[\Ng,\Ng],\Gg}^{\ast}$ has an infinite dimensional image. Since $i_{[\Ng,\Ng],\Ng}^{\ast}\HNRG=0$, we conclude that the image of $\QQQ([\Ng,\Ng])^{\Gg}$ under the canonical projection  $\QNG\twoheadrightarrow \QNG/\HNRG$ is already infinite dimensional. Hence, the space $\QNG/\HNRG$ is infinite dimensional. By \cite[Theorem~7.4]{KKMMM1}, the normed space $(\QNG/\HNRG,\DD)$ is a Banach space. General theory on Banach spaces tells us that the (Hamel) dimension of an infinite dimensional real Banach space is at least the cardinality of the continuum (see for instance \cite{Lacey}). Therefore, the  dimension of $(\QNG/\HNRG,\DD)$ is at least the cardinality of the continuum.

Now, we furthermore assume that $\Ng$ is countable. By the homogenization process of quasi-invariant quasimorphisms (Corollary~\ref{cor=homoge_qinv}), the space $\QNG$ is isomorphic to the quotient space of the $\Gg$-quasi-invariant quasimorphisms (not necessarily homogeneous) on $\Ng$ over $\ell^{\infty}(\Ng;\RR)$. By taking a $\QQ$-valued map as a representative of an element in this quotient space, we deduce from countability of $\Ng$ that the cardinality of $\QNG/\HNRG$ is at most that of the continuum. In particular, so is the dimension of $(\QNG/\HNRG,\DD)$. Therefore, we conclude that the dimension of $(\QNG/\HNRG,\DD)$ equals the cardinality of the continuum, as desired. The final assertion on $\QG/\HGR$ immediately follows from the special case where $\Ng=\Gg$.
\end{proof}

Contrastingly, we will see in Subsection~\ref{subsec=dimW} that the space $\WGN$, defined in the introduction, is \emph{finite dimensional} under certain mild conditions on $(\Gg,\Ng)$; see for instance Corollary~\ref{cor=nilp}.

\section{The generalized mixed Bavard duality}\label{sec=gmBavard}

In this section, we prove the generalized mixed Bavard duality theorem, which is a new result in the present article. For the reader's convenience, here we restate Theorem~\ref{theorem=generalized_mixed_Bavard}.

\begin{theorem}[generalized mixed Bavard duality theorem]\label{theorem=generalized_mixed_Bavardagain}
Let $\Gg$ be a group and $\Ng$ its normal subgroup. For every chain $c \in \CR(\Gg,\Ng)$, we have
\[
\sclGN(c)=\sup_{\muf\in \QNG\setminus \HNRG}\frac{|\muf(c)|}{2\DD(\muf)}.
\]
\end{theorem}

In Subsection~\ref{subsec=Nqhom}, we define a notion of $\Ng$-quasimorphisms and its basic properties, which was one of the keys to proving the mixed Bavard duality (Theorem~\ref{theorem=mixed_Bavard}) in \cite{KKMM1}. In Subsection~\ref{subsection 3.1}, we introduce the mixed commutator length for  integral chains (Definition~\ref{definition=commutator_lengths_for_integral_chains}) and provide a  geometric interpretation of the mixed commutator lengths for integral chains (Theorem~\ref{theorem characterization 2}). In Subsection~\ref{subsection 3.2}, we introduce the stable mixed commutator lengths for rational chains and show some properties of it by using admissible $(\Gg,\Ng)$-simplicial surfaces (Definition~\ref{definition=admissible}). In Subsection~\ref{subsection 3.3} we introduce the filling norm $\fl_{G,N}$ for rational chains (\ref{definition=filling}) and prove the equality $4 \scl_{\Gg,\Ng} = \fl_{\Gg,\Ng}$ (Theorem~\ref{theorem=4scl=fill}).  Subsection~\ref{subsection 3.4} is devoted to the proof of the generalized mixed Bavard duality theorem  for rational chains (Theorem~\ref{generalized mixed bavard rational}). Based on this theorem, we introduce $\scl_{G,N}$ for real chains; finally, we obtain the generalized mixed Bavard duality theorem as Theorem~\ref{generalized mixed bavard}. As is mentioned in the introduction, the generalized mixed Bavard duality recovers the other three Bavard dualities treated in this article. In this manner, we present complete proofs of the four Bavard duality theorems. In this section, for a group $\Gg$ we use the symbol $B\Gg$ to denote the classifying space of $\Gg$.

\subsection{$N$-quasimorphisms}\label{subsec=Nqhom}
In this subsection, we review several notions and facts we need in the proof of the generalized mixed Bavard duality.  These notions (of $\Ng$-quasimorphisms and $\BBB'_1(\Gg,\Ng;\RR)$) were studied in \cite{KKMM1} for the proof of the mixed Bavard duality theorem.


\begin{definition}[{\cite[Section~2]{KKMM1}}] \label{definition_N-quasimorphism}
Let $\Gg$ be a group and $\Ng$ its normal subgroup.
\begin{enumerate}[(1)]
 \item A function $\psf \colon \Gg \to \RR$ is an \emph{$\Ng$-quasimorphism on $\Gg$} if there exists a non-negative number $\DD''$ such that
\[ |\psf(\gl \xl) - \psf(\gl) - \psf(\xl)| \le \DD'' \quad \textrm{and} \quad |\psf(\xl \gl) - \psf(\xl) - \psf(\gl)| \le \DD''\]
hold for every $\gl \in \Gg$ and for every $\xl \in \Ng$. Let $\DD''_{\Gg,\Ng}(\psf)$ denote the infimum of such a non-negative number $\DD''$, and call it the \emph{defect} of the $\Ng$-quasimorphism $\psf$. In other words,
\begin{align*}
&\DD''_{\Gg,\Ng}(\psf)\\
&=\sup\left\{\max\{|\psf(\gl \xl) - \psf(\gl) - \psf(\xl)|,\, |\psf(\xl \gl) - \psf(\xl) - \psf(\gl)|\}\;|\; \gl\in \Gg,\ \xl\in \Ng\right\}.
\end{align*}
 \item An \emph{$\Ng$-homomorphism on $\Gg$} is an $\Ng$-quasimorphism $\psf \colon \Gg \to \RR$ on $\Gg$ such that $\DD''_{\Gg,\Ng}(\psf) = 0$.
  \item Define $\rQQQ_N(G)$ as the $\RR$-linear space of $\Ng$-quasimorphisms on $\Gg$, and $\HHH^1_{\Ng}(\Gg;\RR)$ as the $\RR$-linear space of $\Ng$-homomorphisms on $\Gg$.
\end{enumerate}
\end{definition}

We note that the concept of $N$-quasimorphisms is introduced as a special case of the one of partial quasimorphism, which was intrinsically introduced in \cite{EP06} (see also \cite[Example ~11.4]{KKMMMsurvey}).

An $\Ng$-quasimorphism on $\Gg$ is closely related to a $\Gg$-quasi-invariant quasimorphism on $\Ng$. More precisely, the restriction of an $\Ng$-quasimorphism on $\Gg$  to $\Ng$  is a  $\Gg$-quasi-invariant quasimorphism  on $\Ng$ (Lemma~\ref{lemma=restriction} below); conversely, a $\Gg$-quasi-invariant quasimorphism on $\Ng$ can be extended to an $\Ng$-quasimorphism on $\Gg$ (Proposition~\ref{prop 3.4.1} below). Thus, a real-valued function on $\Ng$  is a $\Gg$-quasi-invariant quasimorphism if and only if this can be extended as an $\Ng$-quasimorphism on $\Gg$.

\begin{lemma}[{\cite[Lemma~2.3]{KKMM1}}]\label{lemma=restriction}
Let $\psf \colon \Gg \to \RR$ be an $\Ng$-quasimorphism on $\Gg$. Then the restriction $\psf|_{\Ng} \colon \Ng \to \RR$ is a $\Gg$-quasi-invariant quasimorphism on $\Ng$.
\end{lemma}
\begin{proof}
By the definition of $\Ng$-quasimorphisms, $\psf|_\Ng$ is a quasimorphism on $\Ng$ whose defect is at most $D''_{\Gg,\Ng}(\psf)$. In what follows, we show that this is $\Gg$-quasi-invariant. For $\gl \in \Gg$ and $\xl \in \Ng$, we have
\[ \psf(\gl \xl \gl^{-1}) + \psf(\gl ) \sim_{D''_{\Gg,\Ng}(\psf)} \psf(\gl \xl) \sim_{D''_{\Gg,\Ng}(\psf)} \psf(\gl ) + \psf(\xl).\]
This means that
\[ \psf(\gl \xl \gl^{-1}) - \psf(\xl) \sim_{2D''_{\Gg,\Ng}(\psf)} 0,\]
and hence $\psf$ is a $\Gg$-quasi-invariant quasimorphism on $\Ng$.
\end{proof}

In the following proposition, we use the symbol $\DD_{\Ng}$ for the defect of elements in $\widehat{\QQQ}(\Ng)$ to clarify that this defect is for quasimorphisms on $\Ng$.

\begin{proposition}[{\cite[Proposition~2.4]{KKMM1}}] \label{prop 3.4.1}
Let $\hat{\muf}$ be a $\Gg$-quasi-invariant quasimorphism on $\Ng$.
Then, there exists an $\Ng$-quasimorphism $\psf \colon \Ng \to \RR$ such that  $\psf |_\Ng = \hat{\muf}$ and that
\[
\DD_{\Gg, \Ng}''(\psf) \le \DD_{\Ng}(\hat{\muf}) + \DD_{\Gg, \Ng}'(\hat{\muf}).
\]
In particular, if $\hat{\muf}$ is moreover homogeneous, then there exists an $\Ng$-quasimorphism $\psf \colon \Ng \to \RR$ such that  $\psf |_\Ng = \hat{\muf}$ and that
\[
\DD''_{\Gg, \Ng}(\psf) = \DD_{\Ng}(\hat{\muf}).
\]
\end{proposition}
\begin{proof}
Set $\Gam=\Gg/\Ng$. Fix a (set-theoretical) section of the quotient map $\Gg\to \Gam$ sending $1_{\Gam}$ to $1_\Gg$, and let $S$ be its image.
Note that the map $S \times \Ng \to \Gg$, $(s, \xl) \mapsto s \xl$ is a bijection. Fix a function  $\varphi \colon S \to \RR$ satisfying that $\varphi(1_{\Gg})=0$. Then, define $\psf\colon \Gg\to \RR$ by
\[ \psf(\gl) = \varphi(s) + \hat{\muf}(\xl).\]
Here for $\gl\in \Gg$, take the unique pair $(s,\xl)\in S\times \Ng$ such that $s\xl=\gl$.

In what follows, we show that $\psf$ is an $\Ng$-quasimorphism such that $\DD''_{\Gg, \Ng}(\psf) \le \DD_{\Ng}(\hat{\muf}) + \DD'_{\Gg, \Ng}(\hat{\muf})$ and $\psf |_\Ng = \hat{\muf}$. By construction, we clearly have $\psf |_\Ng = \hat{\muf}$. Let $\gl \in \Gg$ and $\xl \in \Ng$. Let $s \in S$ and $\xl' \in \Ng$ with $s \xl' = \gl$. Then we have
\begin{align*}
\psf(\gl \xl) & = \psf(s \xl' \xl) = \varphi(s) + \hat{\muf}(\xl' \xl) \\
& \sim_{\DD_{\Ng}(\hat{\muf})} \varphi(s) + \hat{\muf}(\xl') + \hat{\muf}(\xl) \\
& = \psf(s\xl') + \psf(\xl) \\
& = \psf(\gl ) + \psf(\xl),
\end{align*}
and hence
\[|\psf(\gl \xl) - \psf(\gl ) - \psf(\xl)| \le \DD_{\Ng}(\hat{\muf}).\]
Next set $\xl'' = \gl^{-1} \xl \gl$. Then we have
\begin{align*}
\psf(\xl \gl) - \psf(\xl) - \psf(\gl ) & = \psf(\gl \xl'') - \psf(\gl \xl'' \gl^{-1}) - \psf(\gl ) \\
& \sim_{\DD'_{\Gg, \Ng}(\hat{\muf})} \psf(\gl \xl'') - \psf(\xl'') - \psf(\gl ) \\
& \sim_{\DD_{\Ng}(\hat{\muf})} 0,
\end{align*}
and hence
\[|\psf(\xl \gl) - \psf(\xl) - \psf(\gl )| \le \DD_{\Ng}(\hat{\muf}) + \DD'_{\Gg,\Ng}(\hat{\muf}).\]
Thus we have shown that $\psf$ is an $\Ng$-quasimorphism on $\Gg$ satisfying that
\[\DD''_{\Gg, \Ng}(\psf) \le \DD_{\Ng}(\hat{\muf}) + \DD'_{\Gg, \Ng}(\hat{\muf}).
\]

In particular, if $\hat{\muf}$ is homogeneous, then Lemma~\ref{lem=qinv_inv} implies that $\DD'_{\Gg, \Ng}(\hat{\muf}) = 0$. Hence we have $\DD''_{\Gg, \Ng}(\psf) \le \DD_{\Ng}(\hat{\muf})$. On the other hand, it is clear that $\DD_{\Ng}(\hat{\muf}) \le \DD''_{\Gg, \Ng}(\psf)$. This completes the proof.
\end{proof}

It is straightforward to show that the defect $\DD''_{\Gg,\Ng}$ is a seminorm on $\rQQQ_{\Ng}(\Gg)$, and that $\HHH^1_{\Ng}(\Gg;\RR)$ is the kernel of $\DD''_{\Gg,\Ng}$. Hence $\DD''_{\Gg,\Ng}$ induces a norm on $\rQQQ_{\Ng}(\Gg) / \HHH^1_{\Ng}(\Gg;\RR)$; we write $\DD''_{\Gg,\Ng}$ for this norm as well.

In fact, $(\rQQQ_{\Ng}(\Gg)/ \HHH^1_{\Ng}(\Gg;\RR), \DD''_{\Gg,\Ng})$ is not merely a normed space, but is isomorphic to the continuous dual of a certain normed space. To state it precisely, we need to introduce several notions.

Let $A$ be a unital commutative ring. Recall the definitions of $\CCC_2(\Gg;A)$, $\CCC_1(\Gg;A)$ and $\partial \colon \CCC_2(\Gg;A)\to \CCC_1(\Gg;A)$ from Section~\ref{sec=gBavard}.

\begin{definition}\label{defn=BBB'}
Let $\Gg$ be a group and $\Ng$ its normal subgroup.  Let $A$ be a unital commutative ring.
\begin{enumerate}[(1)]
 \item Let $\CCC'_2(\Gg,\Ng ; A)$ be the $A$-submodule of $\CCC_2(G; A)$ generated by the set
\[ \{ (\gl_1,\gl_2) \; | \; \textrm{$\gl_1$ or $\gl_2$ belongs to $\Ng$}\}.\]
  \item Define
\[
\BBB'_1(\Gg,\Ng;A) = \partial \CCC_2'(\Gg,\Ng ; A)
\]
and
\[
\ZZZ'_2(\Gg,\Ng;A) = \Ker \left(\partial |_{\CCC_2'(\Gg,\Ng;A)} \colon \CCC'_2(\Gg,\Ng;A) \to \CCC_1(\Gg;A)\right).\]
Here $\partial \colon \CCC_2(\Gg;A) \to \CCC_1(\Gg;A)$ denotes the boundary map.\item Regard $\CCC_2(\Gg ; \RR)$ as a normed space equipped with the $\ell^1$-norm, and consider $\BBB'_1(\Gg,\Ng; \RR)$ as the normed space $\CCC'_2(\Gg,\Ng; \RR) / \ZZZ_2'(\Gg,\Ng; \RR)$ with the quotient norm. Define $\| \cdot\|'$ as this norm on $\BBB_1'(\Gg,\Ng;\RR)$. In other words, for $c \in \BBB_1'(\Gg,\Ng; \RR)$, set
\begin{equation}\label{eq='norm}
 \| c\|' = \inf \{ \| c' \|_1 \; | \; c' \in \CCC_2'(\Gg,\Ng ; \RR),\ \partial c'=c \}.
\end{equation}
Here $\| \cdot \|_1$ denotes the $\ell^1$-norm.
\end{enumerate}
\end{definition}

For Definition~\ref{defn=BBB'}~(3), $\ZZZ_2'(\Gg,\Ng; \RR)$ is a closed subspace of $\CCC_2'(\Gg,\Ng; \RR)$. Indeed, if we regard $\partial$ as a operator $\partial\colon (\CCC_2(\Gg ; \RR),\|\cdot\|_1) \to (\CCC_1(\Gg ; \RR),\|\cdot\|_1)$, then its operator norm is at most $3$. Hence, $\ZZZ_2(\Gg; \RR)=\Ker(\partial)$ is closed in $(\CCC_2(\Gg ; \RR),\|\cdot\|_1)$, and  $\ZZZ_2'(\Gg,\Ng; \RR)$ is a closed in $(\CCC_2'(\Gg,\Ng; \RR), \|\cdot\|_1)$ as well. Thus, $\|\cdot\|'$ is a genuine norm on $\BBB_1'(\Gg,\Ng;\RR)$. We note that we equip $\BBB_1'(\Gg,\Ng;\RR)$ with the norm $\|\cdot\|'$ defined by \eqref{eq='norm}, not the norm induced by the $\ell^1$-norm on $\CCC_1(\Gg ; \RR)$. We will
use the following assertion on the norm $\|\cdot\|'$ in the proof of Proposition~\ref{proposition trivial}.

\begin{lemma}[{\cite[Proposition~3.3]{KKMM1}}]\label{lem='norm_cl}
\begin{enumerate}[$(1)$]
\item For $\gl \in \Gg$ and $\xl \in \Ng$, we have $[\gl, \xl] \in \BBB'_1(\Gg,\Ng;\RR)$ and $\| [\gl, \xl ]\|' \le 3$.

\item Let $\yl_1, \yl_2 \in [\Gg, \Ng]$, and assume that $\yl_1, \yl_2 \in \BBB'_1(\Gg,\Ng;\RR)$. Then $\yl_1 \yl_2 \in \BBB'_1(\Gg,\Ng;\RR)$ and $\| \yl_1 \yl_2\|' \le \| \yl_1 \|' + \| \yl_2\|' + 1$.
\end{enumerate}
In particular, for every $\yl \in [\Gg, \Ng]$, we have $\yl \in \BBB'_1(\Gg,\Ng;\RR)$ and $\| \yl\|' \leq  4 \cl_{\Gg, \Ng}(\yl) - 1$.
\end{lemma}
\begin{proof}
Item (1) immediately follows from
\[ \partial ([\gl,\xl], \xl \gl) - \partial (\gl, \xl) + \partial (\xl, \gl) = [\xl, \gl].\]
Next we show item (2). Since $\yl_1, \yl_2,$ and $\yl_1 \yl_2 - \yl_1 - \yl_2 = \partial (\yl_1, \yl_2) \in \BBB'_1(\Gg,\Ng;\RR)$, we have $\yl_1 \yl_2 \in \BBB'_1(\Gg,\Ng;\RR)$. The inequality clearly follows from the estimates
\begin{align*}
\| \yl_1 \yl_2\|' & = \| (\yl_1 + \yl_2) + (\yl_1 \yl_2 - \yl_1 - \yl_2) \|' \\
& \le \| \yl_1 + \yl_2\|' + \| \yl_1 \yl_2 - \yl_1 - \yl_2 \|' \\
& \le \| \yl_1\|' + \| \yl_2 \|' + 1.
\end{align*}
Hence item (2) is proved.
The final assertions follow from items (1) and (2).
\end{proof}

Then the following holds.

\begin{proposition}[{\cite[Proposition~3.5]{KKMM1}}]\label{proposition=isometry}
The normed space $(\rQQQ_{\Ng}(\Gg) / \HHH^1_{\Ng}(\Gg;\RR), \DD''_{\Gg,\Ng})$ is isometrically isomorphic to the continuous dual of the space $(\BBB_1'(\Gg,\Ng; \RR), \| \cdot\|')$, and the isomorphism is induced by the  pairing
\begin{equation}\label{eq=B_1pairing}
(\rQQQ_{\Ng}(\Gg) / \HHH^1_{\Ng}(\Gg;\RR))\times \BBB_1'(\Gg,\Ng; \RR)\to \RR;\quad ([\psf],c)\mapsto \psf(c).
\end{equation}
\end{proposition}
\begin{proof}
In this proof, we abbreviate $\CCC_2'(\Gg ;\RR)$ as $\CCC_2'$, $\ZZZ'_2(\Gg,\Ng; \RR)$ as $\ZZZ'_2$, and $\BBB_1'(\Gg,\Ng; \RR)$ as $\BBB_1'$. We also abbreviate $\QGN$ as $\rQQQ_{\Ng}$ and $\HHH^1_{\Ng}(\Gg;\RR)$ as $\HHH^1_{\Ng}$, respectively.  Recall that $\CCC_2' / \ZZZ_2'$ and $\BBB_1'$ are isometrically isomorphic as normed spaces via the isomorphism $\bar{\partial} \colon \CCC_2' / \ZZZ_2' \to \BBB_1'$ induced by the boundary maps $\partial \colon \CCC_2' \to \BBB_1'$ (see Definition~\ref{definition_N-quasimorphism}).  In what follows, we  first construct a surjective  linear isometry $\Phi$ from $\rQQQ_{\Ng} / \HHH^1_\Ng$ to $(\CCC_2' / \ZZZ_2')^{\ast}$;  namely,  this $\Phi$ will give the isometric isomorphism $\rQQQ_{\Ng} / \HHH^1_\Ng\cong (\CCC_2' / \ZZZ_2')^{\ast}$. Here, we consider the norm of $(\CCC_2' / \ZZZ_2')^{\ast}$ as the dual norm  $\|\cdot\|_*$  of $\CCC_2' / \ZZZ_2'$.  That is,  for $f \in (\CCC_2' / \ZZZ_2')^{\ast}$, we set 
\[
\| f\|_* =  \inf\{ a \ge 0 \mid \text{$|f(x)| \le a \| [x]\|_{\CCC_2' / \ZZZ_2'}$ for every $x \in \CCC_2'$}\}. 
\]
Here, and $\| \cdot \|_{\CCC_2' / \ZZZ_2'}$ denotes the quotient norm. Let $\pi$ denote the quotient map from $\CCC_2'$ to $\CCC_2' / \ZZZ_2'$. Since
\[
\| f\|_*= \inf \{ a \ge 0 \mid \text{$|f \circ \pi(x)| \le a \| x\|_1$ for every $x \in \CCC_2'$}\},
\]
we have
\begin{equation}\label{eq=Rudin}
\|f\|_*= \| f \circ \pi\|_*
\end{equation}
(see \cite[Theorem~4.9(b)]{Rudin} for more details). Here $\|\cdot\|_*$ in the right-hand side of \eqref{eq=Rudin} denotes the dual norm on $(\CCC_2',\|\cdot\|_1)^{\ast}$.

Let $\psf \colon \Gg \to \RR$ be an $\Ng$-quasimorphism on $\Gg$. Then $\psf$ is identified with an $\RR$-linear map $\psf \colon \CCC_1(G ; \RR) \to \RR$. Since $\psf$ is an $\Ng$-quasimorphism, we have
\[|\psf \circ \partial (\gl ,\xl)| = |\psf(\xl) - \psf(\gl \xl) + \psf(\gl )| \le D''_{\Gg,\Ng}(\psf),\]
and
\[|\psf \circ \partial (\xl,\gl)| = |\psf(\gl ) - \psf(\xl \gl) + \psf(\xl)| \le D''_{\Gg,\Ng}(\psf)\]
for every $\gl \in \Gg$ and every $\xl \in \Ng$. 
 This implies that $\psf \circ \partial \in (\CCC_2',\|\cdot\|_1)^*$; more precisely, we have
\begin{equation}\label{eq=tocho_ue}
\|\psf \circ \partial\|_{*}\leq D''_{\Gg, \Ng}(\psf).
\end{equation}
Since $\psf \circ \partial$ vanishes on $\ZZZ_2'$, we have that $\psf \circ \partial$ induces a continuous linear map $\hat{\psf} \colon \CCC_2' / \ZZZ_2' \to \RR$, whose dual norm  is at most $D''_{\Gg,\Ng}(\psf)$. Clearly, this correspondence $\psf \mapsto \hat{\psf}$ is linear. Since the kernel of this correspondence $\rQQQ_\Ng(\Gg) \to (\CCC_2' / \ZZZ_2')^*$, $\psf \mapsto \hat{\psf}$ contains $\HHH^1_\Ng(\Gg)$, this correspondence induces a linear map $\Phi \colon \rQQQ_\Ng(\Gg)/ \HHH^1_\Ng(\Gg) \to (\CCC_2' / \ZZZ_2')^*$.

We  claim that $\Phi \colon \rQQQ_\Ng(\Gg)/ \HHH^1_\Ng(\Gg) \to (\CCC_2' / \ZZZ_2')^*$ is a (not necessarily surjective) linear isometry. Indeed, we  will show the following  for every $\psf\in \rQQQ_\Ng(\Gg)$:
\begin{enumerate}[(a)]
\item $\| \Phi([\psf])\|_* \le D''_{\Gg, \Ng}(\psf)$.

\item $\| \Phi([\psf])\|_* \ge D''_{\Gg, \Ng}(\psf)$.
\end{enumerate}
 For (a), combine \eqref{eq=tocho_ue} with \eqref{eq=Rudin}. To see (b), by the definition of $D''_{\Gg,\Ng}$ (see Definition~\ref{definition_N-quasimorphism})  and \eqref{eq=Rudin}, we have
\begin{align*}
D''_{\Gg,\Ng}(\psf) & = \sup \{ |\psf(\gl \xl) - \psf(g) - \psf(x)|,\ |\psf(\xl \gl) - \psf(x) - \psf(\gl)| \mid \gl \in \Gg, \xl \in \Ng \} \\
& = \sup \{ |\psf \circ \partial (\gl, \xl)|,\ |\psf \circ \partial (\xl, \gl)| \mid \gl \in \Gg, \xl \in \Ng \} \\
& \le \| \psi \circ \partial \|_* \cdot 1 = \| \hat{\psf} \|_* = \| \Phi([\psf])\|_*.
\end{align*}
This completes the proof of (b), and we have completed the proof that $\Phi \colon \rQQQ_\Ng / \HHH^1_{\Ng} \to (\CCC_2' / \ZZZ_2')^{\ast}$ is a (not necessarily surjective) linear isometry.

To see the surjectivity of $\Phi$, we construct the  inverse correspondence $ \Psi \colon (\CCC_2' / \ZZZ_2')^{\ast} \to \rQQQ_\Ng / \HHH^1_\Ng$  of $\Phi$. Let $\psf \colon \CCC_2' / \ZZZ_2' \to \RR$ be a bounded operator. Then $\psf$ is identified with a linear map $\overline{\psf} \colon \BBB_1' \to \RR$. Since our coefficient ring is $\RR$ and this ring is injective, there exists a linear map $\psf' \colon \CCC_1(\Gg;\RR) \to \RR$ such that $\psf'|_{\BBB_1'} = \overline{\psf}$, which is not necessarily continuous. The function $\psf' \colon \CCC_1 (\Gg;\RR) \to \RR$ is identified with a real-valued function $\psf' \colon \Gg \to \RR$ on $\Gg$. Let $D''$ be the dual norm of $f \colon \CCC_2' \to \RR$. Then we have
\[| \psf'(\gl \xl) - \psf'(\gl) - \psf'(\xl)| = | \psf'(\partial (\gl ,\xl))| = | \psf(\gl ,\xl)| \le D''.\]
Similarly, we can show $|\psf'(\xl \gl) - \psf'(\xl) - \psf'(\gl )| \le D''$. This implies that $\psf'$ is an $\Ng$-quasimorphism.

To complete the construction of the correspondence $\Psi \colon (\CCC_2' / \ZZZ_2')^{\ast} \to \rQQQ_\Ng / \HHH^1_\Ng$,  we let $f''$ be another linear extension of $\overline{\psf}\colon \BBB'_1 \to \RR$ and show that $\psf' - \psf''$ is an $\Ng$-homomorphism. Since $\psf'$ and $\psf''$ coincide on $\BBB_1'$, we have
\[\psf'(\gl \xl) -\psf'(\gl ) - \psf'(\xl) = \psf'(\gl \xl - \gl - \xl) = \psf''(\gl \xl - \gl - \xl) = \psf''(\gl \xl) - \psf''(\gl ) - \psf''(\xl).\]
This means that
\[(\psf' - \psf'')(\gl \xl) = (\psf' - \psf'')(\gl ) + (\psf' - \psf'')(\xl).\]
Similarly, we can show that
\[(\psf' - \psf'')(\xl \gl) = (\psf' - \psf'')(\xl) + (\psf' - \psf'')(\gl ),\]
and hence $\psf' - \psf''$ is an $\Ng$-homomorphism. Thus we have completed the construction of the correspondence $ \Psi \colon (\CCC_2' / \ZZZ_2')^{\ast} \to \rQQQ_\Ng / \HHH^1_\Ng$, $\psi \mapsto [\psi']$.
Since these correspondences  $\Phi$ and $\Psi$ are mutually inverses, we conclude that $\Phi \colon  \rQQQ_\Ng / \HHH^1_\Ng \to (\CCC_2'/ \ZZZ_2')^{\ast}$ is a surjective  linear isometry.  Therefore, $\rQQQ_\Ng / \HHH^1_\Ng$ and $(\CCC_2'/ \ZZZ_2')^{\ast}$ are isometrically isomorphic, and so are $\rQQQ_\Ng / \HHH^1_\Ng$ and $(\BBB_1')^{\ast}$.

 Finally, we will show that the isometric isomorphism $\rQQQ_\Ng / \HHH^1_\Ng\cong (\BBB_1')^{\ast}$ is given by the pairing \eqref{eq=B_1pairing}. Let $\bar{\partial} \colon \CCC_2' / \ZZZ_2' \to \BBB_1'$ be the isomorphism induced by the boundary map $\partial \colon \CCC_2' \to \BBB_1'$. Then the composite $(\bar{\partial}^{-1})^* \circ \Phi \colon \rQQQ_\Ng / \HHH^1_\Ng \to (\BBB_1')^*$ is a surjective  linear isometry. Then the pairing
\[ \rQQQ_\Ng /\HHH^1_\Ng \times \BBB_1' \to \RR, \quad ([\psf], c) \mapsto (\bar{\partial}^{-1})^* \circ \Phi([\psf]) (c)\]
coincides with the pairing  \eqref{eq=B_1pairing}: $([\psf], c) \mapsto \psf(c)$. Indeed, for $c' \in \CCC_2'$ such that $\bar{\partial} ([c']) = c$, we have
\[ (\bar{\partial}^{-1})^* \circ \Phi([\psf])(c) = (\bar{\partial}^{-1})^*(\Phi([\psf])) (\bar{\partial}([c'])) = \Phi([\psf])([c']) = \psi \circ \partial (c') = \psf(c).\]
This completes the proof.
\end{proof}

Recall the following consequence of the Hahn--Banach theorem in functional analysis.
\begin{theorem}[for instance, see \cite{Rudin}]\label{thm=HB}
Let $(X,\|\cdot\|_X)$ be a real normed space, and let $\|\cdot\|_{X^\ast}$ be the dual norm on the  continuous dual $X^\ast$ of $X$. Let $X^{\ast}\times X\to \RR$; $(\varphi,v)\mapsto \varphi(v)$ be the standard pairing. Then, for every $v\in X$ we have
\[
\|v\|_X=\sup_{\varphi\in X^{\ast}\setminus \{0\}}\frac{|\varphi(v)|}{\|\varphi\|_{X^\ast}}.
\]
\end{theorem}

\begin{corollary}\label{cor=HB}
Let $\Gg$ be a group and $\Ng$ its normal subgroup. Then, for every $c\in \BBB_1'(\Gg,\Ng; \RR)$, we have
\[
\|c\|'=\sup_{\psf\in \rQQQ_{\Ng}(\Gg) \setminus \HHH^1_{\Ng}(\Gg;\RR) } \frac{|\psf(c)|}{\DD''_{\Gg,\Ng}(\psf)}.
\]
\end{corollary}

\begin{proof}
By combining Proposition~\ref{proposition=isometry} with Theorem~\ref{thm=HB}, we have
\[
\|c\|'=\sup_{[\psf]\in (\rQQQ_{\Ng}(\Gg) / \HHH^1_{\Ng}(\Gg;\RR)) \setminus \{0\}} \frac{|\psf(c)|}{\DD''_{\Gg,\Ng}([\psf])}
\]
for every $c\in \BBB_1'(\Gg,\Ng;\RR)$. Now, we obtain the desired equality.
\end{proof}

\subsection{$(G,N)$-simplicial surfaces} \label{subsection 3.1}

In this subsection, we introduce the mixed commutator lengths for chains (Definition~\ref{definition=commutator_lengths_for_integral_chains}) and provide a geometric interpretation of it (Theorem~\ref{theorem characterization 2})  in terms of $(G,N)$-simplicial surfaces.

Recall from the notation and conventions in the introduction that a surface means a (not necessarily connected) compact orientable 2-dimensional manifold in the present article. A \emph{triangulation} of a topological space $X$ means a $\Delta$-complex structure on $X$ (see \cite[Section~2.1]{Hatcher}). A \emph{simplicial surface} is a triangulated surface. Hence every edge ($1$-simplex) is oriented and every triangle ($2$-simplex) is surrounded by three oriented edges as is depicted in Figure~\ref{figure 1}. For a simplicial surface $S$, we write $S_i$ to indicate the set of $i$-simplices of $S$. A \emph{$\Gg$-labelling of a simplicial surface $S$} is a map $f \colon S_1 \to \Gg$ which satisfies
\begin{align} \label{equality_1}
f(\partial_2 \sigma) \cdot f(\partial_0 \sigma) = f(\partial_1 \sigma)
\end{align}
for every $2$-simplex $\sigma$ of $S$.
Here we present the definition of $(\Gg,\Ng)$-simplicial surfaces.

\begin{definition}[$(\Gg,\Ng)$-simplicial surfaces]\label{defn=GN_simpl}
Let $\Gg$ be a group and $\Ng$ its normal subgroup.
\begin{enumerate}[(1)]
 \item A $\Gg$-labelling $f$ of a simplicial surface $S$ is called a \emph{$(\Gg,\Ng)$-labelling} if $f$ satisfies the following condition:  for every $2$-simplex $\sigma$ of $S$, either $f(\partial_0 \sigma)$ or $f(\partial_2 \sigma)$ belongs to $\Ng$.
  \item A \emph{$(\Gg,\Ng)$-simplicial surface} is defined as a pair $(S, f)$ of a simplicial surface $S$ and a $(\Gg,\Ng)$-labelling $f$ of $S$.
\end{enumerate}
\end{definition}

\begin{figure}[t]
\centering
\includegraphics[width=5cm]{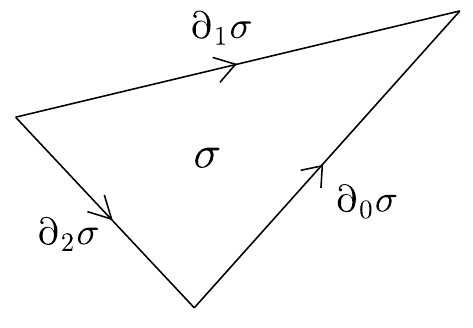}
\caption{$2$-simplex $\sigma$} \label{figure 1}
\end{figure}

\begin{remark} \label{remark continuous map}
Recall from the beginning of this section that $B\Gg$ denotes the classifying space of $\Gg$.
Let $S$ be a simplicial surface and $f$ a $\Gg$-labelling of $S$.
Then there exists a continuous map $\hat{f} \colon S \to B\Gg$ which satisfies the following: first, $\hat{f}$ sends every vertex of $S$ to the basepoint of $B\Gg$. Then, $\hat{f}$ sends an edge $e$ of $S$ to a loop of $B\Gg$, and its homotopy class  with fixed endpoints coincides with $f(e) \in \Gg = \pi_1(B\Gg)$. Then \eqref{equality_1} implies that $\hat{f}$ can be extended to the whole surface $S$, which is unique up to free homotopy.

Assume that $B\Ng$ is a subspace of $B\Gg$ and that the map $\pi_1(B\Ng) \to \pi_1(B\Gg)$ induced by the inclusion $B\Ng \hookrightarrow B\Gg$ coincides with the inclusion $\Ng \hookrightarrow \Gg$. Then we can take $\hat{f}$ such that the following two statements hold true.
\begin{enumerate}[(a)]
\item $\hat{f}(e) \subset B\Ng$ for every edge $e$ labelled by an element of $\Ng$.

\item $\hat{f}(\sigma) \subset B\Ng$ for every $2$-simplex $\sigma$ surrounded by edges labelled by elements of $\Ng$.
\end{enumerate}
\end{remark}

In the paper \cite{KKMM1}, Kawasaki, Kimura, Matsushita and Mimura showed that $(\Gg,\Ng)$-simplicial surfaces and $(\Gg,\Ng)$-mixed commutator lengths are related as follows.

\begin{theorem}[{\cite[Proposition~4.7 and Proposition~4.8]{KKMM1}}] \label{theorem characterization 1}
Let $\yl \in N$ and $k\in \mathbb{Z}_{\geq 0}$. Then the following are equivalent:
\begin{enumerate}
\item[\textup{(i)}] $\yl$ is a product of $k$ mixed commutators.

\item[\textup{(ii)}] There exists a $(\Gg,\Ng)$-simplicial surface of genus $k$ whose boundary component is $\yl$.
\end{enumerate}
\end{theorem}

We will prove a generalization (Theorem~\ref{theorem characterization 2}) of Theorem~\ref{theorem characterization 1}. However, in the proof of Theorem~\ref{theorem characterization 2}, we will employ the fact that (i) implies (ii) in Theorem~\ref{theorem characterization 1}. For the self-containedness of the present article, here we provide the  proof of this one-way implication.

\begin{figure}[t]
\begin{picture}(340,130)(10,0)

  \put(140,30){\vector(-1,0){120}}
  \put(140,30){\vector(-1,1){60}}
  \put(20,30){\vector(1,1){60}}

  \put(125,75){\vector(-3,1){43}}
  \put(125,75){\line(-3,1){45}}
  \put(140,30){\vector(-1,3){15}}

\put(35,75){\vector(3,1){43}}
\put(35,75){\line(3,1){45}}
\put(20,30){\vector(1,3){15}}

\put(20,30){\line(-1,-3){5}}
\put(145,15){\vector(-1,3){5}}

\put(68,18){\tiny $[\gl_i, \xl_i]$}

\put(16,55){\tiny $\xl_i$}
\put(51,89){\tiny $\gl_i$}

\put(100,89){\tiny $\xl_i$}
\put(137,55){\tiny $\gl_i$}

\put(51,52){\tiny $\xl_i \gl_i$}
\put(93,52){\tiny $\gl_i \xl_i$}

\put(300,0){\vector(1,1){40}}
\put(300,0){\vector(-1,0){40}}
\put(300,0){\vector(1,2){40}}

\put(300,0){\vector(-1,3){40}}
\put(300,0){\vector(-1,1){80}}
\put(300,0){\vector(-2,1){80}}

\put(340,40){\vector(0,1){38}}
\put(340,40){\line(0,1){40}}
\put(340,80){\line(-3,4){20}}

\put(260,120){\vector(-1,-1){40}}
\put(220,80){\vector(0,-1){40}}
\put(220,40){\vector(1,-1){40}}

\put(323,13){\tiny $[\gl_1, \xl_1]$}
\put(210,13){\tiny $[\gl_k, \xl_k]$}
\put(170,58){\tiny $[\gl_{k-1}, \xl_{k-1}]$}

\put(343,55){\tiny $[\gl_2,\xl_2]$}

\put(278,-10){\footnotesize $\yl$}

\put(265,115){$\cdots$}
\end{picture}
\caption{$\yl$ to $[\gl_1, \xl_1] \cdots [\gl_k, \xl_k]$} \label{figure bavard}
\end{figure}

\begin{proof}[Deduction of \textup{(ii)} from \textup{(i)} in Theorem~\textup{\ref{theorem characterization 1}}]
Let $\yl \in [\Gg, \Ng]$. Suppose that there exist $\gl_1, \cdots, \gl_k \in \Gg$ and $\xl_1, \cdots, \xl_k \in \Ng$ such that
\[ \yl = [\gl_1, \xl_1] \cdots [\gl_k, \xl_k].\]
We first consider a triangulated $(k+1)$-gon with $(\Gg,\Ng)$-labelling depicted in the right of Figure \ref{figure bavard}. Next, we attach three triangles depicted in the left of Figure~\ref{figure bavard} to each edge labelled by $[\gl_i, \hl_i]$. Finally, identify the edges labelled by the same symbol. Then we have a $(\Gg,\Ng)$-simplicial surface whose boundary component is labelled by $\yl$ and whose genus is $k$. Therefore, we have proved that (i) implies (ii) in Theorem~\ref{theorem characterization 1}.
\end{proof}

\begin{remark}
In \cite[Proposition~4.7]{KKMM1}, the authors showed that (i) implies  (ii) for the case $k = \cl_{G,N}(y)$. However, the proof is valid for every  $k\in \ZZ$ such that $k \ge \cl_{\Gg,\Ng}(\yl)$.
\end{remark}

We define the mixed commutator length $\cl_{\Gg,\Ng}$ for integral chains as follows.

\begin{definition} \label{definition=commutator_lengths_for_integral_chains}
Let $\Gg$ be a group and $\Ng$ its normal subgroup. Let $m\in \NN$ and $\xl_1, \cdots, \xl_m \in \Ng$. Assume that $\xl_1 \cdots \xl_m \in [\Gg,\Ng]$. We define the number $\cl_{\Gg,\Ng}(\xl_1 + \cdots + \xl_m)$ by
\[ \cl_{\Gg,\Ng}(\xl_1 + \cdots + \xl_m) = \inf_{\gl_1, \cdots, \gl_{m-1} \in \Gg} \cl_{\Gg,\Ng}(\xl_1 \gl_1 \xl_2 \gl_1^{-1} \cdots \gl_{m-1} \xl_m \gl_{m-1}^{-1}).\]
In other words,
\[
\cl_{\Gg,\Ng}(\xl_1 + \cdots + \xl_m) = \inf\cl_{\Gg,\Ng}(\xi_1 \xi_2 \cdots \xi_m),
\]
where in the infimum in the right-hand side of the equality, for every $i\in \{1,\ldots,m\}$, $\xi_i$ runs over the $\Gg$-conjugacy class of $\xl_i$ (here, recall that $\cl_{\Gg,\Ng}\colon [\Gg,\Ng]\to \RR$ is $\Gg$-invariant).
\end{definition}

Using $(\Gg,\Ng)$-simplicial surfaces, we obtain another formulation of $(\Gg,\Ng)$-commutator lengths for chains.

\begin{definition} \label{definition=GN_surface}
Let $m\in \NN$. Let $\xl_1, \cdots, \xl_m \in \Ng$. We define a \emph{$(\Gg,\Ng)$-simplicial surface whose boundary components are $\xl_1, \cdots, \xl_m$} to be a connected $(\Gg,\Ng)$-simplicial surface $(S,f)$ which satisfies the following three properties.
\begin{enumerate}[(1)]
\item $S$ has $m$ boundary components.

\item Every boundary component of $S$ consists of one edge and one vertex.

\item There exists a linear order $C_1, \cdots, C_m$ of the set of boundary components of $S$ such that $f(e_i) = x_i$. Here $e_i$ is the edge of $C_i$.
\end{enumerate}
\end{definition}

\begin{remark}\label{rem=productGN}
We note that in Definition \ref{definition=GN_surface}, we do not assume $\xl_1 \cdots \xl_m \in [\Gg,\Ng]$. In fact, in the next theorem (Theorem~\ref{theorem characterization 2}), we show that if there exists a $(\Gg,\Ng)$-simplicial surface whose boundary components are $\xl_1, \cdots, \xl_m$, then the product $\xl_1 \cdots \xl_m$ belongs to $[\Gg,\Ng]$. We also note that the validity of the assumption $\xl_1\cdots \xl_m\in [\Gg,\Ng]$ does not depend on the order of $\xl_1, \cdots, \xl_m\in \Ng$. Indeed, the group quotient $\Ng/[\Gg,\Ng]$ is abelian.
\end{remark}

\begin{theorem}[geometric interpretation of the mixed commutator lengths of integral chains] \label{theorem characterization 2}
Let $m\in \NN$ and  $\xl_1, \cdots, \xl_m \in \Ng$. Let $k\in \ZZ_{\geq 0}$. Then, the following are equivalent.
\begin{enumerate}
\item[\textup{(i)}] There exist $\gl_1, \cdots, \gl_{m-1} \in \Gg$ such that
\[ \xl_1 \gl_1 \xl_2 \gl_1^{-1} \cdots \gl_{m-1} \xl_m \gl_{m-1}^{-1} \]
is a product ot $k$ mixed commutators.

\item[\textup{(ii)}] There exists a $(G,N)$-simplicial surface of genus $k$ whose boundary components are labelled by $x_1, \cdots, x_m$.
\end{enumerate}
In particular, $\cl_{G,N}(x_1 + \cdots + x_m)$ coincides with the minimum genus of $(G,N)$-simplicial surface whose boundary components are $x_1, \cdots, x_m$.
\end{theorem}

The idea of the proof of `(ii) implies (i)' in Theorem~\ref{theorem characterization 2} is similar to that of \cite[Proposition~4.8]{KKMM1}.

\begin{proof}[Proof of Theorem~\textup{\ref{theorem characterization 2}}]
\begin{figure}[t]
\begin{minipage}[b]{0.6\linewidth}
    \centering
\includegraphics[width=6.5cm]{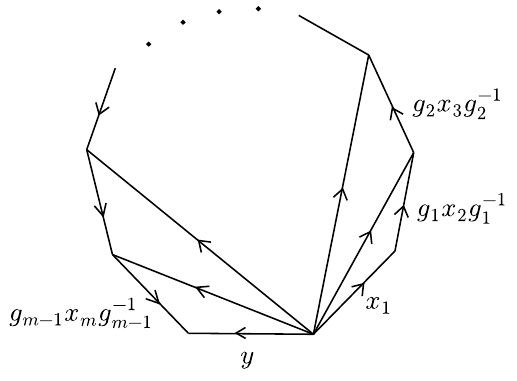}
  \end{minipage}
  \begin{minipage}[b]{0.3\linewidth}
    \centering
\includegraphics[width=3.5cm]{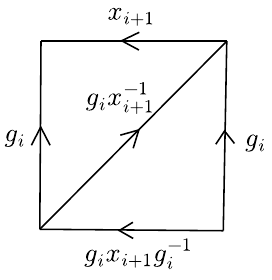}
  \end{minipage}
  \caption{$\yl$ to $\xl_1 \cdots \xl_m$} \label{figure=3}
\end{figure}

Assume (i). Then there exist $\gl_1, \cdots, \gl_{m-1} \in \Gg$ such that
\[ \yl = \xl_1 \gl_1 \xl_2 \gl_1^{-1} \cdots \gl_{m-1} \xl_m \gl_{m-1}^{-1}\]
is a product of $k$ mixed commutators.  Recall that we have already proved that (i) implies  (ii) in Theorem~\ref{theorem characterization 1}. Therefore, there exists a $(\Gg,\Ng)$-simplicial surface whose boundary component is $\yl$ and whose genus is $k$. By attaching triangles as is depicted in Figure~\ref{figure=3}, we obtain a $(\Gg,\Ng)$-simplicial surface whose boundary components are $\xl_1, \cdots, \xl_m$. Now we have shown that (i) implies (ii) in Theorem~\ref{theorem characterization 2}.

Next, assume (ii). Then there exists a $(\Gg,\Ng)$-simplicial surface $S$ of genus $k$ whose boundary components are $\xl_1, \cdots, \xl_m$. We first define the subcomplex $S'$ of $S$ as follows: the set of $0$-simplices of $S'$ coincides with the set of $0$-simplices of $S$. A $1$-simplex $e$ of $S$ belongs to $S'$ if and only if $f(e) \in \Ng$. A $2$-simplex $\sigma$ of $S$ belongs to $S'$ if and only if $f(\partial_i \sigma) \in \Ng$ for $i \in \{0,1,2\}$. Let $\hat{f} \colon S \to B\Gg$ be a continuous map described in Remark~\ref{remark continuous map}, and assume that $\hat{f}$ satisfies conditions (a) and (b) mentioned there. Then $\hat{f} |_{S'} \colon S'\to B\Gg$ can be factored through $B\Ng$.

First, we would like to find disjoint simple closed curves $\gamma_1, \cdots, \gamma_k$ which satisfy the following:
\begin{enumerate}[(1)]
\item $\hat{f}(\gamma_1), \cdots, \hat{f}(\gamma_k)$ are freely homotopic to a loop of $B\Ng$.

\item The surface cut along the loops $\gamma_1, \cdots, \gamma_k$ is connected and of genus $k$.
\end{enumerate}

Now we start to find such simple closed curves. Let $\sigma$ be a triangle of $S$ such that some edge of $\sigma$ is not labelled by an element of $\Ng$. By the definition of $(\Gg,\Ng)$-labelling, the number of edges of $\sigma$ which is not labelled by elements of $\Ng$ is exactly two. Let $l_\sigma$ be the line segment connecting the midpoints of these two edges in $\sigma$. Let $C$ be the union of $l_\sigma$ for all $\sigma$ having an edge not labelled by an element of $N$. Since the boundary of $S$ is labelled by elements of $\Ng$, $C$ is a compact $1$-dimensional manifold without boundary. Thus $C$ is a disjoint union of simple closed curves in $S$.

Here we prove the following two properties of $C$.
\begin{enumerate}[(A)]
\item For each connected component $\gamma$ of $C$, the inclusion $\gamma \hookrightarrow S$ is homotopic to a map factored through $S'$.

\item $S'$ is a deformation retract of $S \setminus C$.
\end{enumerate}
We now show (A). Let $\gamma$ be a connected component of $C$. Since $\gamma$ is a simple closed curve in the  (orientable) surface $S$, the normal bundle of $\gamma$ is trivial and hence we can slide $\gamma$ to $S \setminus C$. Thus (A) follows from (B). To see (B), for every triangle $\sigma$ having an edge not labelled by an element of $N$, the $\sigma \setminus l_\sigma$ can be deformed to the ends at the same speed. See Figure~\ref{figure homotopy} for this homotopy.

\begin{figure}[t]
\centering
\includegraphics[width=6cm]{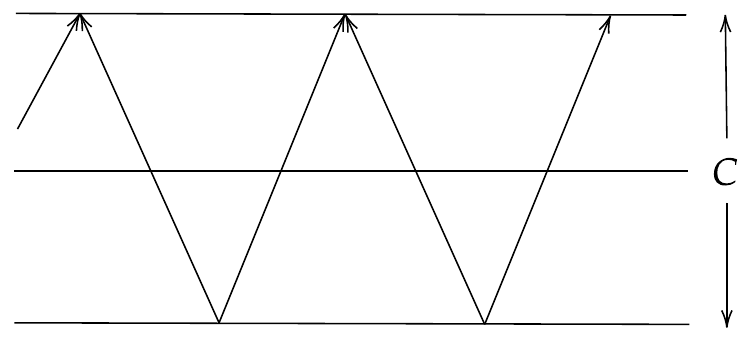}
\caption{a deformation retract of $S \setminus C$ to $S'$} \label{figure homotopy}
\end{figure}

If there exists a component of $C$ which is a non-separating curve of $S$, then we let $\gamma_1$ denote the component. If there exists a component of $C \setminus \gamma_1$ which is a non-separating curve of $S \setminus \gamma_1$, then we let $\gamma_2$ be such a component. Iterating this, we have a sequence of simple closed curve of $\gamma_1, \cdots, \gamma_j$ which has the following:
\begin{enumerate}[(1)]
\item $\gamma_1, \cdots, \gamma_j$ are components of $C$.

\item For $i \in \{1, \cdots, j\}$, $\gamma_i$ is a non-separating curve in $S \setminus (\gamma_1 \cup \cdots \cup \gamma_{i-1})$.

\item There does not exist a component of $C \setminus (\gamma_1 \cup \cdots \cup \gamma_j)$ which is a non-separating curve in $S \setminus (\gamma_1 \cup \cdots \cup \gamma_j)$.
\end{enumerate}

Since the sum of genera of components of $S \setminus C$ is $k - j$, there exist simple closed curves $\gamma_{j+1}, \cdots, \gamma_k$ which satisfy the following:
\begin{enumerate}[(1)]
\item For $i \in \{j+1, \cdots, k\}$, the curve $\gamma_i$ does not intersect $C$.

\item For $i \in \{j+1, \cdots, k\}$, the curve $\gamma_i$ is a non-separating simple closed curve in $S \setminus (\gamma_1 \cup \cdots \cup \gamma_{i-1})$.
\end{enumerate}

Now we have a sequence of simple closed curves $\gamma_1, \cdots, \gamma_k$ such that $S \setminus (\gamma_1 \cup \cdots \cup \gamma_k)$ is a compact connected surface of genus $0$. By the classification theorem of compact surfaces, there exists a homeomorphism from $S$ to a surface depicted in Figure~\ref{figure key}. Let $\delta_1, \cdots, \delta_k$, $\alpha_1, \cdots, \alpha_k$, $\tau_1, \cdots, \tau_{m-1}$ be the paths depicted in Figure~\ref{figure key}. Here we assume that the intersection of $\delta_i$ and $\gamma_i$ is the initial point of the loop $\gamma_i$. Then set
\[ \vl_i = [f \circ (\alpha_i \cdot \delta_i \cdot \bar{\alpha}_i)], \quad y_i = [f \circ (\alpha_i \cdot \gamma_i \cdot \bar{\alpha}_i)], \quad \gl_i = [f \circ \tau_i].\]
Here we consider the basepoint of $S$ is the initial point of $\alpha_i$ and $[-]$ denotes the pointed homotopy classes of loops in $BG$ (note that $\pi_1(BG) \cong G$). Then we have
\[ \xl_1 \cdot \gl_1 \xl_2 \gl_1^{-1} \cdots \gl_{m-1} \xl_m \gl_{m-1}^{-1} = [\vl_1, \ul_1] \cdots [\vl_k, \ul_k].\]
If $\ul_1, \cdots, \ul_k \in \Ng$, then we will have
\[ \cl_{\Gg,\Ng}(\xl_1 + \cdots + \xl_m) \le k,\]
which will yield (i). Thus it suffices to see that $\ul_1, \cdots, \ul_k \in N$. To see this, since $\gamma_1, \cdots, \gamma_j$ are components of $C$, property (A) of $C$ implies that $\gamma_1, \cdots, \gamma_j$ can be slide in $S \setminus C$, and property (B) implies that $\gamma_1, \cdots, \gamma_j$ are freely homotopic to loops in $S'$. The loops $\gamma_{j+1}, \cdots, \gamma_k$ are loops in $S \setminus C$ and hence property (B) implies that $\gamma_{j+1}, \cdots, \gamma_k$ are freely homotopic to loops in $S'$. Thus we can see that $\gamma_1, \cdots, \gamma_k$ are homotopic to loops $\gamma'_1, \cdots, \gamma'_k$ in $S'$. We can assume that the initial point of $\gamma_i'$ is a vertex of $S'$ for every $i \in \{1, \cdots, k\}$. By the homotopy extension property, there exists a continuous map $\varphi \colon S \to S$ which satisfies the following:
\begin{itemize}
\item $\varphi$ and the identity map ${\rm id}_S$ of $S$ are homotopic relative to the boundary of $S$.

\item $\varphi \circ \gamma_i = \gamma'_i$.
\end{itemize}
Set $\alpha'_i = \varphi \circ \alpha_i$. Set $\gbr_i = [f \circ \alpha'_i]$ and $\zl_i = [f \circ \gamma'_i]$. Then we have
\[ \ul_i = [f \circ (\alpha_i \cdot \gamma_i \cdot \bar{\alpha}_i)] = [f \circ \varphi \circ (\alpha_i \cdot \gamma_i \cdot \alpha'_i)] = [f \circ \alpha'_i] \cdot [f \circ \gamma'_i] \cdot [f \circ \bar{\alpha}'_i]= \gbr_i \zl_i \gbr_i^{-1}.\]

\begin{figure}[t]
\centering
\includegraphics[width=10cm]{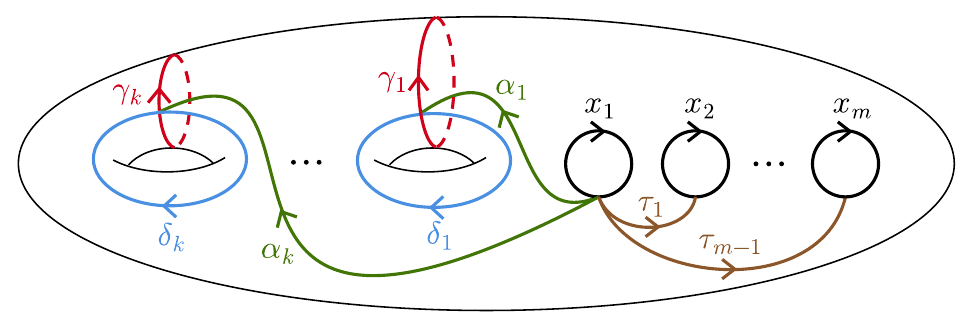}
\caption{paths $\alpha_i$, $\gamma_i$, $\delta_i$ and $\tau_i$} \label{figure key}
\end{figure}

Since $\gamma'_i$ is a loop of $S'$, the loop $f \circ \gamma'_i$ in $B\Gg$ is contained in $B\Ng$ and hence we have $\zl_i \in N$. Since $\Ng$ is a normal subgroup of $\Gg$, we have $\ul_i \in N$. Thus, we now have proved that (ii) implies (i) in Theorem~\ref{theorem characterization 2}; this completes the proof.
\end{proof}

Let $\pi$ be a permutation of $\{ 1, \cdots, m\}$. By definition, a $(G,N)$-simplicial surface whose boundary components are $x_1, \cdots, x_m$ is a $(G,N)$-simplicial surface whose boundary components are $x_{\pi(1)}, \cdots, x_{\pi(m)}$. Hence Theorem~\ref{theorem characterization 2} implies the following (recall also Remark~\ref{rem=productGN}).

\begin{corollary}
Let $m\in \NN$ and $\xl_1,\ldots,\xl_m\in \Ng$ such that $\xl_1\cdots \xl_m\in [\Gg,\Ng]$. Then, the value of $\cl_{\Gg,\Ng}(\xl_1 + \cdots + \xl_m)$ does not depend on the order of $\xl_1, \cdots, \xl_m$.
\end{corollary}

\subsection{$\scl_{G,N}$ for rational chains and admissible $(\Gg,\Ng)$-surfaces} \label{subsection 3.2}

In this subsection, we introduce the definition of $\scl_{G,N}$ for certain integral chains (Definition~\ref{definition=scl_integral_2}). After that, we deduce several algebraic properties of it by using admissible $(G,N)$-surfaces (Definition~\ref{definition=admissible}), and finally extend the definition of $\scl_{\Gg,\Ng}$ to rational chains (Definition~\ref{defn=scl_Q}).


Recall from Lemma~\ref{lem=Fekete} that a real sequence $(a_n)_{n\in \NN}$ is said to be subadditive if $a_{n_1 + n_2} \le a_{n_1} + a_{n_2}$ for every $n_1,n_2\in \NN$.

\begin{lemma} \label{lemma subadditive cl}
Let $m\in \NN$ and $\xl_1, \cdots, \xl_m \in \Ng$ such that $\xl_1 \cdots \xl_m \in [\Gg,\Ng]$. Set $a_n = \cl_{\Gg,\Ng}(\xl_1^n + \cdots + \xl_m^n) + (m-1)$. Then the sequence $(a_n)_{n \in \NN}$ is subadditive. In particular, there exists the limit
\[ \lim_{n \to \infty} \frac{1}{n} \cl_{\Gg,\Ng}(x_1^n + \cdots + x_m^n).\]
\end{lemma}
\begin{proof}
Let $n_1,n_2\in \NN$. Set $k_i = \cl_{\Gg,\Ng}(\xl_1^{n_i} + \cdots + \xl_m^{n_i})$ for $i \in \{1,2\}$. Let $S_i$ be a $(\Gg,\Ng)$-simplicial surface of genus $k_i$ whose boundary components are $\xl_1^{n_i}, \cdots, \xl_m^{n_i}$ for $i \in \{1,2\}$. Consider the disjoint union $S_1 \sqcup S_2$. Identify the basepoint of the boundary component labelled by $\xl_i^{n_1}$ in $S_1$ with the basepoint of the boundary component labelled by $\xl_i^{n_2}$ in $S_2$. After that, attach a triangle to that part for every $i\in \{1,\cdots, m\}$. Then we have a $(\Gg,\Ng)$-simpilcial surface of genus $k_1 + k_2 + (m-1)$ whose boundary components are $\xl_1^{n_1 + n_2}, \cdots, \xl_m^{n_1 + n_2}$. Thus we have
\begin{align*}
& \cl_{\Gg,\Ng}(\xl_1^{n_1 + n_2} + \cdots + \xl_m^{n_1 + n_2})\\
& \le k_1 + k_2 + (m-1) \\
& = \cl_{\Gg,\Ng}(\xl_1^{n_1} + \cdots + \xl_m^{n_1}) + \cl_{\Gg,\Ng}(\xl_1^{n_2} + \cdots + \xl_m^{n_2}) + (m-1).
\end{align*}
Adding $(m-1)$ to the above equality, we have  $a_{n_1 + n_2} \le a_{n_1} + a_{n_2}$.

The latter assertion follows from Fekete's lemma (Lemma~\ref{lem=Fekete}). This completes the proof.
\end{proof}

\begin{definition}\label{defn=scl_chain}
Let $m\in \NN$ and $\xl_1, \cdots, \xl_m \in \Ng$ such that $\xl_1 \cdots \xl_m \in [\Gg,\Ng]$. Then we define
\[ \scl_{\Gg,\Ng}(\xl_1 + \cdots + \xl_m) = \lim_{n \to \infty} \frac{1}{n} \cl_{\Gg,\Ng}(\xl_1^n + \cdots + \xl_m^n).\]
\end{definition}

\begin{definition}[admissible $(\Gg,\Ng)$-simplicial surface] \label{definition=admissible}
Let $\Gg$ be a group and $\Ng$ its normal subgroup.
Let $m\in \NN$ and $\xl_1, \cdots, \xl_m \in \Ng$. A $(\Gg,\Ng)$-simplicial surface $(S,f)$ is said to be \emph{admissible} with respect to $\xl_1, \cdots, \xl_m$ if there exist an integer $n(S)$ and a continuous map $f' \colon \partial S \to \coprod\limits_i S$ which satisfies the following properties.
\begin{enumerate}[(1)]
\item The diagram
\[ \xymatrix{
\partial S \ar[r]^i \ar[d]_{f'} & S \ar[d]^f \\
\coprod\limits_i S^1 \ar[r]^{\coprod \gamma_i} & BG
}\]
commutes up to free homotopy. Here $\gamma_i \colon S^1 \to B\Gg$ is a loop of $B\Gg$ whose pointed homotopy class is $\xl_i$.

\item $f'_* [\partial S] = n(S) \big[ \coprod\limits_i S^1\big]$.
\end{enumerate}
\end{definition}

The following is deduced from a geometric interpretation of the mixed commutator lengths for integral chains.

\begin{lemma}
Let $m\in \NN$ and $\xl_1, \cdots, \xl_m \in \Ng$. Let $n\in \NN$. Then, for every $k\in \NN$, the inequality
\[ \cl_{\Gg,\Ng}(\xl_1^n + \cdots + \xl_m^n) \le k\]
holds if and only if there exists an admissible $(\Gg,\Ng)$-simplicial surface $S$ such that $n(S) = n$ and the genus of $S$ is at most $k$.
\end{lemma}
\begin{proof}
If $\cl_{\Gg,\Ng}(\xl_1^n + \cdots + \xl_m^n) \le k$, then there exists a $(\Gg,\Ng)$-simplicial surface whose boundary components are $\xl_1^n, \cdots, \xl_m^n$ and whose genus is at most $k$. We write $C_i$ to mean the boundary component of $S$ which is labelled by $\xl_i^n$. Define the map
\[ f' \colon \partial S \to \coprod\limits_i S^1 \]
to be the coproduct of the degree $n$ map $C_i \to S^1$. Then the diagram
\[ \xymatrix{
\partial S \ar[r]^i \ar[d]_{f'} & S \ar[d]^f \\
\coprod\limits_i S^1 \ar[r]^{\coprod \gamma_i} & B\Gg
}\]
commutes up to free homotopy. Here $\gamma_i$ is a loop of $B\Gg$ whose pointed homotopy class is $x_i$.

To show the converse, suppose that there exists an admissible $(\Gg,\Ng)$-simplicial surface with respect to $\xl_1, \cdots, \xl_m$ whose genus is at most $k$. Namely, there exists the diagram
\[ \xymatrix{
\partial S \ar[r]^i \ar[d]_{f'} & S \ar[d]^f \\
\coprod\limits_i S^1 \ar[r]^{\coprod \gamma_i} & B\Gg,
}\]
which commutes up to free homotopy. Here $\gamma_i$ is the loop of $BG$ whose pointed homotopy class is $x_i$. Let $C_i$ denote the boundary component of $S$ which sends $i$-th $S^1$ in $\coprod\limits_i S^1$. Then the composition of $C_i \to S^1 \to B\Gg$ is $x_i^n$, and hence $f$ is a $(\Gg,\Ng)$-simplicial surface whose boundary components are $\xl_1^n, \cdots, \xl_m^n$. Hence we have $\cl_{\Gg,\Ng}(\xl_1^n + \cdots + \xl_m^n) \le k$.
\end{proof}

\begin{lemma}
Let $m\in \ZZ_{\geq 0}$. Let $\xl, \xl_1, \cdots, \xl_m \in N$ and $a\in \NN$. Assume that $\xl^a \xl_1 \cdots \xl_m \in [\Gg,\Ng]$. Then we have
\[ \scl_{\Gg,\Ng} \Big( \xl^a + \sum_{i=1}^m \xl_i \Big)
= \scl_{G,N} \Big( \underbrace{\xl + \cdots + \xl}_{a \text{ times}} + \sum_{i=1}^m \xl_i \Big).\]
\end{lemma}
\begin{proof}
Set $k(n) = \cl_{\Gg,\Ng}\Big(\underbrace{\xl^n + \cdots + \xl^n}_{a \text{ times}} + \sum\limits_{i=1}^m \xl_i^n\Big)$. Then there exists a connected admissible $(\Gg,\Ng)$-surface $S$ with respect to $ \xl, \cdots, \xl$ ($a$ times), $\xl_1, \cdots, \xl_m$ whose genus is $k(n)$ and $n(S) = n$. Then attaching an $(a + 1)$-gon depicted in Figure~\ref{figure a+1} to $a$ components labelled by $\xl^n$, we have a connected admissible $(G,N)$-simplicial surface $S'$ with respect to $x_1, \cdots, x_m$ such that the genus of $S'$ is $k(n) + (a-1)$ and $n(S') = n$. Thus we have
\[ \cl_{\Gg,\Ng}(\xl^{na} + \xl_1^n + \cdots + \xl_m^n) \le k(n) + (a-1).\]
By dividing by $n$ and taking the limits as $n\to \infty$, we have
\[ \scl_{\Gg,\Ng}(\xl^a + \xl_1 + \cdots + \xl_m) \le \scl_{\Gg,\Ng}( \underbrace{\xl + \cdots + \xl}_{a \text{ times}} + \xl_1 + \cdots + \xl_m).\]

\begin{figure}[t]
  \begin{minipage}[b]{0.5\linewidth}
    \centering
\includegraphics[width=4.5cm]{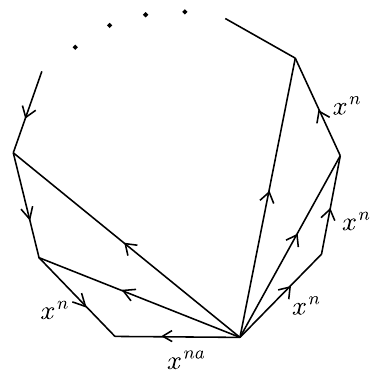}
\caption{$a x^n$ to $x^{na}$} \label{figure a+1}
  \end{minipage}
  \begin{minipage}[b]{0.5\linewidth}
    \centering
\includegraphics[width=4.5cm]{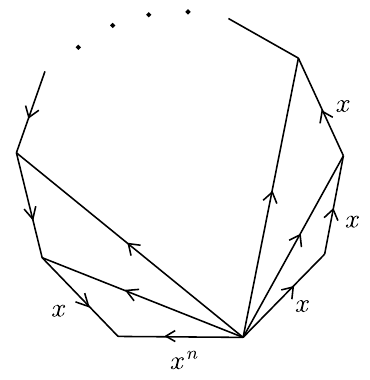}
\caption{$x^n$ to $nx$} \label{figure xn}
  \end{minipage}
\end{figure}



On the other hand, by the definition of $\cl_{\Gg,\Ng}$ we have
\[ \cl_{\Gg,\Ng}(\xl^{na}  + \xl_1^n + \cdots + \xl_m^n) \ge \cl_{\Gg,\Ng}(\underbrace{\xl^n + \cdots + \xl^n}_{a \text{ times}} + \xl_1^n + \cdots + \xl_m^n).\]
By dividing by $n$ and taking the limits as $n\to \infty$, we have
\[ \scl_{\Gg,\Ng}(\xl^a + \xl_1 + \cdots + \xl_m) \ge \scl_{\Gg,\Ng}(\underbrace{\xl + \cdots + \xl}_{a \text{ times}} + \xl_1 + \cdots + \xl_m).\]
This completes the proof.
\end{proof}

\begin{lemma}
Let $m\in \NN$ and $\xl, \xl_1, \cdots, \xl_m \in \Ng$. Assume that $\xl_1 \cdots \xl_m \in [\Gg,\Ng]$. Then the following holds:
\[ \scl_{\Gg,\Ng} \Big( \xl + \xl^{-1} + \sum_{i=1}^m \xl_i \Big) = \scl_{\Gg,\Ng}\Big( \sum_{i=1}^m \xl_i \Big).\]
\end{lemma}
\begin{proof}
Set $k(n) = \cl_{\Gg,\Ng}\Big(\xl^n + \xl^{-n} + \sum\limits_{i=1}^m \xl_i^n\Big)$. Then there exists a connected admissible $(\Gg,\Ng)$-simplicial surface $S$ with respect to $x, x^{-1}, x_1, \cdots, x_m$ such that the genus of $S$ is $k(n)$ and $n(S) = n$. Then attaching a $2$-simplex to the parts of $\xl$ and $\xl^{-1}$, we have a connected admissible $(\Gg,\Ng)$-surface with respect to $e, \xl_1, \cdots, \xl_m$ such that $n(S) = n$ and the genus of $S$ is $k(n) + 1$. Hence we have
\begin{align*}
\cl_{\Gg,\Ng}(\xl_1^n + \cdots + \xl_m^n) &= \cl_{\Gg,\Ng}(e^n + \xl_1^n + \cdots + \xl_m^n) \le k(n) + 1 \\
&= \cl_{\Gg,\Ng}(\xl^n + x^{-n} + \xl_1^n + \cdots + \xl_m^n) + 1.
\end{align*}
By dividing the above inequality by $n$ and taking the limits as $n\to \infty$, we have
\[ \scl_{\Gg,\Ng}(\xl_1 + \cdots + \xl_m) \le \scl_{\Gg,\Ng}(\xl + \xl^{-1} + \xl_1 + \cdots + \xl_m).\]

On the other hand, by the definition of $\cl_{G,N}$, we have
\[ \cl_{\Gg,\Ng}(\xl_1^n + \cdots + \xl_m^n) \ge \cl_{\Gg,\Ng}(\xl^n + \xl^{-n} + \xl_1^n + \cdots + \xl_m^n).\]
By dividing the above inequality by $n$ and taking the limits as $n\to \infty$, we have
\[ \scl_{\Gg,\Ng}(\xl_1 + \cdots + \xl_m) \ge \scl_{\Gg,\Ng}(\xl + \xl^{-1} + \xl_1 + \cdots + \xl_m).\]
This completes the proof.
\end{proof}

Now we define $\scl_{\Gg,\Ng}$ for integral chains in a certain condition.

\begin{definition} \label{definition=scl_integral_2}
Let $\Gg$ be a group and $\Ng$ its normal subgroup. Let $k,l\in \ZZ_{\geq 0}$. Let $\xl_1, \cdots, \xl_k, \xbr_1, \cdots, \xbr_l \in \Ng$ such that $\xl_1 \cdots \xl_k \xbr_1^{-1} \cdots \xbr_l^{-1} \in [\Gg,\Ng]$. Define the \emph{stable mixed commutator length of the chain $\xl_1 + \cdots + \xl_k - \xbr_1 - \cdots - \xbr_l$} by
\[ \scl_{\Gg,\Ng}(\xl_1 + \cdots + \xl_k - \xbr_1 - \cdots - \xbr_l) = \scl_{\Gg,\Ng}(\xl_1 + \cdots + \xl_k + \xbr_1^{-1} + \cdots + \xbr_l^{-1}).\]
\end{definition}

By Definition~\ref{defn=scl_chain}, in the setting of Definition~\ref{definition=scl_integral_2} we have
\begin{align*}
&\scl_{\Gg,\Ng}(\xl_1 + \cdots + \xl_k - \xbr_1 - \cdots - \xbr_l)\\
&=\lim_{n \to \infty} \frac{1}{n} \cl_{\Gg,\Ng}(\xl_1^n + \cdots + \xl_k^n+\xbr_1^{-n}+\cdots +\xbr_l^{-n}).
\end{align*}

Next we provide a characterization of the condition that $\xl_1, \cdots, \xl_k, \xbr_1, \cdots, \xbr_l \in N$ and $\xl_1 \cdots \xl_k \xbr_1^{-1} \cdots \xbr_l^{-1} \in [\Gg, \Ng]$ in Definition~\ref{definition=scl_integral_2}. Here we recall our notation and terminology from Subsection~\ref{subsec=Nqhom}. For a unital commutative ring $A$, $\CCC_2(\Gg ; A)$ is the group of $2$-chains of $\Gg$ with $A$-coefficients and $\CCC_2'(\Gg; A)$ is the $A$-submodule of $\CCC_2(\Gg ; A)$ generated by the set
\[ \{ (\gl_1,\gl_2) \in \CCC_2(\Gg; A)\; | \; \textrm{$\gl_1$ or $\gl_2$ belongs to $\Ng$}\}.\]
Recall from Definition~\ref{defn=BBB'} that $\BBB_1'(\Gg,\Ng ; A)$ is defined to be  the image of $\CCC'_2(\Gg,\Ng; A)$ through the boundary $\partial \colon \CCC_2(\Gg ; A) \to \CCC_1(\Gg; A)$. To state the generalized mixed Bavard duality theorem, we define the following space of chains.

\begin{definition} \label{definition=CA}
Let $\Gg$ be a group and $\Ng$ its normal subgroup.  Let $A$ be a unital commutative ring. Then define $\CA(\Gg,\Ng)$ by
\[ \CA(\Gg,\Ng) = \CCC_1(\Ng ; A) \cap \BBB'_1(\Gg,\Ng; A).\]
\end{definition}

\begin{remark}\label{rem=BBB_C}
Recall from Definition~\ref{defn=BBB} that $\BBB_1(\Gg;A)=\partial \CCC_2(\Gg;A)$. By definition, if $\Ng=\Gg$, then we have $\CA(\Gg,\Ng)=\BBB_1(\Gg;A)$.
\end{remark}

The following lemma characterizes elements in $\CZ(\Gg,\Ng)$. Recall that a part of the assertions of Lemma~\ref{lem='norm_cl} states that $\yl$ is an element of $\BBB'_1(\Gg,\Ng;\RR)$ for every $\yl\in \CGN$. This statement can also be deduced from`(i) implies(ii)' in Lemma~\ref{lemma=domain}. We also note that  the special case of Lemma~\ref{lemma=domain} for $\Ng=\Gg$ recovers Lemma~\ref{lem=domainG} in Section~\ref{sec=gBavard} by Remark~\ref{rem=BBB_C}.

\begin{lemma} \label{lemma=domain}
Let $c \in \CCC_1(\Gg; \ZZ)$. Then the following are equivalent.
\begin{enumerate}
\item[\textup{(i)}] There exist $k,l\in \ZZ_{\geq 0}$ and $\xl_1, \cdots, \xl_k, \xbr_1, \cdots, \xbr_l \in \Ng$ such that
\[ c = \xl_1 + \cdots + \xl_k - \xbr_1 - \cdots - \xbr_l\]
and $\xl_1 \cdots \xl_k \xbr_1^{-1} \cdots \xbr_l^{-1} \in [\Gg,\Ng]$.

\item[\textup{(ii)}] $c$ belongs to $\CZ(\Gg,\Ng)$.
\end{enumerate}
\end{lemma}
\begin{proof}
First, we show that (i) implies (ii). It is clear that $c \in \CCC_1(\Ng ; \ZZ)$. Since $\xl_1 \cdots \xl_k \xbr_1^{-1} \cdots \xbr_l^{-1} \in [\Gg,\Ng]$, there exists a simplicial $(\Gg,\Ng)$-surface whose boundary component is $\xl_1\cdots \xl_k \xbr_1^{-1} \cdots \xbr_l^{-1}$ (see Theorem~\ref{theorem characterization 1}). By attaching triangles depitected in Figure~\ref{figure product}, we obtain a $(G,N)$-simplicial surface whose boundary components are $\xl_1, \cdots, \xl_k, \xbr_1^{-1},\cdots, \xbr_l^{-1}$. Since $\xbr_i + \xbr_i^{-1} \in \BBB_1'(\Gg,\Ng; \ZZ)$, we have $\xl_1 + \cdots + \xl_k - \xbr_1 - \cdots - \xbr_l \in \BBB'_1(\Gg,\Ng; \ZZ)$.

\begin{figure}[t]
    \centering
\includegraphics[width=5cm]{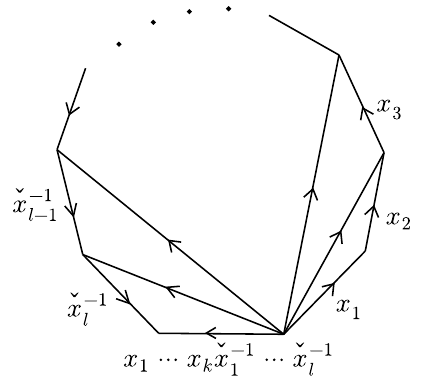}
\caption{$\xl_1 \cdots \xl_k \xbr_1^{-1} \cdots \xbr_l^{-1}$ to $\xl_1 + \cdots + \xl_k + \xbr_1^{-1} + \cdots + \xbr_l^{-1}$} \label{figure product}
\end{figure}

Next, we show that (ii) implies (i). If $\xl_1, \cdots, \xl_k, \xbr_1, \cdots, \xbr_l \in \Ng$ and $\xl_1 + \cdots + \xl_k - \xbr_1 - \cdots - \xbr_l \in \BBB_1'(\Gg,\Ng ; \ZZ)$, then there exists a $(\Gg,\Ng)$-simplicial surface whose boundary components are $\xl_1, \cdots, \xl_k, \xbr_1^{-1}, \cdots, \xbr_l^{-1}$. By Theorem~\ref{theorem characterization 2} there exist $\gl_1, \cdots, \gl_k , \gbr_1, \cdots, \gbr_l \in \Gg$ such that
\[ \gl_1 \xl_1 \gl_1^{-1} \cdots \gl_k \xl_k \gl_k^{-1} \gbr_1 \xbr_1^{-1} \gbr_1^{-1} \cdots \gbr_l \xbr_l^{-1} \gbr_l \in [\Gg,\Ng].\]
Since every element of $\Gg$ and every element of $\Ng$ commute in $\Gg/ [\Gg,\Ng]$, we have
\[ \xl_1 \cdots \xl_k \xbr_1^{-1} \cdots \xbr_l^{-1} \in [\Gg,\Ng].\]
This completes the proof.
\end{proof}

Let $k\in \NN$ and $\xl_1,\ldots,\xl_k\in \Ng$ satisfy that $\xl_1\cdots\xl_k\in \CGN$. Then, clearly we have
\[ |m| \scl_{\Gg,\Ng} \Big( \sum_{i=1}^k \xl_i \Big) = \scl_{\Gg,\Ng} \Big( \sum_{i=1}^k \xl_i^m \Big)\]
for every $m\in \ZZ$. Hence, we can extend the definition of $\scl_{\Gg,\Ng}$ to $\CQ(\Gg,\Ng)$ as follows.

\begin{definition}\label{defn=scl_Q}
Let $\Gg$ be a group and $\Ng$ its normal subgroup.
Let $c \in \CQ(\Gg, \Ng)$. Then there exists $m\in \NN$ such that $mc \in \CZ(\Gg, \Ng)$. Define $\scl_{\Gg,\Ng}(c)$ by
\[ \scl_{\Gg,\Ng}(c) = \frac{1}{m}\scl_{\Gg,\Ng}(mc).\]
\end{definition}

The well-definedness of $\scl_{\Gg,\Ng}(c)$ above is straightforward from the definition.

\subsection{Filling norm for rational chains} \label{subsection 3.3}

In this subsection, we define the filling norm (Definition~\ref{definition=filling}) for elements in $\CQ(\Gg,\Ng)$. Then, we provide the relationship between this norm and $\scl_{\Gg,\Ng}$ (Theorem~\ref{theorem=4scl=fill}), which plays a key role in the proof of the generalized mixed Bavard duality for rational chains.

\begin{lemma}\label{lem=n-jo}
Let $k,l\in \ZZ_{\geq 0}$. Let $\xl_1, \cdots, \xl_k, \xbr_1, \cdots, \xbr_l \in \Ng$. If
\[ (\xl_1 + \cdots + \xl_k) - (\xbr_1 + \cdots + \xbr_l) \in \BBB'_1(\Gg,\Ng ; \ZZ),\]
then for every $n\in \NN$ we have
\[ (\xl_1^n + \cdots + \xl_k^n) - (\xbr_1^n + \cdots + \xbr_l^n) \in \BBB'_1(\Gg,\Ng ; \ZZ).\]
\end{lemma}
\begin{proof}
Observe that $\xl^n - n \xl \in \BBB'_1(\Gg,\Ng; \ZZ)$ for every $\xl\in \Ng$ and that
\[ (n \xl_1 + \cdots + n \xl_k) - (n\xbr_1 + \cdots + n\xbr_l) \in \BBB'_1(\Gg,\Ng; \ZZ).\]
Now the assertion immediately follows.
\end{proof}

Next we define the filling norm of rational chains. Recall from Definition~\ref{defn=BBB'}~(3) that  we set
\[ \| c\|' = \inf \{ \| c' \|_1 \; | \; c' \in \CCC_2'(\Gg,\Ng ; \RR),\ \partial c'=c \}\]
for $c \in \BBB_1'(\Gg,\Ng; \RR)$.

\begin{lemma} \label{lemma=existence_filling}
Let $k,l\in \ZZ_{\geq 0}$. Let $\xl_1, \cdots, \xl_k, \xbr_1, \cdots, \xbr_l \in \Ng$, and assume that
\[ (\xl_1 + \cdots + \xl_k) - (\xbr_1 + \cdots + \xbr_l) \in \BBB'_1(\Gg,\Ng; \ZZ).\]
For every $n\in \NN$, set
\[ a_n = \big\| (\xl_1^n + \cdots + \xl_k^n) - (\xbr_1^n + \cdots + \xbr_l^n) \big\|' + (k + l).\]
Then the sequence $(a_n)_{n\in \NN}$ is subadditive.
\end{lemma}
\begin{proof}
Let $m,n\in \NN$.
Then we have
\begin{align*}
a_m + a_n & = \left\| \sum_{i=1}^k \xl_i^m - \sum_{j=1}^l \xbr_j^m \right\|' + \left\| \sum_{i=1}^k \xl_i^n - \sum_{j=1}^l \xbr_j^n \right\|' + 2 (k + l) \\
& \ge \left\| \Big( \sum_{i=1}^k \xl_i^m - \sum_{j=1}^l \xbr_j^m \Big) - \Big( \sum_{i=1}^k \xl_i^n - \sum_{j=1}^l \xbr_j^n \Big) \right\|' + 2 (k + l) \\
& = \left\| \sum_{i=1}^k \partial (\xl_i^m, \xl_i^n) - \sum_{j=1}^l \partial (\xbr_j^m, \xbr_j^n) + \sum_{i=1}^k \xl_i^{m+n} - \sum_{j=1}^l \xbr_j^{m+n}\right\|' + 2(k+l) \\
& \ge \left\| \sum_{i=1}^k \xl_i^{m+n} - \sum_{j=1}^l \xbr_j^{m+n} \right\|' - (k + l) + 2 (k + l) \\
& = a_{m+n},
\end{align*}
as desired.
\end{proof}

We are now ready to introduce the filling norm $\fl_{G,N}$ on $\CQ(\Gg ; \Ng)$.

\begin{definition} \label{definition=filling}
Let $\Gg$ be  a group and $\Ng$ its normal subgroup.
\begin{enumerate}[(1)]
 \item Let $c \in \CZ(\Gg,\Ng)$, and write $c=(\xl_1 + \cdots + \xl_k) - (\xbr_1 + \cdots + \xbr_l)$, where $k,l\in \ZZ_{\geq 0}$ and $\xl_1, \cdots, \xl_k, \xbr_1, \cdots, \xbr_l \in \Ng$. Then,   define the \emph{filling norm $\fl_{G,N}(c)$ of $c$} to be the number
\begin{equation}\label{eq=lim_c}
 \fl_{\Gg,\Ng}(c) = \lim_{n \to \infty} \frac{\| (\xl_1^n + \cdots + \xl_k^n) - (\xbr_1^n + \cdots + \xbr_l^n)\|' }{n}.
\end{equation}
 \item For $c \in \CQ(\Gg,\Ng)$, let $m \in \NN$ such that $mc \in \CZ(\Gg,\Ng)$. Define the \emph{filling norm $\fl_{G,N}(c)$ of $c$} by
\[ \fl_{\Gg,\Ng}(c) = \frac{1}{m} \fl_{\Gg,\Ng}(mc).\]
\end{enumerate}
\end{definition}
In Definition~\ref{definition=filling}, the existence of the limit appearing in \eqref{eq=lim_c} follows from Lemma~\ref{lemma=existence_filling} (and Lemma~\ref{lem=Fekete}). We also note that for $c \in \CZ(\Gg,\Ng)$ and for $m \in \ZZ$ we have
\begin{align} \label{filling linear}
\fl_{\Gg,\Ng}(mc) = |m| \cdot \fl_{\Gg,\Ng}(c);
\end{align}
this ensures the well-definedness of  $\fl_{\Gg,\Ng}(c)$ for $c \in \CQ(\Gg,\Ng)$.

\begin{theorem} \label{theorem=4scl=fill}
For $c \in \CQ(\Gg,\Ng)$, we have
\[ 4 \scl_{\Gg,\Ng}(c) = \fl_{\Gg,\Ng}(c).\]
\end{theorem}

This follows from Proposition~\ref{proposition 4scl le fill} and Corollary~\ref{corollary trivial} below.

\begin{proposition} \label{proposition 4scl le fill}
For $c \in \CQ(\Gg,\Ng)$, we have
\[ 4 \scl_{\Gg,\Ng}(c) \le \fl_{\Gg,\Ng}(c).\]
\end{proposition}
\begin{proof}
It suffices to show that $4 \scl_{\Gg,\Ng}(c)\leq \fl_{\Gg,\Ng}(c)$ for $c \in \CZ(\Gg,\Ng)$. Let $\xl_1, \cdots, \xl_k, \xbr_1, \cdots, \xbr_l \in \Ng$ such that
\[ c = (\xl_1 + \cdots + \xl_k) - (\xbr_1 + \cdots + \xbr_l) \in \BBB_1'(\Gg,\Ng ;\ZZ).\]
Set $A = \| (\xl_1 + \cdots + \xl_k) - (\xbr_1 + \cdots + \xbr_l)\|'$. Then there exists a $(\Gg,\Ng)$-simplicial surface $S'$ such that the number of $2$-simplices of $S'$ is $A$ and we have
\[ \partial S' = (\xl_1 + \cdots + \xl_k) - (\xbr_1 + \cdots + \xbr_l).\]
We attach two $2$-simplices to $\xl_i$ for each $i \in \{1, \cdots, k\}$ and two $2$-simplices to $\xbr_j$ for each $j \in \{1, \cdots, l\}$ depicted in Figure~\ref{figure x to inverse}. Then we have a $(\Gg,\Ng)$-simplicial surface $S$ such that the number of $2$-simplices of $S$ is $A + 2(k+l)$ and the boundary of $S$ is labelled by $\xl_1, \cdots, \xl_k, \xbr_1^{-1}, \cdots, \xbr_l^{-1}$. Let $v,e,f$ be the numbers of $0$-simplices, $1$-simplices, $2$-simplices of $S$, respectively. Let $g$ be the genus of $S$. We observe the following two points.
\begin{enumerate}[(1)]
\item Comparing the Euler characteristics, we have
\begin{align} \label{equality Euler}
f - e + v = 2 - 2g - (k+l).
\end{align}

\begin{figure}[t]

    \centering
\includegraphics[width=9cm]{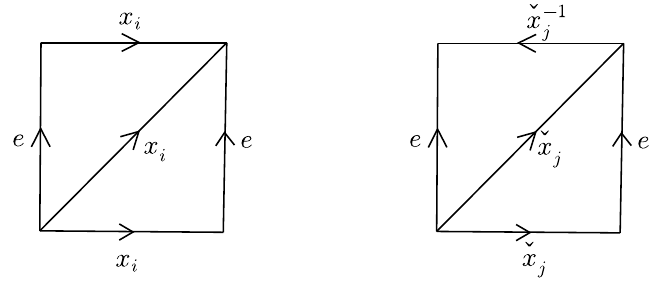}
\caption{$\xl_i$ to $\xl_i$ and $-\xbr_j$ to $\xbr_j^{-1}$}\label{figure x to inverse}
\end{figure}

\item Every $2$-simplex of $S$ is surrounded by three $1$-simplices. Every $1$-simplex not on the boundary is contained in exactly two $2$-simplices, and every $1$-simplex on the boundary is contained in exactly one $2$-simplex. These imply
\begin{align} \label{equality simplices}
(k+ l) + 2 (e - k - l) = 3f.
\end{align}
\end{enumerate}
By \eqref{equality Euler} and \eqref{equality simplices} and $A + 2(k+l) = s$, we have
\[ A + 2 (k + l) = 4g + 4 + k + l + 2v.\]
Since $p \ge k + l$ and
\[ g \ge \cl_{\Gg,\Ng}(\xl_1 + \cdots + \xl_k + \xbr_1^{-1} + \cdots + \xbr_l^{-1}),\]
we have
\begin{align*}
& \| (\xl_1 + \cdots + \xl_k) - (\xbr_1 + \cdots + \xbr_l) \|' \\
& \ge 4g + 4 + k + l \\
& \ge 4 \cl_{\Gg,\Ng}(\xl_1 + \cdots + \xl_k + \xbr_1^{-1} + \cdots + \xbr_l^{-1}) + k + l + 4.
\end{align*}
Thus for every $n\in \NN$, we have
\[ \| (\xl_1^n + \cdots + \xl_k^n) - (\xbr_1^n + \cdots + \xbr_l^n) \|' \ge 4 \cl_{\Gg,\Ng}(\xl_1^n + \cdots + \xl_k^n + \xbr_1^{-n} + \cdots + \xbr_l^{-n}) + k + l + 4.\]
By dividing the above inequality by $n$ and taking the limits as $n\to \infty$, we have
\[ \fl_{\Gg,\Ng}\Big( \sum_{i=1}^k \xl_i - \sum_{j=1}^l \xbr_j\Big) \ge 4 \scl_{\Gg,\Ng} \Big( \sum_{i=1}^k \xl_i + \sum_{j=1}^l \xbr_j^{-1} \Big) = 4 \scl_{\Gg,\Ng}\Big(\sum_{i=1}^k \xl_i - \sum_{j=1}^l \xbr_j\Big).\]
This completes the proof. \end{proof}

\begin{proposition} \label{proposition trivial}
Let $k,l\in \ZZ_{\geq 0}$. Let $\xl_1, \cdots, \xl_k, \xbr_1, \cdots, \xbr_l \in \Ng$, and assume that
\[ (\xl_1 + \cdots + \xl_k) - (\xbr_1 + \cdots + \xbr_l) \in \BBB_1'(\Gg,\Ng;\ZZ).\]
Then we have
\[ \| \xl_1 + \cdots + \xl_k - \xbr_1 - \cdots - \xbr_l \|' \le 4 \cl_{\Gg,\Ng}(\xl_1 + \cdots + \xl_k + \xbr_1^{-1} + \cdots + \xbr_l^{-1}) - 1 + 3(k+l).\]
\end{proposition}
\begin{proof}
Set $M = \cl_{\Gg,\Ng}(\xl_1 + \cdots + \xl_k + \xbr_1^{-1} + \cdots + \xbr_l^{-1})$. Then there exist simple  $(\Gg,\Ng)$-commutators $\yl_1, \cdots, \yl_M$ and $\gl_1, \cdots, \gl_k, \gbr_1, \cdots, \gbr_l \in G$ such that
\[ \yl_1 \cdots \yl_M = \gl_1 \xl_1 \gl_1^{-1} \cdots \gl_k \xl_k \gl_k^{-1} \cdot \gbr_1 \xbr_1^{-1} \gbr_1^{-1} \cdots \gbr_l \xbr_l^{-1} \gbr_l^{-1}.\]
By Lemma~\ref{lem='norm_cl},  we have
\begin{align*}
\| \gl_1 \xl_1 \gl_1^{-1} \cdots \gl_k \xl_k \gl_k^{-1} \cdot \gbr_1 \xbr_1^{-1} \gbr_1^{-1} \cdots \gbr_l \xbr_l^{-1} \gbr_l^{-1}\|' &= \| \yl_1 \cdots \yl_M\|' \\
&\le \| \yl_1\|' + \cdots + \| \yl_M\|' + (M - 1) \\
&\le 4M - 1.
\end{align*}
Attaching $3(k+l)$ triangles to the boundary $\gl_1 \xl_1 \gl_1^{-1} \cdots \gl_k \xl_k \gl_k^{-1} \cdot \gbr_1 \xbr_1^{-1} \gbr_1^{-1} \cdots \gbr_l \xbr_l^{-1} \gbr_l^{-1}$, we have
\begin{align*}
& \| \xl_1 + \cdots + \xl_k - \xbr_1 - \cdots - \xbr_l\|' \\
& \le 4 M - 1 + 3 (k+l) \\
& = 4 \cl_{\Gg,\Ng}(\xl_1 + \cdots + \xl_k + \xbr_1^{-1} + \cdots + \xbr_l^{-1}) - 1 + 3 (k+l).
\end{align*}
This completes the proof.
\end{proof}

\begin{corollary} \label{corollary trivial}
For $c \in \CQ(\Gg,\Ng)$, we have
\[ \fl_{\Gg,\Ng}(c) \le 4 \scl_{\Gg,\Ng}(c).\]
\end{corollary}
\begin{proof}
It suffices to show the case that $c \in \CZ(\Gg ,\Ng)$. Let $\xl_1, \cdots, \xl_k, \xbr_1, \cdots, \xbr_l \in N$ such that
\[ c = \xl_1 + \cdots + \xl_k - \xbr_1 - \cdots - \xbr_l.\]
By Proposition~\ref{proposition trivial}, we have
\[ \| \xl_1^n + \cdots + \xl_k^n - \xbr_1^n - \cdots - \xbr_l^n \|' \le 4 \cl_{\Gg,\Ng}(\xl_1^n + \cdots + \xl_k^n + \xbr_1^{-n} + \cdots + \xbr_l^{-n}) - 1 + 2(k+l).\]
By dividing the above inequality by $n$ and taking the limits as $n\to \infty$, we have
\[ \fl_{G,N}(c) \le 4 \scl_{G,N}(c),\]
which completes the proof.
\end{proof}

\subsection{Proof of the generalized mixed Bavard duality theorem} \label{subsection 3.4}

Recall that  the goal in this section is to prove the generalized Bavard duality theorem for stable mixed commutator length. In this subsection, we complete the proof. For brevity of the description, we introduce the following notation.

\begin{definition}\label{defn=mufR}
Let $\Gg$ be a group and $\Ng$ its normal subgroup.
Let $\muf \in \QQQ(\Ng)^{\Gg}$. Then we regard $\muf$ as a map from $\CR(\Gg,\Ng)$ to $\RR$ in the following manner: for $c \in \CR(\Gg,\Ng)$, write $c = \sum\limits_{i=1}^m t_i \xl_i$, where  $m\in \ZZ_{\geq 0}$, $t_1,\ldots, t_m\in\RR$, and $\xl_1,\ldots ,\xl_m\in \Ng$. Then, define
\[ \muf(c) = \sum_{i=1}^m t_i \muf(\xl_i).\]
\end{definition}

\begin{lemma}\label{lem=H^1is0}
Let  $ \hf \in \HNRG$. Then we have $\hf(c) = 0$ for every $c \in \CR(\Gg,\Ng)$.
\end{lemma}
\begin{proof}
Since $\RR$ is flat over $\ZZ$, the map
\[ \CZ(\Gg, \Ng) \otimes \RR \xrightarrow{\cong} \CR(\Gg,\Ng); \quad \sum_{i=1}^m x_i \otimes r_i \mapsto \sum_{i=1}^m r_i x_i\]
is an isomorphism.  Hence, for every $\muf\in \QQQ(\Ng)^{\Gg}$, the  map $\muf\colon \CR(\Gg,\Ng)\to \RR$ (defined in Definition~\ref{defn=mufR}) is induced by the bilinear form
\[ \CZ(\Gg,\Ng) \times \RR \to \RR, \quad ( c', r) \mapsto r \cdot \muf( c').\]
 Since we are dealing with the case where $\muf=\hf\in \HNRG$,  it suffices to show that $ \hf( c') = 0$ for every $ c'\in \CZ(\Gg, \Ng)$  in order to prove the lemma. By Lemma~\ref{lemma=domain}, there exist $\xl_1, \cdots, \xl_k, \xbr_1, \cdots, \xbr_l \in \Ng$ such that
\[  c' = \xl_1 + \cdots + \xl_k - \xbr_1 - \cdots - \xbr_l, \quad \xl_1 \cdots \xl_k \xbr_1^{-1} \cdots \xbr_l^{-1} \in [\Gg,\Ng].\]
Since $ \hf \in \HNRG$, we have
\begin{align*}
\hf(c') &= \hf(\xl_1) + \cdots + \hf(\xl_k) - \hf(\xbr_1) - \cdots - \hf(\xbr_l)\\
&= \hf(\xl_1 \cdots \xl_k \xbr_1^{-1} \cdots \xbr_l^{-1}) \\
&= 0.
\end{align*}
This completes the proof.
\end{proof}

Lemma~\ref{lem=H^1is0} means that the map $\muf \colon \CR(\Gg, \Ng) \to \RR$ induced by $\muf \in \QQQ(\Ng)^{\Gg}$ is determined by the equivalence class in $\QQQ(\Ng)^{\Gg} / \HNRG$.

\begin{theorem}[generalized mixed Bavard duality theorem for rational chains] \label{generalized mixed bavard rational}
Let $\Gg$ be a group and $\Ng$ its normal subgroup.
Then, for every rational chain $c \in \CQ(\Gg,\Ng)$, the equality
\[ \scl_{\Gg,\Ng}(c) =\sup_{\muf \in \QQQ(N)^G \setminus \HNRG} \frac{|\muf(c)|}{2\DD(\muf)}\]
holds true.
\end{theorem}
\begin{proof}
It suffices to show the case that $c$ belongs to $\CZ(\Gg,\Ng)$ with positive coefficients. Namely, there exist $\xl_1, \cdots, \xl_k \in \Ng$ such that $c = \xl_1 + \cdots + \xl_k \in \BBB_1'(\Gg,\Ng; \ZZ)$. We first show
\[ \scl_{\Gg,\Ng}(\xl_1 + \cdots + \xl_k) \ge \sup_{\muf \in \QQQ(\Ng)^{\Gg} \setminus \HNRG} \frac{\muf(\xl_1) + \cdots + \muf(\xl_k)}{2\DD(\muf)}.\]
To see this, we set $\cl_{\Gg,\Ng}(\xl_1 + \cdots + \xl_k) = M$. Then there exist $\gl_1, \cdots, \gl_k \in \Gg$ and  simple $(\Gg,\Ng)$-commutators $\yl_1, \cdots, \yl_M$ such that $\gl_1 \xl_1 \gl_1^{-1} \cdots \gl_k \xl_k \gl_k^{-1} = \yl_1 \cdots \yl_M$. For every $\muf \in \QQQ(\Ng)^{\Gg}$, we have
\[ \muf(\gl_1 \xl_1 \gl_1^{-1} \cdots \gl_k \xl_k \gl_k^{-1}) \sim_{(k-1)\DD(\muf)} \muf(\xl_1) + \cdots + \muf(\xl_k)\]
and
\[ \muf(\yl_1 \cdots \yl_M) \sim_{(M-1)\DD(\muf)} \muf(\yl_1) + \cdots + \muf(\yl_M) \sim_{M \DD(\muf)} 0.\]
Thus we have
\[ |\muf(\xl_1) + \cdots + \muf(\xl_k)| \le (2M + k - 2) \DD(\muf) = (2 \cl_{\Gg,\Ng}(\xl_1 + \cdots + \xl_k) + k -2) \DD(\muf).\]
Therefore, for every $n\in \NN$, we have
\[ n \cdot |\muf(\xl_1) + \cdots + \muf(\xl_k)| = |\muf(\xl_1^n) + \cdots + \muf(\xl_k^n) | \le (2 \cl_{\Gg,\Ng}(\xl_1^n + \cdots + \xl_k^n) + k -2) \DD(\muf)
\]
(here, recall Lemma~\ref{lem=n-jo}).
Hence, by dividing the above inequality by $n$ and taking the limits as $n\to \infty$, we have
\[ |\muf(\xl_1) + \cdots + \muf(\xl_k)| \le 2 \DD(\muf) \scl_{\Gg,\Ng}(\xl_1 + \cdots + \xl_k).\]

Secondly, we show the  converse of the inequality. We note that
\begin{align*}
4 \scl_{\Gg,\Ng}(\xl_1 + \cdots + \xl_k) &= \fl_{\Gg,\Ng}(\xl_1 + \cdots + \xl_k) \\
&= \lim_{n \to \infty} \frac{\| \xl_1^n + \cdots \xl_k^n \|'}{n} \\
&= \lim_{n \to \infty} \left( \sup_{\psf \in \rQQQ_{\Ng}(\Gg) \setminus \HHH^1_{\Ng}(\Gg;\RR)} \frac{|\psf (\xl_1^n + \cdots + \xl_k^n)|}{n \DD''_{\Gg,\Ng}(\psf)}\right).
\end{align*}
Here, recall from Proposition~\ref{proposition=isometry} that the continuous dual of $(\BBB_1'(\Gg,\Ng ; \ZZ), \|\cdot\|')$ is isometrically isomorphic to $(\rQQQ_{\Ng}(\Gg) / \HHH^1_{\Ng}(\Gg;\RR),\DD_{\Gg,\Ng})$; we employ Corollary~\ref{cor=HB} to obtain the last equality above. In what follows, we abbreviate $\rQQQ_{\Ng}(\Gg)$ as $\rQQQ_{\Ng}$ and $\HHH^1_{\Ng}(\Gg;\RR)$ as $\HHH^1_{\Ng}$, respectively.
For every $n,m\in \NN$, we take $\psf_{n,m} \in \rQQQ_{\Ng}(\Gg)$ such that
\[ \sup_{\psf \in \rQQQ_{\Ng}(\Gg) \setminus \HHH^1_{\Ng}(\Gg)} \frac{|\psf(\xl_1^n + \cdots + \xl_k^n)|}{n \DD''_{\Gg,\Ng}(\psf)} \sim_{m^{-1}} \frac{|\psf_{n,m}(\xl_1^n + \cdots + \xl_k^n)|}{n D''_{\Gg,\Ng}(\psf_{n,m})}.\]
Set $\mufh_{n,m}=\psf_{n,m}|_{\Ng}$; by Lemma~\ref{lemma=restriction}, this is a $\Gg$-quasi-invariant quasimorphism on $\Ng$. We define $\muf_{n,m}$ as the homogenization of $\mufh_{n,m}$; by Corollary~\ref{cor=homoge_qinv}, $\muf_{n,m}$ belongs to $\QNG$. Also, by Lemma~\ref{lem=homoge2}~(2), we have
\[ \| \mufh_{n,m} - \muf_{n,m}\|_{\infty} \le \DD(\mufh_{n,m}) \le \DD''_{\Gg,\Ng}(\psf_{n,m}).\]
Hence, we obtain that
\begin{align*}
\frac{|\muf_{n,m}(\xl_1 + \cdots + \xl_k)|}{\DD''_{\Gg,\Ng}(\psf_{n,m})} & = \frac{|\muf_{n,m}(\xl_1^n + \cdots + \xl_k^n)|}{n \DD''_{\Gg,\Ng}(\psf_{n,m})} \\
& \sim_{n^{-1}} \frac{|\mufh_{n,m}(\xl_1^n+ \cdots + \xl_k^n)|}{n \DD''_{\Gg,\Ng}(\psf_{n,m})} \\
& = \frac{|\psf_{n,m}(\xl_1^n+ \cdots + \xl_k^n)|}{n \DD''_{\Gg,\Ng}(\psf_{n,m})} \\
& \sim_{m^{-1}} \sup_{\psf \in \rQQQ_{\Ng} \setminus \HHH^1_{\Ng}} \frac{|\psf(\xl_1^n + \cdots + \xl_k^n)|}{n \DD''_{\Gg,\Ng}(\psf)}.
\end{align*}
Therefore we conclude that
\begin{align*} \lim_{n \to \infty} \frac{|\muf_{n,n}(\xl_1 + \cdots + \xl_k)|}{\DD''_{\Gg,\Ng}(\psf_{n,n})} &= \lim_{n \to \infty} \left(\sup_{\psf \in \rQQQ_{\Ng} \setminus \HHH^1_{\Ng}}  \frac{|\psf(\xl_1^n + \cdots + \xl_k^n)|}{n \DD''_{\Gg,\Ng}(\psf)}\right) \\
&= 4 \scl_{\Gg,\Ng}(\xl_1 + \cdots + \xl_k)
\end{align*}
and
\begin{align*}
\scl_{\Gg,\Ng}(c) &= \scl_{\Gg,\Ng}(\xl_1 + \cdots + \xl_k) \\
& \le \frac{1}{4} \sup_{\psf \in \rQQQ_{\Ng} \setminus \HHH^1_{\Ng}} \frac{(\psf|_{\Ng})_{\mathrm{h}}(\xl_1 + \cdots + \xl_k)}{\DD''_{\Gg,\Ng}(\psf)} \\
& \le \frac{1}{4} \sup_{\psf \in \rQQQ_{\Ng} \setminus \HHH^1_{\Ng}} \frac{(\psf|_{\Ng})_{\mathrm{h}}(\xl_1 + \cdots + \xl_k)}{\DD(\psf|_{\Ng})} \\
& \le \frac{1}{2} \sup_{\psf \in \rQQQ_{\Ng} \setminus \HHH^1_{\Ng}} \frac{(\psf|_{\Ng})_{\mathrm{h}}(\xl_1 + \cdots + \xl_k)}{\DD((\psf|_{\Ng})_{\mathrm{h}})} \\
& \le \frac{1}{2} \sup_{\muf \in \QQQ(\Ng)^{\Gg} \setminus \HNRG} \frac{|\muf(\xl_1 + \cdots + \xl_k)|}{\DD(\muf)} \\
& = \frac{1}{2} \sup_{\muf \in \QQQ(\Ng)^{\Gg} \setminus \HNRG} \frac{|\mu(c)|}{\DD(\muf)}.
\end{align*}
Here, for $\psf\in \QGN$, $(\psf|_{\Ng})_{\mathrm{h}}$ means the homogenization of $\psf|_{\Ng}$ (recall Corollary~\ref{cor=homoge_qinv}); by Corollary~\ref{cor=defect2bai}, we have $\DD((\psf|_{\Ng})_{\mathrm{h}}) \le 2 \DD(\psf|_{\Ng}) \le 2 \DD''_{\Gg,\Ng}(\psf)$.
\end{proof}

We have defined $\scl_{\Gg,\Ng}$ on $\CQ(\Gg, \Ng)$, but we have not defined $\scl_{\Gg,\Ng}$ on $\CR(\Gg, \Ng)$. However, the right-hand side of Theorem~\ref{generalized mixed bavard rational} can be interpreted on $\CR(\Gg,\Ng)$, and thus we define $\scl_{G,N}$ on $\CR(\Gg,\Ng)$ as follows.

\begin{definition} \label{definition=scl_real}
Let $\Gg$ be a group and $\Ng$ its normal subgroup.
For $c \in \CR(\Gg,\Ng)$, define the stable mixed commutator length $\scl_{\Gg,\Ng}(c)$ by
\[ \scl_{\Gg,\Ng}(c) = \sup_{\muf \in \QQQ(N)^G \setminus \HNRG} \frac{|\muf(c)|}{2\DD(\muf)}.\]
\end{definition}

Now  we are ready to close up this section with the following statement.

\begin{theorem}[generalized mixed Bavard duality theorem] \label{generalized mixed bavard}
Let $\Gg$ be a group and $\Ng$ its normal subgroup.
Then the definition of $\scl_{\Gg,\Ng}$ on $\CR(\Gg,\Ng)$ in Definition~\textup{\ref{definition=scl_real}}:
\[ \scl_{\Gg,\Ng}(c) =\sup_{\muf \in \QQQ(N)^G \setminus \HNRG} \frac{|\muf(c)|}{2\DD(\muf)}
\]
for every $c\in \CR(\Gg,\Ng)$, is consistent with that on $\CQ(\Gg,\Ng)$.
\end{theorem}

\begin{proof}
This follows from the generalized mixed Bavard duality theorem for rational chains (Theorem~\ref{generalized mixed bavard rational}).
\end{proof}

\section{Further properties of mixed scl for chains}\label{sec=further}
In this section, we investigate further properties of mixed scl for chains. In Subsection~\ref{subsection 3.4.5}, we show some properties of $\scl_{G,N}$ for real chains, such as another duality theorem (Theorem~\ref{thm=another}). In Subsection~\ref{subsection 3.5}, we provide a geometric interpretation of the stable mixed commutator lengths of integral chains (Proposition~\ref{proposition=interpretation_mixed_scl})  in terms of admissible $(G,N)$-simplicial surfaces.

\subsection{Another duality theorem} \label{subsection 3.4.5}

In this subsection, we show some properties of $\scl_{G,N}$ on $\CR(\Gg,\Ng)$.

\begin{lemma}
The stable mixed commutator length $\scl_{G,N}$ on $\CR(\Gg,\Ng)$ is a seminorm.
\end{lemma}
\begin{proof}
By the definition, it is clear that $\scl_{G,N}(ac) = |a| \scl_{G,N}(c)$ for $a \in \RR$ and $c \in \CR(\Gg,\Ng)$. The triangle inequality holds because
\begin{align*}
\scl_{\Gg,\Ng}(c_1 + c_2) & =  \sup_{\muf \in \QQQ(\Ng)^{\Gg} \setminus \HNRG} \frac{|\muf(c_1 + c_2)|}{2D(\muf)}\\
& \le \sup_{\muf \in \QQQ(\Ng)^{\Gg} \setminus \HNRG} \frac{|\muf(c_1)| + |\muf(c_2)|}{2D(\muf)} \\
& \le  \sup_{\muf \in \QQQ(\Ng)^{\Gg} \setminus \HNRG} \frac{|\muf(c_1)|}{2D(\muf)} + \sup_{\muf \in \QQQ(\Ng)^{\Gg} \setminus \HNRG} \frac{|\muf(c_2)|}{2D(\muf)}  \\
&= \scl_{\Gg,\Ng}(c_1) + \scl_{\Gg,\Ng}(c_2).
\end{align*}
This completes the proof.
\end{proof}

In Section~\ref{sec=gmBavard}, we study the following pairing between two $\RR$-linear spaces equipped with seminorms:
\begin{equation}\label{eq=pairing}
(\QNG/\HNRG, 2\DD)\times (\CR(\Gg,\Ng),\scl_{\Gg,\Ng})\to \RR;\ ([\muf],c)\mapsto \muf(c).
\end{equation}
The generalized mixed Bavard theorem can be interpreted that the $\scl_{\Gg,\Ng}$-seminorm of a chain $c\in \CR(\Gg,\Ng)$ is determined by the pairing \eqref{eq=pairing} (this interpretation explains the terminology of `\emph{duality theorem}'). In what follows, we prove another duality theorem (easier than Theorem~\ref{theorem=generalized_mixed_Bavard}), stating that the $2\DD$-norm of an element $[\muf]$ of $\QNG/\HNRG$ is also determined by the pairing \eqref{eq=pairing}.

\begin{theorem}[another duality theorem in the setting of the generalized mixed Bavard duality]\label{thm=another}
Let $\Gg$  be a group and $\Ng$ its normal subgroup. Then for every $\muf \in \QQQ(\Ng)^{\Gg}$, the equality
\begin{equation}\label{eq=another}
2\DD(\muf) = \sup_{c \in \CR(\Gg,\Ng),\ \scl_{G,N}(c) \ne 0} \frac{|\muf(c)|}{\scl_{\Gg,\Ng}(c)}
\end{equation}
holds true.
\end{theorem}

\begin{proof}
It is easy to see that the right-hand side of \eqref{eq=another} equals $0=2\DD(\muf)$ if $\muf\in \HNRG$. Hence, in what follows, we assume that $\muf\in \QQQ(\Ng)^{\Gg} \setminus \HNRG$.

Let $c \in \CR(\Gg,\Ng)$ such that $\scl_{G,N}(c) \ne 0$. By Definition~\ref{definition=scl_real}, we have
\begin{align} \label{second duality 1}
\scl_{\Gg,\Ng}(c) \ge \frac{|\muf(c)|}{2\DD(\muf)}.
\end{align}
Hence we have
\[ 2 \DD(\muf) \ge \frac{|\muf(c)|}{\scl_{\Gg,\Ng}(c)}.\]
In what follows, we will show that
\begin{equation}\label{eq=anothersup}
2\DD(\muf) \leq  \sup_{c \in \CR(\Gg,\Ng),\ \scl_{G,N}(c) \ne 0} \frac{|\muf(c)|}{\scl_{\Gg,\Ng}(c)}.
\end{equation}
Let $\varepsilon$ be an arbitrary positive number such that $D(\mu) - \varepsilon > 0$. Then by the definition of $\DD(\muf)$, there exist $\xl_1, \xl_2 \in \Ng$ such that
\[
\DD(\mu) - \varepsilon \le |\muf(\xl_1\xl_2)-\muf(\xl_1)-\muf(\xl_2)|.
\]
Set $c_{\varepsilon}\in \CR(\Gg,\Ng)$ by $c_{\varepsilon}=\xl_1\xl_2-\xl_1-\xl_2$. Note that by Lemma~\ref{lem=commutatorcal}, 
\[
0<\scl_{\Gg,\Ng}(c_{\varepsilon})\leq \frac{1}{2}.
\]
Hence, we have
\[
2 (\DD(\muf) - \varepsilon) \le \frac{|\mu(c_{\varepsilon})|}{\scl_{\Gg,\Ng}(c_{\varepsilon})} .
\]
By letting $\varepsilon\searrow 0$, we obtain \eqref{eq=anothersup}. 
By combining \eqref{second duality 1} and \eqref{eq=anothersup}, we complete the proof of \eqref{eq=another}.
\end{proof}

\begin{remark}
In fact, we have the following refinement of \eqref{eq=another}:
\begin{equation}\label{eq=another2}
2\DD(\muf) = \sup_{\xl_1,\xl_2 \in \Ng,\ \scl_{G,N}([\xl_1,\xl_2]) \ne 0} \frac{|\muf([\xl_1,\xl_2])|}{\scl_{\Gg,\Ng}([\xl_1,\xl_2])}.
\end{equation}
To see \eqref{eq=another2}, let $\varepsilon \in (0,D(\mu))$. By Proposition~\ref{prop=commutator_bavard} there exist $\xl_1, \xl_2 \in \Ng$ such that
\[ \DD(\mu) - \varepsilon \le |\muf([\xl_1, \xl_2])|.\]
Then for such a pair $(\xl_1, \xl_2)$,  $\muf([\xl_1, \xl_2])$ is non-zero, and  the mixed Bavard duality theorem implies that
\[
0 < \scl_{\Gg,\Ng}([\xl_1, \xl_2]) \le \frac{1}{2}.
\]
Hence, we have
\[ 2 (\DD(\muf) - \varepsilon) \le \frac{|\mu([\xl_1,\xl_2])|}{\scl_{\Gg,\Ng}([\xl_1,\xl_2])}.\]
By letting $\varepsilon\searrow 0$, we have
\begin{align} \label{second duality 2}
2\DD(\muf) \le \sup_{\xl_1,\xl_2 \in \Ng,\ \scl_{G,N}([\xl_1,\xl_2]) \ne 0} \frac{|\muf([\xl_1,\xl_2])|}{\scl_{\Gg,\Ng}([\xl_1,\xl_2])}.
\end{align}
Now, \eqref{second duality 1} and \eqref{second duality 2} yield \eqref{eq=another2}.
\end{remark}

\subsection{Geometric interpretation of $\scl_{\Gg,\Ng}$} \label{subsection 3.5}
A geometric interpretation of the ordinary $\scl$ of integral chains  in terms of admissible surfaces  is known (\cite[Proposition~2.74]{Calegari}). Here we provide a geometric interpretation of the mixed $\scl$ of integral chains (Proposition~\ref{proposition=interpretation_mixed_scl})  in terms of admissible $(G,N)$-simplicial surfaces (Definition~\ref{definition=admissible}).

The following result is well known.  Recall from the introduction that we assume that surfaces are orientable, unless otherwise stated.

\begin{lemma}[see {\cite[Lemma~1.12]{Calegari}}] \label{lemma Calegari cover}
Let $m,p$ be integers at least $2$ such that $m$ and $p-1$ are coprime. Let $S$ be an \textup{(}orientable\textup{)} connected compact surface with $p$ boundary components. Then there exists an $m$-fold cover $\hat{S}$ over $S$ whose number of boundary components is $p$.
\end{lemma}

\begin{definition}[see \cite{Calegari}]
For a compact orientable surface $S$, set
\[ \chi^{-}(S) = \sum \chi(S').\]
Here, the sum in the right-hand side of the equality is taken for every connected component $S'$ which is not homeomorphic to a sphere.
\end{definition}

\begin{proposition}[geometric interpretation of the stable mixed commutator lengths of integral chains]\label{proposition=interpretation_mixed_scl}
Let $\Gg$ be a group and $\Ng$ its normal subgroup. Let $m\in \NN$ and  $x_1, \cdots, x_m \in N$. Assume that $x_1 \cdots x_m \in [G,N]$. Then we have
\[ \scl_{G,N}(x_1 + \cdots + x_m) = \inf_S \frac{- \chi^- (S)}{2n(S)}.\]
Here, in the supremum in the right-hand side of the equality, $S$ runs through all admissible $(G,N)$-simplicial surfaces with respect to $x_1, \cdots, x_m$.
\end{proposition}
\begin{proof}
Set $\cl'_{\Gg,\Ng}(\xl_1 + \cdots + \xl_m) = \cl_{\Gg,\Ng}(\xl_1 + \cdots + \xl_m) + (m-1)$. Note that Lemma~\ref{lemma subadditive cl}, together with Fekete's lemma (Lemma~\ref{lem=Fekete}), implies that
\begin{align*}
\scl_{\Gg,\Ng}(\xl_1 + \cdots + \xl_m) &= \inf_n \frac{\cl'_{\Gg,\Ng}(\xl_1^n + \cdots + \xl_m^n)}{n} \\
&= \inf_n \frac{\cl_{\Gg,\Ng}(\xl_1^n + \cdots + \xl_m^n) + (m-1)}{n}.
\end{align*}

If $\cl_{\Gg,\Ng}(\xl_1^n + \cdots + \xl_m^n) = k$, then there exists an admissible $(\Gg,\Ng)$-simplicial surface $S_0$ with respect to $\xl_1, \cdots, \xl_m$ such that $n(S_0) = n$. Since $2k-1 = - \chi^-(S_0)$, we have
\[ \frac{\cl_{\Gg,\Ng}(\xl_1^n + \cdots + \xl_m^n)}{n} = \frac{k}{n} \ge \frac{2k-1}{2n} = \frac{- \chi^-(S_0)}{2 n(S_0)} \geq  \inf_S \frac{- \chi^-(S)}{2n(S)}.\]
Thus, by taking the limits as $n\to \infty$, we have
\[ \scl_{\Gg,\Ng}(\xl_1 + \cdots + \xl_m) \ge \inf_S \frac{- \chi^-(S)}{n(S)}.\]

Next we show the inequality
\[ \scl_{\Gg,\Ng}(\xl_1 + \cdots + \xl_m) \le \inf_S \frac{- \chi^-(S)}{n(S)}.\]
Let $S$ be an admissible $(\Gg,\Ng)$-simplicial surface with respect to $\xl_1, \cdots, \xl_m$. Let $S_1, \cdots, S_n$ be the connected components of $S$. Without loss of generality, we may assume that $S_j$ for each $j$ has a boundary and is not a sphere. Let $S'_j$ be the connected double cover of $S_j$ such that $\partial S'_j$ is homeomorphic to  $\partial S_j \sqcup \partial S_j$.
 Let $M$ be an integer at least two which is coprime to $\# \pi_0(\partial S_j) - 1$ for every $j \in \{ 1, \cdots, n\}$.
By Lemma~\ref{lemma Calegari cover}, there exists an $M$-fold connected cover $T_j$ of $S'_j$ such that the number of boundary components of $T_j$ coincides with that of $S'_j$
. Set $T = \coprod\limits_{i=1}^n T_j$. Then $T$ is a $2M$-fold cover of $S$. The triangulation of $S$ gives rise to the triangulation of $T$, and the $(\Gg,\Ng)$-labelling of $S$ gives rise to the $(G,N)$-labelling of $T$. Thus $T$ is a $(\Gg,\Ng)$-simplicial surface.

Now we would like to obtain an admissible $(\Gg,\Ng)$-surface $T'$ with respect to $\xl_1, \cdots, \xl_m$ by attaching $2$-simplices to $T$ in such a way that $n(T') = 2 M n(S)$. Let $T'$ be the $(\Gg,\Ng)$-simplicial surface obtained by taking the following two operations.

\begin{itemize}
\item Each boundary component of $T$ consists of $M$ edges labelled by $\xl_i$. Then attach $2$-simplices depicted in the left of Figure~\ref{figure Mx}.

\begin{figure}[t]
\centering
\includegraphics[width=9cm]{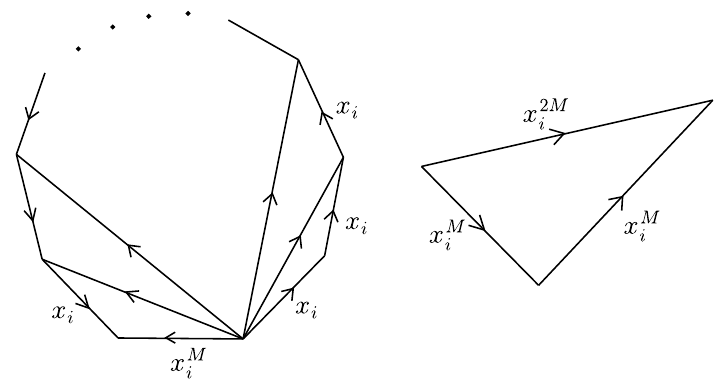}
\caption{$Mx_i$ to $x_i^M$ and $2x_i^M$ to $x_i^{2M}$} \label{figure Mx}
\end{figure}

\item After that, there exist two boundaries labelled by $\xl_i^M$. Then identify the vertices of these boundaries and attach a $2$-simplex in such a way that the boundary has one edge labelled by $\xl_i^{2M}$ (see the right of Figure~\ref{figure Mx}). Note that this operation changes $\chi^-$ at most $1$. Hence if we do this operations for all $i$, then $\chi^-$ changes at most $m$.
\end{itemize}

Then $n(T') = 2M n(S)$. Note that $\chi^-(T) = 2M \chi^-(S)$ and $\chi^-(T')$ differs from $\chi^-(T)$ at most $m$. Hence we have
\begin{align*}
\frac{\cl_{\Gg,\Ng}(\xl_1^{2M} + \cdots + \xl_m^{2M})}{2M} &\ge \frac{- \chi^-(T') - m}{2n(T)} \\
&\ge \frac{-2M \chi^-(S) - 2m}{4M n(S)} \\
&= \frac{- \chi^- (S)}{2n(S)} - \frac{m}{2 M n(S)}.
\end{align*}
Since $M$ can be taken to be arbitrary large, we have
\[ \scl_{\Gg,\Ng}(\xl_1 + \cdots + \xl_m) \ge \inf_S \frac{-\chi^-(S)}{n(S)}.\]
This completes the proof.
\end{proof}




\section{The space $\mathrm{W}(G,N)$ of non-extendable quasimorphisms}
\label{sec=W}

In this section and the next section, if we write cochain groups and group cohomology (ordinary and bounded)  in the setting of Subsection~\ref{subsec=cohomology} for $A=\RR$ (equipped with the standard absolute value $|\cdot|$), then we omit to indicate the coefficient group $A$. For instance, we abbreviate $\HHH^*(-;\RR)$ as $\HHH^*(-)$ and $\HHH^*_b(-;\RR)$ as $\HHH^*_b(-)$, respectively. As mentioned in the introduction, the space $\WGN$ is defined as follows.

\begin{definition}[{\cite{KKMMM1,KKMMMsurvey}}]\label{defn=W}
Let $\Gg$ be a group and $\Ng$ its normal subgroup. Then, we define
\[
\WGN=\Coker \Big(\overline{i^{\ast}}\colon \QG/\HG\to \QNG/\HNG \Big).
\]
In other words, by using the $\RR$-linear map $i^{\ast}\colon \QG\to \QNG$, we can write\begin{equation}\label{eq=nonext}
\WGN=\QNG /\big(\HNG +i^{\ast}\QG\big).
\end{equation}
\end{definition}

Here, we note that $i$ induces $i^{\ast}\colon \QG\to \QNG$ since $\QG=\QG^{\Gg}$ (Lemma~\ref{lem=invariant}). By \eqref{eq=nonext}, $\WGN$ may be regarded as a space of `unobvious' $\Gg$-invariant homogeneous quasimorphisms in the following sense: elements in $\HNG$ are `obvious' as quasimorphisms, for they may be algebraically understood. Also, an element $\muf$ of  $\QNG$ belongs to $i^{\ast}\QG$ if and only if $\muf$ is \emph{extendable} to a homogeneous quasimorphism on $\Gg$; in this case, $\muf$ is $\Gg$-invariant for an `obvious' reason because of  Lemma~\ref{lem=invariant}.

\subsection{$\mathrm{W}(G,N)$ and scl}\label{subsec=Wscl}


In this subsection, we discuss some relationship between $\WGN$ and the comparison problem of $\scl_{\Gg}$ and $\scl_{\Gg,\Ng}$ on $\CGN$.
The first topic of this subsection is the following theorem. Proposition~\ref{prop=bilip_criterion} below is one of the keys; this proposition may be regarded as an application of the mixed Bavard duality theorem (Theorem~\ref{theorem=mixed_Bavard}). We refer the reader to \cite[Section~8]{KKMMMsurvey} and \cite{KKMMMcg} for further directions on this comparison problem.

\begin{theorem}[{\cite[Theorem 2.1 (2)]{KKMMM1}}, comparison between $\scl_{\Gg}$ and $\scl_{\Gg,\Ng}$]\label{thm=comparison_amenable}
Let $\Gg$ be a group and $\Ng$ its normal subgroup. Assume that $\WGN=0$ and that $\Gam=\Gg/\Ng$ is amenable. Then, for every $\yl \in \CGN$ we have
\begin{equation}\label{eq=tkdk2bai}
\scl_{\Gg}(\yl)\leq \scl_{\Gg,\Ng}(\yl)\leq 2\scl_{\Gg}(\yl).
\end{equation}
\end{theorem}

Abelian groups, finite groups and groups of subexponential growth are amenable; the class of (discrete) amenable groups is closed under taking subgroups, group quotients, inductive limits, and group extensions. In particular, a virtually solvable group, meaning that a group admitting a solvable subgroup of finite index, is  amenable. See \cite{Paterson} for more details on amenability.

\begin{remark}
We make the following two remarks on the inequalities \eqref{eq=tkdk2bai}.
\begin{enumerate}[(1)]
  \item The multiplicative factor `$2$' on the right-hand side of \eqref{eq=tkdk2bai} comes from Lemma~\ref{lemma=qm_defect_seminorm} below. To the best knowledge of the authors, it might not be known whether this multiplicative factor is optimal.
  \item In \cite[Theorem~2.1~(3)]{KKMMM1}, the authors showed that  $\scl_{\Gg,\Ng}$ in fact coincides with $\scl_{\Gg}$ on $\CGN$ if $\Gam=\Gg/\Ng$ is solvable.
\end{enumerate}
\end{remark}

In what follows, we prove Theorem~\ref{thm=comparison_amenable}. In this subsection, for an extendable $\Gg$-invariant homogeneous quasimorphism $\nuf\in \QNG$, that is $\nuf$ belongs to $i^{\ast}\QG$, we take an extension $\phf\in \QG$ (in other words, $i^{\ast}\phf=\nuf$) and make a comparison between the defect of $\phf$ and that of $\nuf$. In order to distinguish these two types of defects, in this subsection we use the symbol $\DDG$ for the defect on  $\QG$ (or, more generally on $\QhG$) and the symbol $\DDN$ for that on $\QNG$.

\begin{proposition}[{\cite[Proposition 7.2]{KKMMM1}}]\label{prop=bilip_criterion}
  Let $C$ be a real positive number, and assume that for every $\muf \in \QNG$, there exists $\phf \in \QG$ such that $\phf|_{\Ng} - \muf \in \HNG$ and $\DDG(\phf) \leq C \cdot \DDN(\muf)$.
  Then for every $\yl \in \CGN$,
  \begin{align*}
    \scl_{\Gg}(\yl) \leq \scl_{\Gg, \Ng} (\yl) \leq C \scl_{\Gg}(\yl).
  \end{align*}
\end{proposition}

\begin{proof}
  Let $\yl$ be an element of $\CGN$.
  The inequality $\scl_{\Gg}(\yl) \leq \scl_{\Gg, \Ng}(\yl)$ is clear.

  Let $\varepsilon$ be a positive real number.
  Then by Theorem \ref{theorem=mixed_Bavard}, there exists $\muf \in \QNG$ such that
  \begin{align*}
    \sclGN(y) - \varepsilon \leq \frac{|\muf(\yl)|}{2\DDN(\muf)}.
  \end{align*}
  By assumption, there exists $\phf \in \QG$ such that $\phf|_{\Ng}-\muf \in \HNG$ and $\DDG(\phf) \leq C \cdot \DDN(\muf)$.
  Since $\yl \in \CGN$, we have $(\phf|_{\Ng}-\muf)(\yl) = 0$.
  Hence we obtain
  \begin{align*}
    \sclGN(\yl) - \varepsilon \leq \frac{|\muf(\yl)|}{2\DDN(\muf)} \leq C\cdot \frac{|\phf(\yl)|}{2\DDG(\phf)} \leq C \sclG(\yl),
  \end{align*}
  where the last inequality comes from the easy part (Corollary~\ref{cor=sclabove}) of Theorem \ref{theorem=Bavard}.
\end{proof}

Recall that the $\ell^\infty$-norm $\|\cdot\|_{\infty}$ on $\CCC_b^n(\Gg)$ is defined as
\begin{align*}
  \| u \|_{\infty} = \sup_{\gl_i \in G}\{ |u(\gl_1, \cdots, \gl_n)| \}
\end{align*}
for $u\in \CCC_b^n(\Gg)$.
This norm induces a seminorm $\| \cdot \|$ on $\HHH_b^n(\Gg)$, that is,
\begin{equation}\label{eq=seminorm_bcohom}
  \| \alpha \| = \inf_{u \in \alpha} \{ \| u \|_{\infty} \}
\end{equation}
for $\alpha \in \HHH_b^n(G)$. We use the following well-known estimate.

\begin{lemma}\label{lemma=qm_defect_seminorm}
  For $\phf \in \QG$, the inequalities
  \begin{align*}
    \| [\delta \phf] \| \leq \DDG(\phf) \leq 2 \| [\delta \phf] \|
  \end{align*}
  hold, where $\| \cdot \|$ is the seminorm on $\HHH_b^2(\Gg)$ defined in \eqref{eq=seminorm_bcohom}.
\end{lemma}

\begin{proof}
  For $\phf \in \QG$, the defect $\DDG(\phf)$ is nothing but $\| \delta \phf \|_{\infty}$, and hence we have $\| [\delta \phf] \| \leq \DDG(\phf)$.

  Let $u$ be a bounded cocycle representing $[\delta \phf]$.
  Then there exists a unique $\phfh \in \QhG$ such that $u = \delta \phfh$ and $\phfh - \phf \in \CCC_b^1(\Gg)$.
  Indeed, since $[u] = [\delta \phf]$, there exists $b \in \CCC_b^1(\Gg)$ such that $u = \delta \phf + \delta b = \delta (\phf + b)$.
  Hence $\phf + b$ is a desired quasimorphism.
  Let $\phfh'$ be another quasimorphism satisfying $u = \delta \phfh'$ and $\phfh' - \phf \in \CCC_b^1(\Gg)(=\ell^{\infty}(\Gg;\RR))$.
  Then $\phfh' - \phfh$ is a homomorphism since $\delta (\phfh' - \phfh) = 0$, and is also a bounded function since $\phfh' - \phfh = (\phfh' - \phf) - (\phfh - \phf)$.
  Hence $\phfh' = \phfh$.

  Now we have
  \begin{align*}
    \| [\delta \phf] \| &= \inf \{ \| u \|_{\infty} \mid [u] = [\delta \phf] \} \\
&= \inf \{ \| \delta \phfh \|_{\infty} \mid \phfh - \phf \in \CCC_b^1(\Gg) \} \\
&= \inf \{\DDG(\phfh) \mid \phfh - \phf \in \CCC_b^1(\Gg) \}.
  \end{align*}
Take an arbitrary element $\phfh$ of $\QhG$ such that $\phfh - \phf \in \CCC_b^1(\Gg)$. Then the homogenization $\phfh_{\mathrm{h}}$ of $\phfh$ equals $\phf$. Therefore, by Corollary~\ref{cor=defect2bai} we obtain that
\[
\DDG(\phf)=\DDG(\phfh_{\mathrm{h}})\leq 2\DDG(\phfh).
\]
Thus, we conclude that  $\DDG(\phf) \leq 2 \| [\delta \phf] \|$. This completes our proof.
\end{proof}

We employ the following theorem, where we use the assumption of amenability of $\Gam$ in the proof of  Theorem \ref{thm=comparison_amenable}.

\begin{theorem}[{\cite[Proposition 8.6.6]{Monod}}]\label{theorem_isometric_isom}Let $  1 \to \Ng \xrightarrow{i} \Gg \to \Gam \to 1$ be a short exact sequence of groups, and assume that  $\Gam$ is amenable. Then, the map $\HHH_b^2(\Gg) \to \HHH_b^2(\Ng)^{\Gg}$, induced by $i$, is an isometric isomorphism.
\end{theorem}

\begin{proof}[Proof of Theorem~\textup{\ref{thm=comparison_amenable}}]
Let $\muf\in \QNG$. Since $\WGN=0$, we can decompose $\muf$ as $\muf=\hf+i^{\ast}\phf$, where $\hf\in \HNG$ and $\phf\in \QG$. Then,   by Lemma \ref{lemma=qm_defect_seminorm} and Theorem \ref{theorem_isometric_isom}, we have
  \begin{align*}
    2 \DDN(\muf)&=2 \DDN(\muf-\hf)=2 \DDN(\phf|_{\Ng}) \\
&\geq 2\|[\delta\phf|_{\Ng}] \| = 2 \| [\delta \phf] \| \\
&\geq \DDG(\phf).
  \end{align*}
Therefore, by Proposition \ref{prop=bilip_criterion} we complete our proof.
\end{proof}

The second topic of this subsection is the following vanishing result (Proposition~\ref{prop=vsection}) of $\WGN$. Proposition~\ref{prop=vsection} suggests that the study of $\WGN$ should be related to that of the short exact sequence \eqref{eq=shortexact}; Corollary~\ref{cor=vsplitscl} relates virtually splitting short exact sequences to the comparison problem between $\scl_{\Gg}$ and $\scl_{\Gg,\Ng}$ on $\CGN$.

\begin{definition}[virtual splitting]\label{defn=vsection}
Let
\begin{equation}\label{eq=shortexact}
  1 \to \Ng \xrightarrow{i} \Gg \xrightarrow{p} \Gam \to 1
\end{equation}
be a short exact sequence of groups.
\begin{enumerate}[(1)]
 \item We say that $(s,\Lambda)$ is a \emph{virtual section} of $p$ if $\Lambda$ is a subgroup of $\Gam$ of finite index and $s$ is a group homomorphism from $\Lambda$ to $\Gg$ satisfying that $p\circ s=\mathrm{id}_{\Lambda}$.
 \item We say that the sequence \eqref{eq=shortexact} \emph{virtually splits} if a virtual section of $p$ exists.
\end{enumerate}
\end{definition}

\begin{proposition}[{\cite[Proposition~1.6]{KKMM1}}]\label{prop=vsection}
Assume that a short exact sequence
\[
  1 \to \Ng \xrightarrow{i} \Gg \xrightarrow{p} \Gam \to 1
\]
virtually splits. Then, for every $\muf\in \QNG$ there exists $\phf\in \QQQ(\Gg)$ such that $i^{\ast}\phf=\muf$ and $\DDG(\phf)\leq 2\DDN(\muf)$.

In particular, $\QQQ(\Ng)^{\Gg}=i^{\ast}\QQQ(\Gg)$ and $\WGN=0$.
\end{proposition}

\begin{proof}
Take a virtual section $(s,\Lambda)$ of $p$. Take a finite subset $B$ of $\Gam$ such that the map $\Lambda\times B \to \Gam\colon (\lambda,b)\mapsto \lambda b$ is bijective. Take a map $s'\colon B\to \Gg$ such that for every $b\in B$, $(p\circ s')(b)=b$ holds. Then, define a map $t\colon \Gam \to \Gg$ by setting $t(\gamma)=s(\lambda)s'(b)$ for every $\gamma \in \Gam$, where we write $\gamma=\lambda b$ for a unique $(\lambda,b)\in \Lambda\times B$. Note that $t$ is a set-theoretic section of $p$. We observe that for every $\lambda\in \Lambda$ and for every $\gamma\in \Gamma$,
\begin{equation}\label{eq=lefthom}
t(\lambda \gamma)=t(\lambda)t(\gamma)
\end{equation}
holds: this follows from the definition of $t$ and the fact that $s$ is a group homomorphism.

Define a map $\phfh\colon \Gg\to \RR$ by
\[
\phfh(\gl)=\frac{1}{\# B}\sum_{b\in B}\muf\big(\gl\cdot t(b\cdot p(\gl))^{-1}\cdot t(b) \big)
\]
for every $\gl\in \Gg$; here, note that $\gl\cdot t(b\cdot p(\gl))^{-1}\cdot t(b)\in \Ng$ because $p(\gl\cdot t(b\cdot p(\gl))^{-1}\cdot t(b))=p(\gl)\cdot ((b\cdot p(\gl))^{-1}\cdot b=1_{\Gam}$.
First, by construction, $\phfh|_{\Ng}=\muf$ holds. Secondly, we will show that $\phfh\in \QhG$. Let $\gl_1,\gl_2\in \Gg$. Then, we  have
\begin{align*}
\phfh(\gl_1\gl_2)&=\frac{1}{\# B}\sum_{b\in B}\muf\big( \gl_1\gl_2 \cdot t(b\cdot p(\gl_1\gl_2))^{-1}\cdot t(b)\big)\\
&=\frac{1}{\# B}\sum_{b\in B}\muf\big(t(b)\cdot  \gl_1\gl_2 \cdot t(b\cdot p(\gl_1\gl_2))^{-1}\big)\\
&=\frac{1}{\# B}\sum_{b\in B}\muf\big(t(b)\cdot  \gl_1\cdot t(b\cdot p(\gl_1))^{-1} \cdot  t(b\cdot p(\gl_1))\cdot \gl_2 \cdot t(b\cdot p(\gl_1\gl_2))^{-1}\big)\\
&\sim_{\DDN(\muf)} \frac{1}{\# B}\sum_{b\in B}\muf\big(t(b)\cdot  \gl_1\cdot t(b\cdot p(\gl_1))^{-1} \big) \\
&\qquad +\frac{1}{\# B}\muf\big( t(b\cdot p(\gl_1))\cdot \gl_2 \cdot t(b\cdot p(\gl_1\gl_2))^{-1}\big)\\
&=\phfh(\gl_1)+ \frac{1}{\# B}\sum_{b\in B}\muf\big( t(b\cdot p(\gl_1))\cdot \gl_2 \cdot t(b\cdot p(\gl_1\gl_2))^{-1}\big)\\
&=\phfh(\gl_1)+ \frac{1}{\# B}\sum_{b\in B}\muf\big( \gl_2 \cdot t(b\cdot p(\gl_1\gl_2))^{-1}\cdot t(b\cdot p(\gl_1))\big).
\end{align*}
In what follows, we treat the term
\begin{align*}
&\frac{1}{\# B}\sum_{b\in B}\muf\big( \gl_2 \cdot t(b\cdot p(\gl_1\gl_2))^{-1}\cdot t(b\cdot p(\gl_1))\big)\\
&=\frac{1}{\# B}\sum_{a\in B}\muf\big( \gl_2 \cdot t(a\cdot p(\gl_1\gl_2))^{-1}\cdot t(a\cdot p(\gl_1))\big).
\end{align*}
Here, recall that $\muf$ is $\Gg$-invariant. For each $a\in B$, write $a\cdot p(\gl_1)=\lambda_{a} b_a$ for (unique) $(\lambda_a,b_a)\in \Lambda\times B$. Then, by \eqref{eq=lefthom} we obtain that
\begin{align*}
&\frac{1}{\# B}\sum_{a\in B}\muf\big( \gl_2 \cdot t(a\cdot p(\gl_1\gl_2))^{-1}\cdot t(a\cdot p(\gl_1))\big)\\
&=\frac{1}{\# B}\sum_{a\in B}\muf\big( \gl_2 \cdot t(a\cdot p(\gl_1)p(\gl_2))^{-1}\cdot t(a\cdot p(\gl_1))\big)\\
&=\frac{1}{\# B}\sum_{a\in B}\muf\big( \gl_2 \cdot t(\lambda_a b_a\cdot p(\gl_2))^{-1}\cdot t(\lambda_a b_a)\big)\\
&=\frac{1}{\# B}\sum_{a\in B}\muf\big( \gl_2 \cdot t(b_a\cdot p(\gl_2))^{-1}\cdot t(b_a)\big)\\
&=\phfh(g_2).
\end{align*}
Here, observe that the map $B\to B$ sending  $a\in B$ to $b_a$ is a bijection. Therefore, we conclude that $\phfh \in \QhG$; in fact, we have $\DDG(\phfh)\leq \DDN(\muf)$ (this inequality implies that $\DDG(\phfh)= \DDN(\muf)$ since $\phfh|_{\Ng}=\muf$).

Finally, by taking the homogenization $\phfh_{\mathrm{h}}$ of $\phfh$, we have $\phfh_{\mathrm{h}}\in \QQQ(\Gg)$ and $i^{\ast}\phfh_{\mathrm{h}}=\muf$. Moreover, we obtain that $\DDG(\phf)\leq 2\DDG(\phfh)= 2\DDN(\muf)$ by Corollary~\ref{cor=defect2bai}. This completes our proof.
\end{proof}

Proposition~\ref{prop=vsection} yields the following two corollaries.

\begin{corollary}[{\cite[Theorem~6.2]{KKMM1}}]\label{cor=vsplitscl}
Assume that a short exact sequence
\[
  1 \to \Ng \to \Gg \to \Gam \to 1
\]
virtually splits. Then, for every $\yl \in \CGN$ we have
\[
\scl_{\Gg}(\yl)\leq \scl_{\Gg,\Ng}(\yl)\leq 2\scl_{\Gg}(\yl).
\]
\end{corollary}

\begin{proof}
This corollary follows  from Proposition~\ref{prop=bilip_criterion} and the main assertion of  Proposition~\ref{prop=vsection}.
\end{proof}

\begin{corollary}[{\cite{KKMM1}}]\label{cor=vfree}
Let $\Gg$ be a group and $\Ng$ its normal subgroup. Assume that $\Gam=\Gg/\Ng$ is virtually free, meaning that $\Gam$ contains a free subgroup $($possibly the trivial subgroup$)$ of finite index. Then, $\WGN=0$ holds, and we have
\[
\scl_{\Gg}(\yl)\leq \scl_{\Gg,\Ng}(\yl)\leq 2\scl_{\Gg}(\yl)
\]
for every $\yl\in \CGN$.
\end{corollary}

\begin{proof}
Take a free subgroup $\Lambda$ of $\Gam$ of finite index. Then, we can construct $s\colon \Lambda\to \Gg$ such that $(s,\Lambda)$ is a virtual section of $p\colon  \Gg \to \Gam$. By Proposition~\ref{prop=vsection}, we have $\WGN=0$; by Corollary~\ref{cor=vsplitscl} we complete our proof.
\end{proof}

For instance, we can apply Corollary~\ref{cor=vfree} to the case where $\Gam=\mathrm{SL}(2,\ZZ)$, $\mathrm{GL}(2,\ZZ)$, or $\Gam$ is a free product of two finite groups.

\subsection{Five-term exact sequences}\label{subsec=five-term}
For a short exact sequence of groups
\[
  1 \to \Ng \xrightarrow{i} \Gg \xrightarrow{p} \Gam \to 1,
\]
cohomology groups of low degrees form the following exact sequence:
\begin{align}\label{align=five-term-grp-coh}
  0 \to \HHH^1(\Gam) \xrightarrow{p^*} \HHH^1(\Gg) \xrightarrow{i^*} \HHH^1(\Ng)^{\Gg} \to \HHH^2(\Gam) \xrightarrow{p^*} \HHH^2(\Gg).
\end{align}
This is called the \emph{five-term exact sequence of group cohomology}.

In \cite{KKMMM1}, the authors of the present article established an analogous exact sequence related to the space of homogeneous quasimorphisms:
\begin{align}\label{align=five-term-rel-coh}
  0 \to \QQQ(\Gam) \xrightarrow{p^*} \QQQ(\Gg) \xrightarrow{i^*} \QQQ(\Ng)^{\Gg} \xrightarrow{\tau_{/b}} \HHH_{/b}^2(\Gam) \xrightarrow{p^*} \HHH_{/b}^2(\Gg).
\end{align}
Here $\HHH_{/b}^*(G)$ is the cohomology group of the relative complex $\CCC_{/b}^*(G) = \CCC^*(G)/\CCC_b^*(G)$.
Moreover, sequence \eqref{align=five-term-rel-coh} is compatible with \eqref{align=five-term-grp-coh}, that is, the following diagram commutes:
\begin{align}\label{align=diagram_five_term_coh_rel}
\xymatrix{
0 \ar[r] & \HHH^1(\Gam) \ar[r] \ar[d] & \HHH^1(\Gg) \ar[r] \ar[d] & \HHH^1(\Ng)^{\Gg} \ar[r] \ar[d] & \HHH^2(\Gam) \ar[r] \ar[d] & \HHH^2(\Gg) \ar[d] \\
0 \ar[r] & \QQQ(\Gam) \ar[r] & \QQQ(\Gg) \ar[r] & \QQQ(\Ng)^{\Gg} \ar[r] & \HHH_{/b}^2(\Gam) \ar[r] & \HHH_{/b}^2(\Gg).
}
\end{align}
Since $\HHH_{/b}^*(G)$ is the relative cohomology, we have exact sequences
\begin{align*}
  \cdots \to \HHH_b^2(\Gam) \xrightarrow{c_{\Gam}^2} \HHH^2(\Gam) \xrightarrow{j^*} \HHH_{/b}^2(\Gam) \xrightarrow{\mathbf{d}} \HHH_b^3(\Gam) \to \cdots
\end{align*}
and
\begin{align*}
  \cdots \to \HHH_b^2(\Gg) \xrightarrow{c_{\Gg}^2} \HHH^2(\Gg) \xrightarrow{j^*} \HHH_{/b}^2(\Gg) \xrightarrow{\mathbf{d}} \HHH_b^3(\Gg) \to \cdots.
\end{align*}
Here $j^*$ and $\mathbf{d}$ are the maps induced by the quotient map $j \colon \CCC^2(\Gam) \to \CCC_{/b}^2(\Gam)$ and the coboundary map $\delta \colon \CCC^2(\Gam) \to \CCC^3(\Gam)$, respectively.
Hence we obtain the following commutative diagram, whose rows and columns are exact:
\begin{align*}\label{align=big_diagram}
\xymatrix{
& & & & \HHH_b^2(\Gam) \ar[r] \ar[d]^-{c_{\Gamma}^2} & \HHH_b^2(\Gg) \ar[d]^-{c_{\Gg}^2} \\
0 \ar[r] & \HHH^1(\Gam) \ar[r] \ar[d] & \HHH^1(\Gg) \ar[r] \ar[d] & \HHH^1(\Ng)^{\Gg} \ar[r] \ar[d] & \HHH^2(\Gam) \ar[r]^-{p^*} \ar[d]^-{j_*} & \HHH^2(\Gg) \ar[d] \\
0 \ar[r] & \QQQ(\Gam) \ar[r] & \QQQ(\Gg) \ar[r] & \QQQ(\Ng)^{\Gg} \ar[r]^-{\tau_{/b}} & \HHH_{/b}^2(\Gam) \ar[r] \ar[d]^-{\mathbf{d}} & \HHH_{/b}^2(\Gg) \ar[d] \\
& & & & \HHH_b^3(\Gam) \ar[r] & \HHH_b^3(\Gg).
}
\end{align*}

\subsection{Dimension of $\WGN$}\label{subsec=dimW}

The diagram in the previous subsection is a useful tool to estimate the dimension of $\WGN$.
\begin{theorem}\label{theorem=dimension_no_conditions}
  The dimension of $\WGN$ is equal to
  \begin{align*}
    \dim_{\RR} \Im \, (\mathbf{d} \circ \tau_{/b}) + \dim_{\RR} \left(\frac{\Im \, c_{\Gg}^2 \cap \Im \, p^*}{\Im \, (p^* \circ c_{\Gamma}^2)} \right).
  \end{align*}
  In particular, the inequality $\dim_{\RR} \WGN \leq \dim_{\RR} \Ker \, c_{\Gamma}^3 + \dim_{\RR} \HHH^2(G)$ holds.
\end{theorem}

\begin{proof}
  The composite $\mathbf{d} \circ \tau_{/b}$ induces a map $\mathbf{d}_{W} \colon \WGN \to \Im \, \mathbf{d} \subset \HHH_b^3(\Gamma)$.
  Since $\Im \, \mathbf{d}_W = \Im \, (\mathbf{d} \circ \tau_{/b})$, it suffices to show that the kernel of $\mathbf{d}_W$ is isomorphic to $(\Im \, c_{\Gg}^2 \cap \Im \, p^*)/\Im \, (p^* \circ c_{\Gamma}^2)$.
  For an element $[\mu]$  of $\Ker \, \mathbf{d}_W$, we take an element $c$ of $\HHH^2(\Gamma)$ with $j_*(c) = \tau_{/b}(\mu)$.
  Then it is straightforward to see  that the class $p^*(c)$ gives rise to a well-defined map $f \colon \Ker \, \mathbf{d}_W \to (\Im \, c_{\Gg}^2 \cap \Im \, p^*)/\Im \, (p^* \circ c_{\Gamma}^2); [\mu] \mapsto [p^*(c)]$, which is an isomorphism.
\end{proof}

To estimate the dimension of $\Im \, (\mathbf{d \circ \tau_{/b}})$, we need some information on the third bounded cohomology group $\HHH_b^3(\Gam)$. We use the following definition of bounded $n$-acyclicity; the case of $n=3$ is of our main concern in this subsection.

\begin{definition}[bounded $n$-acyclicity]\label{defn=bdd_acyc}
Let $n\in \NN$. A group $\Gg$ is said to be \emph{boundedly $n$-acyclic} if $\HHH^i_b(\Gg) = 0$ holds for every $i\in \NN$ with $i \leq n$. We say that $\Gg$ is \emph{boundedly acyclic} if for every $n\in \NN$, $\Gg$ is boundedly $n$-acyclic.
\end{definition}

Boundedly $3$-acyclicity of $\Gam$ makes the situation a bit simpler. In fact, Theorem \ref{theorem=dimension_no_conditions} and its proof gives the following.
\begin{theorem}[\cite{KKMMM1}]\label{theorem=WGN_isom}
  If $\Gam$ is boundedly $3$-acyclic, then the map $p^* \circ j_*^{-1} \circ \tau_{/b}$ induces an isomorphism $\WGN \cong \Im \, p^* \cap \Im \, c_{\Gg}^2$.
  In particular, we have
\[
\dim_{\RR} \WGN \leq \min\{\dim_{\RR} \HHH^2(\Gg),\dim_{\RR} \HHH^2(\Gam)\}.
\]
\end{theorem}

\begin{corollary}[\cite{KKMMM1}]\label{cor=dim_est_comp_surj}
  Assume that $\Ng$ is contained in the commutator subgroup $[\Gg, \Gg]$ of $\Gg$, and $\Gam$ is boundedly $3$-acyclic. Then the following inequality holds:
  \begin{align*}
    \dim_{\RR} \WGN \leq \dim_{\RR} \HHH^2(\Gam) - \dim_{\RR} \HHH^1(\Ng)^{\Gg}.
  \end{align*}
  Moreover, if the comparison map $c^2_{\Gg}\colon \HHH_b^2(\Gg) \to \HHH^2(\Gg)$ is surjective, then
  \begin{align*}
    \dim_{\RR} \WGN = \dim_{\RR} \HHH^2(\Gam) - \dim_{\RR} \HHH^1(\Ng)^{\Gg}.
  \end{align*}
\end{corollary}

If $\HHH^2(\Gg) = 0$ (for example, free groups, the fundamental groups of non-orientable connected surfaces, and braid groups) and $\Ng$ contains a commutator subgroup of $\Gg$, then we have $\WGN = 0$ by Theorem \ref{theorem=WGN_isom}.

The following theorem  due to Mineyev (see also \cite{Gr}) supplies several examples for the latter statement of Corollary \ref{cor=dim_est_comp_surj}.
\begin{theorem}[\cite{Mineyev}]\label{thm=Mineyev}
  If $\Gg$ is non-elementary Gromov-hyperbolic, then the comparison map $c_{\Gg}^n \colon \HHH_b^n(\Gg) \to \HHH^n(\Gg)$ is surjective for $n \geq 2$.
\end{theorem}

 We exhibit the following dimension computation of $\WGN$: this is the first example of $(\Gg,\Ng)$  such that  $\WGN\ne 0$  and that  $\dim_{\RR} \WGN$ is  determined.  Other examples of estimates of $\dim_{\RR} \WGN$ can be also found in \cite{KKMMM1}.

\begin{corollary}[\cite{KKMMM1}]
  Let $\Gg$ be the fundamental group of an \textup{(}orientable\textup{)} connected closed surface of genus $g \geq 2$ and $\Ng$ the commutator subgroup of $\Gg$.
  Then $\dim_{\RR} \WGN = 1$.
\end{corollary}

\begin{proof}
  Since $\Gg$ is non-elementary Gromov-hyperbolic, the comparison map $c_{\Gg}^2 \colon \HHH_b^2(\Gg) \to \HHH^2(\Gg)$ is surjective by Theorem \ref{thm=Mineyev}.
  Moreover, the map $p^* \colon \HHH^2(\Gam) \to \HHH^2(\Gg)$ is also surjective.
  Indeed, $\dim_{\RR} \HHH^1(N)^G = g(2g-1)-1$ (see for instance \cite[Proposition 4.7]{KKMMM1}), $\dim_{\RR} \HHH^2(\Gam) = \dim_{\RR} \HHH^2(\ZZ^{2g}) = g(2g-1)$ and $\dim_{\RR}\HHH^2(\Gg) = 1$.
  Further, the map $p^* \colon \HHH^1(\Gam) \to \HHH^1(\Gg)$ is isomorphic since $p \colon \Gg \to \Gam$ is the abelianization.
  Hence the five-term exact sequence \eqref{align=five-term-grp-coh} implies the surjectivity of $p^* \colon \HHH^2(\Gam) \to \HHH^2(\Gg)$.
  By Theorem \ref{theorem=WGN_isom}, we obtain $\WGN \cong \Im \, p^* \cap \Im \, c_{\Gg}^2 = \HHH^2(\Gg) \cong \RR$.
\end{proof}


For the reader's convenience, we collect known facts on the study of bounded acyclicity by various researchers.

\begin{theorem}[known results for boundedly $n$-acyclic groups]\label{thm=bdd_acyc}
The following hold.
\begin{enumerate}[$(1)$]
\item \textup{(\cite{Gr})} Every amenable group is boundedly acyclic.
\item \textup{(}see \textup{\cite{MR21})} Let $n\in \NN$. Let $1\to N \to G \to \Gamma \to 1$ be a short exact sequence of groups. Assume that $N$ is boundedly $n$-acyclic. Then $G$ is boundedly $n$-acyclic if and only if $\Gamma$ is.
\item \textup{(\cite{MatsumotoMorita})}  Let $m\in \NN$. Then, the group $\Homeo_{c}(\RR^m)$ of homeomorphisms on $\RR^m$ with compact support is boundedly acyclic.
\item \textup{(}combination of \textup{\cite{Mon04}} and \textup{\cite{MS04})}
For $m \in \NN_{\geq 3}$, every lattice in $\SL(m,\RR)$ is boundedly $3$-acyclic.
\item \textup{(\cite{BucherMonod})} Burger--Mozes groups  are boundedly $3$-acyclic.
\item \textup{(\cite{Monod2021})} Richard Thompson's group $F$ is boundedly acyclic.
\item \textup{(\cite{Monod2021})} Let $L$ be an arbitrary group. Let $\Gamma$ be an infinite amenable group. Then the wreath product $L\wr \Gamma =\left( \bigoplus\limits_{\Gamma}L\right)\rtimes \Gamma$ is boundedly acyclic.
\item \textup{(\cite{MN21})} For $m\in \NN_{\geq 2}$, the identity component $\Homeo_0(S^m)$ of the group of homeomorphisms of $S^m$ is boundedly $3$-acyclic. The group $\Homeo_0(S^3)$ is boundedly $4$-acyclic.
\item \textup{(\cite{CFLM})} For $m\in \NN_{\geq 2}$, the group ${\rm Diff}_{c,{\rm vol}}(\RR^m)$ of volume-preserving diffeomorphisms on $\RR^m$ with compact support is boundedly acyclic.
\item \textup{(\cite{CFLM})} The stable mapping class group $\Gamma_{\infty}=\bigcup\limits_{g=1}^\infty \Gamma_g^1$ and the stable braid group $B_{\infty}=\bigcup\limits_{n=1}^\infty B_n$ \textup{(}see \textup{\cite[1.1.1 and 1.1.2]{CFLM}} for these definitions, respectively\textup{)} are boundedly acyclic.
\end{enumerate}
\end{theorem}

The following corollary to Theorem~\ref{theorem=WGN_isom} supplies a wide class of pairs $(\Gg,\Ng)$ for which $\WGN$ is finite dimensional, in contrast to Proposition~\ref{prop=AH}.

\begin{corollary}\label{cor=nilp}
Let $\Gg$ be a group and $\Ng$ its normal subgroup. Set $\Gam=\Gg/\Ng$. If $\Gam$ is a finitely presented amenable group, then $\WGN$ is finite dimensional. In particular, if $\Gg$ is finitely generated and $\Gam$ is nilpotent, then $\WGN$ is finite dimensional.
\end{corollary}

\begin{proof}
In what follows, we show the first assertion. By Theorem~\ref{thm=bdd_acyc}~(1), $\Gam$ is boundedly $3$-acyclic. Hence Theorem~\ref{theorem=WGN_isom} applies to the present case, and we have
\[
\dim_{\RR} \WGN\leq \dim_{\RR} \HHH^2(\Gam);
\]
the right-hand side of the inequality above is finite because $\Gam$ is finitely presented, as desired.

Finally, we see the latter assertion. Observe that  $\Gam$ is a finitely generated nilpotent group, and  hence that $\Gam$ is finitely presented (see for instance \cite[2.3 and 2.4]{Mann}). Now, we can apply the first assertion.
\end{proof}

 From here to the end of this subsection, let $\Gg_{\genus}$ be the fundamental group of  an (orientable) connected closed  surface of genus $g \geq 2$ and $\Ng_{\genus}$ the commutator subgroup of $\Gg_{\genus}$.
As explained above, we have $\dim_{\RR} \WW(\Gg_{\genus}, \Ng_{\genus}) = 1$.
By using circle actions of $\Gg_{\genus}$, we can obtain an explicit invariant homogeneous quasimorphism representing a non-zero element of $\WW(\Gg_{\genus},\Ng_{\genus})$ as follows (see \cite{MMM} for details). 

Let $\Homeo_+(S^1)$ be the group of orientation preserving homeomorphisms of the circle $S^1$.
Let $T \colon \RR \to \RR$ be the translation by one and set
\begin{align*}
  \tHomeo_+(S^1) = \{ h \in \Homeo(\RR) \mid hT = Th \}.
\end{align*}
Then the \emph{Poincar\'{e} translation number} $\trot \colon \tHomeo_+(S^1) \to \RR$ is defined by $\trot(h) = \lim\limits_{n \to \infty} \dfrac{h^n(0)}{n}$.
It is known that $\trot$ is a homogeneous quasimorphism on $\tHomeo_+(S^1)$ with $\DD(\trot)=1$.

Let $\rho \colon \Gg_{\genus} \to \Homeo_+(S^1)$ be a homomorphism.
Since $\Ng_{\genus}$ is a free group (of infinite rank), there exists a homomorphism $\widetilde{\rho} \colon \Ng_{\genus} \to \tHomeo_+(S^1)$ such that the diagram
\begin{align*}
\xymatrix{
\Ng_{\genus} \ar[r]^-{\widetilde{\rho}} \ar[d] & \tHomeo_+(S^1) \ar[d]^-{\pi} \\
\Gg_{\genus} \ar[r]^-{\rho} & \Homeo_+(S^1)
}
\end{align*}
commutes.

Let $\chi$ be Matsumoto's canonical cocycle on $\Homeo_+(S^1)$, which is defined by
\begin{align*}
  \chi(\pi(\gl_1), \pi(\gl_2)) = \trot(\gl_1\gl_2) - \trot (\gl_1) - \trot(\gl_2)
\end{align*}
for $\gl_1,\gl_2 \in \tHomeo_+(S^1)$ (\cite{Matsumoto}).
Since $p^* \colon \HHH^2(\Gam_{\genus}) \to \HHH^2(\Gg_{\genus})$ is surjective, there exists a normalized $2$-cocycle $A$ on $\Gam_{\genus}$ such that $p^*[A] = \rho^*[\chi]$.
Hence there exists a $1$-cochain $u$ on $\Gg_{\genus}$ such that $\delta u = p^* A - \rho^* \chi$.
\begin{theorem}[\cite{MMM}]
  The restriction $u|_{\Ng_{\genus}}$ is contained in $\QQQ(\Ng_{\genus})^{\Gg_{\genus}}$.
  Moreover, if $\rho^*[\chi]$ is non-zero, then $[u|_{\Ng_{\genus}}]$ is non-zero in $\WW(\Gg_{\genus},\Ng_{\genus})$.
\end{theorem}
We note that the invariant homogeneous quasimorphism $u|_{\Ng_{\genus}}$ itself \emph{does} depend on the choices of $A$ and $u$.
However, the class of $u|_{\Ng_{\genus}}$ in $\QQQ(\Ng_{\genus})^{\Gg_{\genus}}/\HHH^1(\Ng_{\genus})^{\Gg_{\genus}}$ only depends on $\rho$.
According to this, the value of $u|_{\Ng_{\genus}}$ on $[\Gg_{\genus}, \Ng_{\genus}]$ is determined by $\rho$, and in fact, the following formula holds.
\begin{theorem}[\cite{MMM}]
  For $\gl_1, \cdots, \gl_n \in \Gg_{\genus}$ and $\xl_1, \cdots, \xl_n \in \Ng_{\genus}$, the equality
  \begin{align*}
    (u|_{\Ng_{\genus}})\big([\gl_1, \xl_1]\cdots [\gl_n,\xl_n]\big) = \trot\big([\widetilde{\rho(\gl_1)},\widetilde{\rho(\xl_1)}]\cdots [\widetilde{\rho(\gl_n)},\widetilde{\rho(\xl_n)}]\big)
  \end{align*}
  holds, where $\widetilde{\rho(\gl_i)}, \widetilde{\rho(\xl_i)} \in \tHomeo_+(S^1)$ are lifts of $\rho(\gl_i), \rho(\xl_i) \in \Homeo_+(S^1)$.
\end{theorem}

 Here, despite the fact that the choices $\widetilde{\rho(\gl_i)}$ and $\widetilde{\rho(\xl_i)}$ are not unique, the value $\trot\big([\widetilde{\rho(\gl_1)},\widetilde{\rho(\xl_1)}]\cdots [\widetilde{\rho(\gl_n)},\widetilde{\rho(\xl_n)}]\big)$ does not depend on these choices. Indeed, the choices $\widetilde{\rho(\gl_i)}$ and $\widetilde{\rho(\xl_i)}$ may differ only by elements in the center of $\tHomeo_+(S^1)$, respectively, so that the commutator $[\widetilde{\rho(\gl_i)}, \widetilde{\rho(\xl_i)}]$ is independent of their choices.


\subsection{Naturality}
In this subsection, we prove the naturalities of diagram \eqref{align=diagram_five_term_coh_rel} and the isomorphism in Theorem \ref{theorem=WGN_isom}.
In \cite{KKMMM1}, the authors of the present article proved a certain result \cite[Lemma 9.2]{KKMMM1} actually by using  a special case of  these naturalities.
For completeness, here we explicitly state these naturalities in full generality as Proposition~\ref{proposition=naturality_five-term}, and present the proof.

\begin{proposition}\label{proposition=naturality_five-term}
  Let
  \begin{align*}
  \xymatrix{
  1 \ar[r] & \Ng_1 \ar[r] \ar[d]^-{F_{\Ng}} & \Gg_1 \ar[r] \ar[d]^-{F_{\Gg}} & \Gam_1 \ar[r] \ar[d]^-{F_{\Gam}} & 1 \\
  1 \ar[r] & \Ng_2 \ar[r] & \Gg_2 \ar[r] & \Gam_2 \ar[r] & 1
  }
  \end{align*}
  be a commutative diagram of groups whose rows are exact.
  Then the diagram
  \begin{align*}
  \xymatrix{
  \HHH^1(\Gam_2) \ar[r] \ar[d] \ar[rdd] & \HHH^1(\Gg_2) \ar[r] \ar[d] \ar[rdd] & \HHH^1(\Ng_2)^{\Gg_2} \ar[r] \ar[d] \ar[rdd] & \HHH^2(\Gam_2) \ar[r] \ar[d] \ar[rdd] & \HHH^2(\Gg_2) \ar[d] \ar[rdd] & \\
  \QQQ(\Gam_2) \ar[r]|\hole \ar[rdd] & \QQQ(\Gg_2) \ar[r]|\hole \ar[rdd]|\hole & \QQQ(\Ng_2)^{\Gg_2} \ar[r]|\hole \ar[rdd]|\hole & \HHH_{/b}^2(\Gam_2) \ar[r]|\hole \ar[rdd]|\hole & \HHH_{/b}^2(\Gg_2) \ar[rdd]|\hole & \\
  & \HHH^1(\Gam_1) \ar[r] \ar[d] & \HHH^1(\Gg_1) \ar[r] \ar[d] & \HHH^1(\Ng_1)^{\Gg_1} \ar[r] \ar[d] & \HHH^2(\Gam_1) \ar[r] \ar[d] & \HHH^2(\Gg_1) \ar[d] \\
  & \QQQ(\Gam_1) \ar[r] & \QQQ(\Gg_1) \ar[r] & \QQQ(\Ng_1)^{\Gg_1} \ar[r] & \HHH_{/b}^2(\Gam_1) \ar[r] & \HHH_{/b}^2(\Gg_1)
  }
  \end{align*}
  induced by the maps $F_{\Ng}, F_{\Gg}$ and $F_{\Gam}$ is commutative.
\end{proposition}

We will only prove the commutativity of the following diagram:
\begin{align*}
\xymatrix{
\QQQ(\Ng_2)^{\Gg_2} \ar[r]^-{\tau_{/b}} \ar[d]^-{F_{\Ng}^*} & \HHH_{/b}^2(\Gam_2) \ar[d]^-{F_{\Gam}^*} \\
\QQQ(\Ng_1)^{\Gg_1} \ar[r]^-{\tau_{/b}} & \HHH_{/b}^2(\Gam_1).
}
\end{align*}
Indeed, for the other parts, the commutativity of the top face is a well-known property of the five-term exact sequence of group cohomology; those of the side faces come from that of \eqref{align=diagram_five_term_coh_rel}; and that of the bottom face other than the above diagram is straightforward.

For a short exact sequence $1 \to \Ng \xrightarrow{i} \Gg \xrightarrow{p} \Gam \to 1$ of groups, the map $\tau_{/b} \colon \QQQ(\Ng)^{\Gg} \to \HHH_{/b}^2(\Gam)$ is defined as follows (see  \cite[Section 6]{KKMMM1} for details).
Take a set-theoretic section $s \colon \Gam \to \Gg$ with $s(1_{\Gam}) = 1_{\Gg}$.
For $\mu \in \QQQ(\Ng)^{\Gg}$, we define $\mu_{s} \colon \Gg \to \RR$ by
\begin{align*}
  \mu_{s}(\gl) = \mu(\gl \cdot (sp(\gl))^{-1}).
\end{align*}
Recall that $j \colon \CCC^2(\Gam) \to \CCC_{/b}^2(\Gam)$ is the quotient map.
Then, $j(s^*\delta \mu_s) \in \CCC_{/b}^2(\Gam)$ is a $2$-cocycle. 
We note that the cocycle $j(s^*\delta \mu_s)$ does not depend on the choice of $s$.
The map $\tau_{/b} \colon \QQQ(\Ng)^{\Gg} \to \HHH_{/b}^2(\Gam)$ is given by
\begin{align*}
  \tau_{/b}(\mu) = [j(s^*\delta \mu_s)].
\end{align*}

\begin{proof}[Proof of Proposition \textup{\ref{proposition=naturality_five-term}}]
  We only prove the commutativity of
  \begin{align*}
  \xymatrix{
  \QQQ(\Ng_2)^{\Gg_2} \ar[r]^-{\tau_{/b}} \ar[d]^-{F_{\Ng}^*} & \HHH_{/b}^2(\Gam_2) \ar[d]^-{F_{\Gam}^*} \\
  \QQQ(\Ng_1)^{\Gg_1} \ar[r]^-{\tau_{/b}} & \HHH_{/b}^2(\Gam_1).
  }
  \end{align*}
  The other parts are straightforward.

  For $i \in \{ 1, 2 \}$, let $s_i \colon \Gam_i \to \Gg_i$ be a set-theoretical section satisfying $s_i(1_{\Gam_i}) = 1_{\Gg_i}$.
  Let $\mu$ be an element of $\QQQ(\Ng_2)^{\Gg_2}$.
  By the definition of $\tau_{/b}$, representatives of $\tau_{/b}(\mu)$ and $\tau_{/b}(F_{\Ng}^*\mu)$ are given by $j(s_2^*\delta \mu_{s_2})$ and $j(s_1^*\delta (F_{\Ng}^* \mu)_{s_1})$, respectively.
  By straightforward calculations, we have
  \begin{align*}
    F_{\Gam}^*(s_2^*\delta \mu_{s_2})(\gamma_1, \gamma_2) = -\mu(s_2(F_{\Gam}(\gamma_1))s_2(F_{\Gam}(\gamma_2))s_2(F_{\Gam}(\gamma_1\gamma_2))^{-1})
  \end{align*}
  and
  \begin{align*}
    s_1^*\delta (F_{\Ng}^* \mu)_{s_1}(\gamma_1, \gamma_2) = -\mu(F_{\Gg}(s_1(\gamma_1))F_{\Gg}(s_1(\gamma_2))F_{\Gg}(s_1(\gamma_1\gamma_2))^{-1}).
  \end{align*}
  We define a $1$-cochain $a \in \CCC^1(\Gam_1)$ by
  \begin{align*}
    a(\gamma) = \mu(F_{\Gg}(s_1(\gamma))^{-1}s_2(F_{\Gam}(\gamma))).
  \end{align*}
  Then, since $\mu$ is a $\Gg_2$-invariant homogeneous quasimorphism, we have
  \begin{align*}
    F_{\Gam}^*(j(s_2^*\delta \mu_{s_2})) - j(s_1^*\delta (F_{\Ng}^* \mu)_{s_1}) = \delta (j a) \in \delta\CCC_{/b}^1(\Gam_1).
  \end{align*}
  Hence $F_{\Gam}^* \tau_{/b}(\mu) = \tau_{/b}(F_{\Ng}^*\mu)$.
\end{proof}

Proposition \ref{proposition=naturality_five-term} and the diagram chasing arguments show the following.

\begin{corollary}
  Let
  \begin{align*}
  \xymatrix{
  1 \ar[r] & \Ng_1 \ar[r] \ar[d]^-{F_{\Ng}} & \Gg_1 \ar[r] \ar[d]^-{F_{\Gg}} & \Gam_1 \ar[r] \ar[d]^-{F_{\Gam}} & 1 \\
  1 \ar[r] & \Ng_2 \ar[r] & \Gg_2 \ar[r] & \Gam_2 \ar[r] & 1
  }
  \end{align*}
  be a commutative diagram of groups whose rows are exact.
  Assume that $\Gam_1$ and $\Gam_2$ are boundedly $3$-acyclic.
  Then the induced diagram
  \begin{align*}
  \xymatrix{
  \WW(\Gg_2, \Ng_2) \ar[r]^-{\cong} \ar[d] & \Im \, (p^*\colon \HHH^2(\Gam_2) \to \HHH^2(\Gg_2)) \cap \Im \, (c_{\Gg_2} \colon \HHH_{b}^2(\Gg_2) \to \HHH^2(\Gg_2)) \ar[d] \\
  \WW(\Gg_1, \Ng_1) \ar[r]^-{\cong} & \Im \, (p^*\colon \HHH^2(\Gam_1) \to \HHH^2(\Gg_1)) \cap \Im \, (c_{\Gg_1} \colon \HHH_{b}^2(\Gg_1) \to \HHH^2(\Gg_1))
  }
  \end{align*}
  commutes, where the horizontal isomorphisms are the ones in Theorem \textup{\ref{theorem=WGN_isom}}.
\end{corollary}

\section{Exact sequences related to $\WGN$}
\label{sec=exactW}

Recall that Theorem \ref{theorem=WGN_isom} connects the space $\WGN$ and $\Im p^* \cap \Im c_{\Gg}^2$.
In this section, we show some exact sequences (Theorem~\ref{theorem=ex_seq_rel_to_WGN}) which are related to $\WGN$ and to the space $\Ker i^* \cap \Im c_{\Gg}^2$; the latter space contains $\Im p^* \cap \Im c_{\Gg}^2$ as a subspace.
The contents of this section existed in an early draft of \cite{KKMMM1} as an appendix.
Recall from the beginning of Section~\ref{sec=W} that we continue to adopt the following convention: if we write cochain groups and group cohomology (ordinary and bounded)  in the setting of Subsection~\ref{subsec=cohomology} for $A=\RR$ (equipped with the standard absolute value $|\cdot|$), then we omit to indicate the coefficient group $A$.

Let $\EH_b^2$ denote the kernel of the comparison map $\HHH_b^2 \to \HHH^2$ of degree $2$. For a group $\Gg$, we recall from Lemma~\ref{lem=Q/H} that $\delta\colon \CCC^1(\Gg)\to \CCC^2(\Gg)$ induces an isomorphism between   $\QG/\HG$ and $\EH_b^2(\Gg)$.


\begin{theorem} \label{theorem=ex_seq_rel_to_WGN}
Let $1 \to \Ng \xrightarrow{i} \Gg \xrightarrow{p} \Gam \to 1$ be a short exact sequence of groups.
Then the following hold true:
\begin{itemize}
\item[\textup{(1)}] There exists the following exact sequence
\[0 \to \WGN \to \EH^2_b(\Ng)^{\Gg} / i^* \EH^2_b(\Gg) \xrightarrow{\alpha} \HHH^1(\Gg ; \HHH^1(\Ng)).\]

\item[\textup{(2)}] There exists the following exact sequence
\[\HHH^2_b(\Gam) \to \Ker(i^*) \cap \Im(c_{\Gg}^2) \xrightarrow{\beta} \EH^2_b(\Ng)^{\Gg} / i^* \EH^2_b(\Gg) \to \HHH^3_b(\Gam).\]
Here $i^*$ is the map $\HHH^2(\Gg) \to \HHH^2(\Ng)$ induced by the inclusion $\Ng \hookrightarrow \Gg$, and $c_{\Gg}^2 \colon \HHH^2_b(\Gg) \to \HHH^2(\Gg)$ is the comparison map of degree $2$.
\end{itemize}
\end{theorem}

By Theorem \ref{theorem=ex_seq_rel_to_WGN}~(1) and (2), we obtain the following:
\begin{corollary}\label{corollary=ex_seq_rel_to_WGN}
Let $1 \to \Ng \xrightarrow{i} \Gg \xrightarrow{p} \Gam \to 1$ be a short exact sequence of groups.
If $\Gam$ is boundedly $3$-acyclic, then there exists the following exact sequence
\[0 \to \WGN \to \Ker(i^*) \cap \Im(c_{\Gg}^2) \to \HHH^1(\Gg ; \HHH^1(\Ng)).\]
\end{corollary}

To prove Theorem \ref{theorem=ex_seq_rel_to_WGN}~(2), we use the following.

\begin{theorem}[{\cite[Example~12.4.3]{Monod}}] \label{theorem=five_term_bdd}
Let $1 \to \Ng \xrightarrow{i} \Gg \xrightarrow{p} \Gam \to 1$ be a short exact sequence of groups.
Then there exists an exact sequence
\[0 \to \HHH^2_b(\Gam) \xrightarrow{p^*} \HHH^2_b(\Gg) \xrightarrow{i^*} \HHH^2_b(\Ng)^{\Gg} \to \HHH^3_b(\Gam) \xrightarrow{p^*} \HHH_b^3(\Gg).\]
\end{theorem}

\begin{proof}[Proof of Theorem \textup{\ref{theorem=ex_seq_rel_to_WGN}}]
First we prove item~(1).
Recall that $\EH^2_b(\Gg)$ is the kernel of the comparison map $c_{\Gg}^2 \colon \HHH^2_b(\Gg) \to \HHH^2(\Gg)$ of degree $2$. By Lemma \ref{lem=Q/H}, $\EH^2_b(\Gg)$ coincides with the image of $\delta \colon \QQQ(\Gg) \to \HHH^2_b(\Gg)$. Hence we have a short exact sequence
\begin{align*} 
0 \to \HHH^1(\Gg) \to \QQQ(\Gg) \to \EH^2_b(\Gg) \to 0.
\end{align*}
For a $\Gg$-module $V$, we write $V^{\Gg}$ the subspace consisting of the elements of $V$ which are fixed by every element of $G$. Since this functor $(-)^G$ is a left exact and its right derived functor is $V \mapsto \HHH^\bullet(\Gg ; V)$, we have an exact sequence
\begin{align*} 
0 \to \HHH^1(\Ng)^{\Gg} \to \QQQ(\Ng)^{\Gg} \to \EH^2_b(\Ng)^{\Gg} \to \HHH^1(\Gg ; \HHH^1(\Ng)).
\end{align*}
Thus we have the following commutative diagram
\begin{align} \label{align=3.3}
\xymatrix{
0 \ar[r] & \HHH^1(\Gg) \ar[r] \ar[d]_{i^*} & \QQQ(\Gg) \ar[r] \ar[d]^{i^*} & \EH^2_b(\Gg) \ar[r] \ar[d]^{i^*} & 0 \ar[d] \\
0 \ar[r] & \HHH^1(\Ng)^{\Gg} \ar[r] & \QQQ(\Ng)^{\Gg} \ar[r] & \EH^2_b(\Ng)^{\Gg} \ar[r] & \HHH^1(\Gg ; \HHH^1(\Ng)).
}
\end{align}
Taking the cokernels of the vertical maps, we have a sequence
\begin{align*} 
\HHH^1(\Ng)^{\Gg} /i^* \HHH^1(\Gg) \to \QQQ(\Ng)^{\Gg} / i^* \QQQ(\Gg) \to \EH^2_b(\Ng)^{\Gg} / i^* \EH^2_b(\Gg) \to \HHH^1(\Gg ; \HHH^1(\Ng)).
\end{align*}
The exactness of the first three terms of this sequence follows from the snake lemma. The exactness of the last three terms
can be checked by the diagram chasing. Since the cokernel of $\HHH^1(N)^G / i^* \HHH^1(G) \to \QQQ(N)^G / i^* \QQQ(G)$ is $\WGN$, we have an exact sequence
\begin{align*} 
0 \to \WGN \to \EH^2_b(\Ng)^{\Gg} / i^* \EH^2_b(\Gg) \to \HHH^1(\Gg ; \HHH^1(\Ng)).
\end{align*}
This completes the proof of Theorem \ref{theorem=ex_seq_rel_to_WGN}~(1).

Next we prove item~(2).
By Lemma~\ref{lem=Q/H}, we have the following commutative diagram
\begin{align} \label{align=3.6}
\xymatrix{
0 \ar[r] & \EH^2_b(\Gg) \ar[r] \ar[d] & \HHH^2_b(\Gg) \ar[r] \ar[d]^{i^*} & \Im(c_{\Gg}^2) \ar[r] \ar[d] & 0 \\
0 \ar[r] & \EH^2_b(\Ng)^{\Gg} \ar[r] & \HHH^2_b(\Ng)^{\Gg} \ar[r] & \HHH^2(\Ng)^{\Gg} &
\;\;\;\;\; ,}
\end{align}
where each row is exact. The exactness of the second row follows from Lemma \ref{lem=Q/H} and the left exactness of the functor $(-)^{\Gg}$. Let $K$ and $W$ denote the kernel and cokernel of the map $i^* \colon \HHH^2_b(\Gg) \to \HHH^2_b(\Ng)^{\Gg}$.
Note that the kernel of $\Im(c_{\Gg}^2) \to \HHH^2(\Ng)^{\Gg}$ is $\Im(c_{\Gg}^2) \cap \Ker(i^* \colon \HHH^2(\Gg) \to \HHH^2(\Ng)^{\Gg})$.
Applying the snake lemma, we have the following exact sequence
\begin{align*} 
K \to \Ker(i^*) \cap \Im(c_{\Gg}^2) \to \EH^2_b(\Ng)^{\Gg} / i^* \EH^2_b(\Gg) \to W.
\end{align*}

By Theorem \ref{theorem=five_term_bdd}, $K$ is isomorphic to $\HHH^2_b(\Gam)$, and there exists a monomorphism from $W$ to $\HHH^3_b(\Gg)$.
Hence we have an exact sequence
\begin{align*} 
\HHH^2(\Gam) \to \Ker(i^*) \cap \Im(c_{\Gg}^2) \to \EH^2_b(\Ng) / i^* \EH^2_b(\Gg) \to \HHH_b^3(\Gg).
\end{align*}
Here the last map $\EH^2_b(\Ng)^{\Gg} / i^* \EH^2_b(\Gg) \to \HHH^3_b(\Gg)$ is the composite of the map $\EH^2_b(\Ng)^{\Gg} / i^* \EH^2_b(\Gg) \to W$ and the monomorphism $W \to \HHH^3_b(\Gg)$.
This completes the proof of Theorem \ref{theorem=ex_seq_rel_to_WGN}~(2).
\end{proof}



Finally, we discuss a relation between the maps $\alpha$ and $\beta$ in the exact sequences in Theorem \ref{theorem=ex_seq_rel_to_WGN} and a map in the seven-term exact sequence.
Let us recall the seven-term exact sequence with trivial real coefficients.
\begin{theorem}[seven-term exact sequence] \label{theorem=seven-term}
  Let $1 \to \Ng \xrightarrow{i} \Gg \xrightarrow{p} \Gam \to 1$ be a short exact sequence of groups. Then there exists the following exact sequence:
\begin{align*}
  0 \to &\HHH^1(\Gam) \xrightarrow{p^*} \HHH^1(\Gg) \xrightarrow{i^*} \HHH^1(\Ng)^{\Gg} \to \HHH^2(\Gam) \\
  &\to {\rm Ker}(i^* \colon \HHH^2(\Gg) \to \HHH^2(\Ng)) \xrightarrow{\rho} \HHH^1(\Gam ; \HHH^1(\Ng)) \to \HHH^3(\Gam).
\end{align*}
\end{theorem}

A cocycle description of the map $\rho$ is given as follows.

\begin{theorem}[{\cite[Section 10.3]{DHW}}] \label{theorem=DHW}
Let $\Gg$ be a group and $\Ng$ its normal subgroup.
Let $c \in \Ker (i^* \colon \HHH^2(\Gg) \to \HHH^2(\Ng))$, and let $f$ be a 2-cocycle on $\Gg$ satisfying $f|_{\Ng \times \Ng} = 0$ and $[f] = c$. Then
\[\big( \rho (c) (p(\gl))\big) (\xl) =  f(\gl, g^{-1} \xl\gl) - f(\xl,\gl),\]
where $\gl \in \Gg$ and $\xl \in \Ng$.
\end{theorem}

The maps $\alpha$, $\beta$ and $\rho$ are related as follows.

\begin{theorem}\label{theorem=alpha_beta_rho}
  Let $1 \to \Ng \xrightarrow{i} \Gg \xrightarrow{p} \Gam \to 1$ be a short exact sequence of groups.
  Then the following diagram commutes:
  \[\xymatrix{
  \Ker(i^*) \cap \Im(c_{\Gg}^2) \ar[r]^-{j} \ar[d]_{\beta} & \Ker(i^*) \ar[r]^-{\rho} & \HHH^1(\Gam ; \HHH^1(\Ng)) \ar[d] \\
  \EH^2_b(\Ng)^{\Gg} / i^* \EH^2_b(\Gg) \ar[rr]^-{\alpha} & {} & \HHH^1(\Gg ; \HHH^1(\Ng)).
  }\]
  Here $j$ is the inclusion, and $\alpha$, $\beta$, and $\rho$ are the maps appearing in Theorem \ref{theorem=ex_seq_rel_to_WGN}, and the seven-term exact sequence, respectively.
\end{theorem}

\begin{proof}
  Recall that the map $\alpha \colon \EH^2_b(\Ng)^{\Gg} / i^* \EH^2_b(\Gg) \to \HHH^1(\Gg ; \HHH^1(\Ng))$ in Theorem \ref{theorem=ex_seq_rel_to_WGN}~(1) is induced by the last map $\siph$ of the exact sequence
  \[0 \to \HHH^1(\Ng)^{\Gg} \to \QQQ(\Ng)^{\Gg} \xrightarrow{\delta} \EH^2_b(\Ng)^{\Gg} \xrightarrow{\siph} \HHH^1(\Gg; \HHH^1(\Ng)).\]

  In this proof, for $\gl \in \Gg$ and a function $a \colon \Ng^{n} \to \RR$, we use the symbol ${}^{\gl}a$ to denote the left-conjugation of $a$ by $\gl$.

  We first describe the map $\siph$. Let $c \in \EH^2_b(\Ng)^{\Gg}$. Since $\delta \colon \QQQ(\Ng) \to \EH^2_b(\Ng)$ is surjective, there exists a homogeneous quasimorphism $\qm$ on $\Ng$ such that $c = [\delta \qm]$. Since $c$ is $\Gg$-invariant, we have ${}^{\gl} c = c$ for every $\gl \in \Gg$.
  Namely, for each $\gl \in \Gg$, there exists a bounded $1$-cochain $b_{\gl} \in \CCC^1_b(\Ng)$ such that
  \begin{align} \label{align=3.12}
  ^{\gl} (\delta \qm) = \delta \qm + \delta b_{\gl}.
  \end{align}
  Note that this $b_{\gl}$ is unique. Indeed, if $\delta b_{\gl} = \delta b_{\gl}'$, then $b_{\gl} - b'_{\gl}$ is a homomorphism $\Gg \to \RR$ which is bounded, and is $0$.

  Define a cochain $\siph_\qm \in \CCC^1(\Gg ; \HHH^1(\Ng))$ by
  \[\siph_\qm(\gl) = \qm - {}^{\gl} \qm - b_{\gl}.\]
  It follows from \eqref{align=3.12} that $\siph_\qm \in \HHH^1(\Ng)$.
  Now we show that this correspondence induces a map from $\EH^2_b(\Ng) / i^* \EH^2_b(\Gg)$ to $\HHH^1(\Gg ; \HHH^1(\Ng))$.
  Suppose that $c = [\delta \qm] = [\delta \qm']$ for $\qm, \qm' \in \QQQ(\Ng)$. Then $h = \qm - \qm' \in \HHH^1(\Ng)$.
  Therefore we have $\delta \qm = \delta \qm'$, and hence we have
  \[^{\gl}(\delta \qm') = \delta \qm' + \delta b_{\gl}.\]
  Hence we have
  \[(\siph_{\qm'} - \siph_\qm)(\gl) = (\qm' - {}^{\gl}\qm' + b_{\gl}) - (\qm - {}^{\gl} \qm + b_{\gl}) = {}^{\gl} h - h = \delta h (\gl).\]
  Therefore $\siph_{\qm'}$ and $\siph_\qm$ represent the same cohomology class of $\HHH^1(\Gg ; \HHH^1(\Ng))$.
  This correspondence $c \mapsto [\varphi_{\mu}]$ is the precise description of $\varphi \colon \EH^2_b(\Ng)^{\Gg} \to \HHH^1(\Gg ; \HHH^1(\Ng))$.

  Next, we see the precise description of the composite of
  \[\Ker (i^*) \cap \Im (c_{\Gg}^2) \xrightarrow{\beta} \EH^2_b(\Gg) / i^* \EH^2_b(\Gg) \xrightarrow{\alpha} \HHH^1(\Gg ; \HHH^1(\Ng)).\]
  Let $c \in \Ker (i^*) \cap \Im (c_{\Gg}^2)$.
  Since $c \in \Im (c_{\Gg}^2)$, there exists a bounded cocycle $f \colon \Gg \times \Gg \to \RR$ with $c = [f]$ in $\HHH^2(\Gg)$.
  Since $i^* c = 0$, there exists $\mufh \in \CCC^1(\Ng)$ such that $f|_{\Ng \times \Ng} = \delta \mufh$ in $\CCC^2(\Ng)$. Since $f$ is bounded, $\mufh$ is a quasimorphism on $\Ng$. Define $\qm$ to be the homogenization of $\mufh$.
  Then $b_{\Ng} = \qm - \mufh \colon \Ng \to \RR$ is a bounded 1-cochain on $\Ng$.
  Next define a function $b \colon \Gg \to \RR$ by
  \[b(\gl) = \begin{cases}
  b_{\Ng}(\gl), & \textrm{if}\ \gl \in \Ng, \\
  0, & {\rm otherwise.}
  \end{cases}\]
  Since $b \in \CCC^1_b(\Gg)$, $f + \delta b$ is a bounded cocycle representing $c$ in $\HHH^2(\Gg)$. Replacing $f + \delta b$  by  $f$, we can assume that $f|_{\Ng \times \Ng} = \delta \qm$.
  Then by the definition of the connecting homomorphism in snake lemma, we have $\beta(c) = [\delta \qm]$.

  Recall that there exists a unique bounded function $b_{\gl} \colon \Ng \to \RR$ such that
  \[\siph([\delta \qm])(\gl) = \qm - {}^{\gl} \qm + b_{\gl}.\]

  \vspace{2mm}
  \noindent
  {\bf Claim.} $b_{\gl}(\xl) = f(\gl, \gl^{-1} \xl \gl) - f(\xl, \gl)$.

  \vspace{2mm}
  Define a function $a_{\gl} \colon N \to \RR$ by $a_{\gl}(\xl) = f(\gl, \gl^{-1}\xl\gl) - f(\xl, \gl)$. Let $\xl$ and $\yl$ be elements of $\Ng$. Since $\delta f = 0$, we have
  \begin{align*}
  \delta a_{\gl}(\xl,\yl) & = \delta a_{\gl}(\xl,\yl) + \delta f (\gl, \gl^{-1} \xl\gl, \gl^{-1} \yl \gl) + \delta f(\xl,\yl,\gl) - \delta f(\xl,\gl,\gl^{-1} \yl\gl) \\
  & = f(\gl^{-1} \xl\gl, \gl^{-1} \yl \gl) - f(\xl,\yl) \\
  & = ({}^{\gl} \delta \qm - \delta \qm)(\xl,\yl).\\
  & = \delta b_{\gl}(\xl,\yl).
  \end{align*}
  By the uniqueness of $b_{\gl}$, we have $a_{\gl} = b_{\gl}$.
  This completes the proof of Claim.
  Hence we have $\siph_\qm(\gl) = \qm - {}^{\gl} \qm + a_{\gl}$, and thus we obtain a precise description of $\alpha \circ \beta$.

  Now we complete the proof of Theorem \ref{theorem=alpha_beta_rho}.
  For $c \in \Ker (i^*) \cap \Im (c_{\Gg}^2)$, there exist a bounded $2$-cocycle $f$ of $\Gg$ and $\qm \in \QQQ(\Ng)$ such that $c = [f]$ and $f |_{\Ng \times \Ng} = \delta \qm$.
  Define $\coch \colon \Gg \to \RR$ by
  \[\coch(\gl) = \begin{cases}
  \qm (\gl), & \textrm{if}\ \gl \in \Ng, \\
  0, & {\rm otherwise.}
  \end{cases}\]
  Then $f - \delta \coch$ is a (possibly unbounded) cocycle such that $(f - \delta \coch) |_{\Ng \times \Ng} = 0$. Hence Theorem \ref{theorem=DHW} implies that
  \begin{align*}
  ((p^* \rho (c))(\gl))(\xl) & = (\rho(c) (p(\gl))) (\xl) \\
  & = (f - \delta \coch)(\gl , \gl^{-1} \xl \gl) - (f - \delta \coch)(\xl, \gl) \\
  & = f(\gl, \gl^{-1} \xl \gl) - f(\xl,\gl) + \coch(\xl\gl) - \coch (\gl) - \coch(\gl^{-1} \xl \gl) \\
   & \quad + \coch(\gl) - \coch (\xl\gl) + \coch(\xl) \\
  & = \qm(\xl) - {}^{\gl} \qm(\xl) + b_{\gl}(\xl) \\
  & = (\qm - {}^{\gl} \qm + b_{\gl})(\xl) \\
  & = \siph_\qm(\gl)(\xl).
  \end{align*}
  Here the second equality follows from Theorem \ref{theorem=DHW} and the fourth equality follows from Claim. Hence we have
  \begin{align} \label{align=B}
  ((p^* \rho (c))(\gl))(\xl) = \varphi_\qm(\gl) (\xl),
  \end{align}
  and $\alpha \circ \beta (c) = p^* \circ \rho (c)$ follows from the description of $\alpha \circ \beta$ and \eqref{align=B}. This completes the proof.
\end{proof}

\section*{Acknowledgement}
The third-named author and the fifth-named author are grateful to Professor Sang-hyun Kim for his kind invitations to the KIAS in  November 2023 and May-June 2023, respectively. The plan of writing this article was initiated from this occasion of the fifth-named author's visit.  The authors thank the anonymous referee for  helpful  comments. 
The first-named author, the second-named author, the fourth-named author and the fifth-named author are partially supported by JSPS KAKENHI Grant Number JP21K13790, JP24K16921, JP23K12975, and JP21K03241, respectively.
The third-named author is partially supported by JSPS KAKENHI Grant Number JP23KJ1938 and JP23K12971.

\end{document}